\documentclass[12pt]{amsart}
  \usepackage{eucal}
\usepackage{mathrsfs}
\usepackage{amsmath}
\usepackage{graphicx}
\usepackage{amssymb}
\usepackage{amscd}
\usepackage{color}
\usepackage{xcolor}
\usepackage[percent]{overpic}
\usepackage{tikz}
\usepackage{tikz-cd}
\usetikzlibrary{arrows.meta}
\usetikzlibrary{math,angles,intersections,calc, decorations.pathreplacing, shapes.misc}
\usepackage[all,cmtip]{xy}
\usepackage{hyperref}
\usepackage{blkarray}

\tikzstyle{vecArrow} = [thick, decoration={markings,mark=at position
   1 with {\arrow[semithick]{open triangle 60}}},
   double distance=1.4pt, shorten >= 5.5pt,
   preaction = {decorate},
   postaction = {draw,line width=1.4pt, white,shorten >= 4.5pt}]
\tikzstyle{innerWhite} = [semithick, white,line width=1.4pt, shorten >= 4.5pt]

\newtheorem{thm}{Theorem}[section]
\newtheorem{prop}[thm]{Proposition}
\newtheorem{lemma}[thm]{Lemma}

\newtheorem{cor}[thm]{Corollary}

\theoremstyle{definition}

\newtheorem{problem}[thm]{Problem}

\newtheorem{definition}[thm]{Definition}

\newtheorem{remark}[thm]{Remark}

\newcommand{\cF}{\mathcal{F}}
\newcommand{\cG}{\mathcal{G}}
\newcommand{\cH}{\mathcal{H}}

\newcommand{\cL}{\mathcal{L}}

\newcommand{\cN}{\mathcal{N}}
\newcommand{\cS}{\mathcal{S}}
\newcommand{\cT}{\mathcal{T}}
\newcommand{\cU}{\mathcal{U}}
\newcommand{\cV}{\mathcal{V}}

\newcommand{\cMod}{\mathcal{M}\mathrm{od}}

\newcommand{\bk}{\mathbf{k}}

\newcommand{\bR}{\mathbf{R}}
\newcommand{\bZ}{\mathbf{Z}}
\newcommand{\bL}{\mathbf{L}}

\newcommand{\bC}{\mathbf{C}}

\newcommand{\bX}{\mathbf{X}}
\newcommand{\bY}{\mathbf{Y}}
\newcommand{\bW}{\mathbf{W}}
\newcommand{\bS}{\mathbf{S}}

\newcommand{\frj}{\mathfrak{j}}
\newcommand{\fri}{\mathfrak{i}}
\newcommand{\frf}{\mathfrak{f}}
\newcommand{\frg}{\mathfrak{g}}

\newcommand{\Supp}{\mathrm{Supp}}
\newcommand{\Comm}{\mathrm{Comm}}
\newcommand{\Cone}{\mathrm{Cone}}
\newcommand{\rank}{\mathrm{rank}}

\newcommand{\Hom}{\mathrm{Hom}}
\newcommand{\Maps}{\mathrm{Maps}}

\newcommand{\Aut}{\mathrm{Aut}}

\newcommand{\LagGr}{\mathrm{LagGr}}
\newcommand{\cLagGr}{\mathcal{L}\mathsf{agGr}}

\newcommand{\Loc}{\mathrm{Loc}}
\newcommand{\Mod}{\mathrm{Mod}}

\newcommand{\CAlg}{\mathrm{CAlg}}

\newcommand{\Fl}{\mathrm{F}\ell}

\newcommand{\Gr}{\mathrm{Gr}}

\newcommand{\pr}{\mathrm{pr}}
\newcommand{\proj}{\mathrm{proj}}

\newcommand{\Pic}{\mathrm{Pic}}

\newcommand{\Fin}{\mathrm{Fin}}

\newcommand{\Corr}{\mathrm{Corr}}

\newcommand{\Seq}{\mathrm{Seq}}

\newcommand{\Alg}{\mathrm{Alg}}
\newcommand{\Slch}{\mathrm{S}_{\mathrm{LCH}}}

\newcommand{\dagg}{\dagger}
\newcommand{\rightlax}{\mathrm{right}\text{-}\mathrm{lax}}

\newcommand{\twoop}{2\text{-}\mathrm{op}}

\newcommand{\bCorr}{\mathbf{Corr}}
\newcommand{\Gray}{\text{\circled{$\star$}}}

\newcommand{\fib}{\mathsf{fib}}
\newcommand{\openmap}{\mathsf{open}}
\newcommand{\all}{\mathsf{all}}

\newcommand{\QHam}{\mathsf{QHam}}
\newcommand{\propmap}{\mathsf{prop}}
\newcommand{\dgnl}{\mathrm{dgnl}}

\newcommand{\fC}{\mathsf{C}}
\newcommand{\fD}{\mathsf{D}}
\newcommand{\sfB}{\mathsf{B}}
\newcommand{\sfOmega}{\mathsf{\Omega}}

\newcommand{\lng}{\langle}
\newcommand{\rng}{\rangle}

\newcommand{\cY}{\mathcal{Y}}
\newcommand{\cC}{\mathcal{C}}

\newcommand{\cB}{\mathcal{B}}
\newcommand{\cP}{\mathcal{P}}
\newcommand{\cQ}{\mathcal{Q}}

\newcommand{\Sp}{\mathrm{Sp}}
\newcommand{\cLoc}{\cL\mathrm{oc}}

\newcommand{\Fun}{\mathrm{Fun}}

\newcommand{\bq}{\mathbf{q}}
\newcommand{\bp}{\mathbf{p}}

\newcommand{\Open}{\mathsf{Open}}

\newcommand{\Spc}{\mathcal{S}\mathrm{pc}}

\newcommand{\alg}{\mathsf{alg}}
\newcommand{\sfmod}{\mathsf{mod}}
\newcommand{\Shv}{\mathrm{Shv}}
\newcommand{\coShv}{\mathrm{coShv}}
\newcommand{\OneCat}{1\text{-}\mathrm{Cat}}
\newcommand{\TwoCat}{2\text{-}\mathrm{Cat}}
\newcommand{\Spectr}{\mathsf{Spectr}}
\newcommand{\ShvSp}{\mathrm{ShvSp}}

\newcommand{\PrstL}{\mathrm{Pr}_{\mathsf{st}}^\mathsf{L}}
\newcommand{\PrstR}{\mathrm{Pr}_{\mathsf{st}}^\mathsf{R}}
\newcommand{\bPrstL}{\mathbf{P}\mathrm{r}_{\mathsf{st}}^\mathsf{L}}
\newcommand{\bPrstR}{\mathbf{P}\mathrm{r}_{\mathsf{st}}^\mathsf{R}}

\newcommand{\bOneCat}{\textbf{1}\text{-}\mathbf{Cat}}

\newcommand{\leftlax}{\mathrm{left}\text{-}\mathrm{lax}}
\newcommand{\FT}{\mathsf{FT}}
\newcommand{\fj}{\mathfrak{j}}
\newcommand{\RKan}{\mathsf{RKan}}
\newcommand{\LKan}{\mathsf{LKan}}

\newcommand{\rel}{\mathsf{rel}}
\newcommand{\Sing}{\mathsf{Sing}}

\newcommand{\inert}{\mathsf{inert}}
\newcommand{\act}{\mathsf{active}}
\newcommand{\comp}{\mathsf{comp}}
\newcommand{\ord}{\mathsf{ord}}
\newcommand{\ConicOpen}{\mathsf{ConicOpen}}

\newcommand{\Path}{\mathcal{P}\mathrm{ath}}

\newcommand{\Bl}{\mathsf{Bl}}
\newcommand{\Gauss}{\mathsf{Gauss}}
\newcommand{\aff}{\mathsf{aff}}
\newcommand{\Graph}{\mathsf{Graph}}
\newcommand{\cShv}{\mathcal{S}\mathsf{hv}}
\newcommand{\bH}{\mathbb{H}}
\newcommand{\End}{\mathrm{End}}
\newcommand{\op}{\text{-op}}
\newcommand{\opincl}{\mathsf{opincl}}

\newcommand{\PcoShv}{\mathcal{P}\mathrm{coShv}}

\newcommand{\PrkL}{\mathrm{Pr}_\bk^{\mathsf{L}}}
\newcommand{\fG}{\mathfrak{G}}
\newcommand{\fF}{\mathfrak{F}}
\newcommand{\PShv}{\cP\Shv}

\newcommand{\bE}{\mathbb{E}}
\newcommand{\Seg}{\mathrm{Seg}}

\newcommand*\circled[1]{\tikz[baseline=(char.base)]{
            \node[shape=circle,draw,inner sep=0.05pt] (char) {#1};}}
            
   \tikzset{cross/.style={cross out, draw=black, fill=none, minimum size=2*(#1-\pgflinewidth), inner sep=0pt, outer sep=0pt}, cross/.default={3pt}}
            
\makeatletter
\newcommand{\xleftrightarrow}[2][]{\ext@arrow 3359\leftrightarrowfill@{#1}{#2}}
\newcommand{\xdashrightarrow}[2][]{\ext@arrow 0359\rightarrowfill@@{#1}{#2}}
\newcommand{\xdashleftarrow}[2][]{\ext@arrow 3095\leftarrowfill@@{#1}{#2}}
\newcommand{\xdashleftrightarrow}[2][]{\ext@arrow 3359\leftrightarrowfill@@{#1}{#2}}
\def\rightarrowfill@@{\arrowfill@@\relax\relbar\rightarrow}
\def\leftarrowfill@@{\arrowfill@@\leftarrow\relbar\relax}
\def\leftrightarrowfill@@{\arrowfill@@\leftarrow\relbar\rightarrow}
\def\arrowfill@@#1#2#3#4{%
  $\m@th\thickmuskip0mu\medmuskip\thickmuskip\thinmuskip\thickmuskip
   \relax#4#1
   \xleaders\hbox{$#4#2$}\hfill
   #3$%
}
\makeatother

\numberwithin{equation}{section}

\DeclareMathAlphabet{\mathpzc}{OT1}{pzc}{m}{it}

\begin{document}

\subjclass[2000]{}

\title{Microlocal sheaf categories and the $J$-homomorphism}
\author[X. Jin]{Xin Jin}
\address{Math Department,
Boston College,
Chestnut Hill, MA 02467.}
\email{xin.jin@bc.edu}

\maketitle

\begin{abstract}
Let $X$ be a smooth manifold and $\bk$ be a commutative (or at least $\bE_2$) ring spectrum. Given a smooth exact Lagrangian $L\hookrightarrow T^*X$, the microlocal sheaf theory (following Kashiwara--Schapira) naturally assigns a locally constant sheaf of categories on $L$ with fiber equivalent to the category of $\bk$-module spectra $\Mod(\bk)$. We show that the classifying map for the local system of categories factors through the stable Gauss map $L\rightarrow U/O$ and the delooping of the $J$-homomorphism $U/O\rightarrow B\Pic(\bS)$. As an application, combining with previous results of Guillermou \cite{Guillermou}, we recover a result of Abouzaid--Kragh \cite{AbKr} on the triviality of the composition $L\rightarrow U/O\rightarrow B\Pic(\bS)$, when $L$ is in addition compact. 
\end{abstract}

\tableofcontents

\section{Introduction}

Let $X$ be a smooth manifold and $\bk$ be a ring spectrum, either commutative or at least $\bE_2$. Given a smooth exact Lagrangian $L\hookrightarrow T^*X$, the microlocal sheaf theory (developed by Kashiwara--Schapira) naturally assigns a locally constant sheaf of categories on $L$ with fiber equivalent to the category of $\bk$-spectra $\Mod(\bk)$. It is a standard fact that such a locally constant sheaf of categories is determined (up to homotopy) by a classifying map
\begin{align}\label{eq: intro classify}
L\longrightarrow B\Aut_\bk(\Mod(\bk))\simeq B\Pic(\bk).
\end{align}
The goal of this paper is to give a solution to the following fundamental problem:
\begin{problem}\label{intro: problem}
Give a topological description of the classifying map for the sheaf of microlocal categories on $L$. 
\end{problem}

\begin{thm}[see Theorem \ref{thm: sec proof} for the precise statement]\label{thm: main intro}
The classifying map (\ref{eq: intro classify}) can be factored, up to homotopy, as 
\begin{align*}
L\overset{\gamma}{\longrightarrow} U/O\overset{BJ}{\longrightarrow}B\Pic(\bS)\longrightarrow B\Pic(\bk), 
\end{align*}
where $\gamma$ is the stable Gauss map, $BJ$ is the delooping of the $J$-homomorphism in stable homotopy theory, and the last map is the tautological one as the sphere spectrum $\bS$ is the initial object in the $\infty$-category of all ring spectra. 
\end{thm}
The above answer was predicted in \cite{JiTr}, and here we give a full proof. We remark that for $\bk$ an ordinary ring, the classifying map factors through $U/O\rightarrow B\Pic(\bZ)\simeq B\bZ\times B^2(\bZ/2\bZ)$, and the obstruction classes given by (\ref{eq: intro classify}) for $\bk=\bZ$ are exactly the Maslov class and the second relative Stiefel--Whitney class. The case for an ordinary ring $\bk$ has been obtained in \cite{Guillermou}.

As a consequence of Theorem \ref{thm: sec proof}, combining with previous results of \cite{Guillermou} and \cite{JiTr}, we get a sheaf theoretic proof of a result of Abouzaid--Kragh, which was based on Floer homotopy types. 
\begin{cor}[Proposition 2.2 \cite{AbKr}]\label{cor: AbKr}
For any compact smooth exact Lagrangian $L\hookrightarrow T^*X$, the classifying map (\ref{eq: intro classify}) for $\bk=\bS$ is homotopically trivial. 
\end{cor}

\subsection{Motivations and applications}
In this subsection, we give a brief overview of microlocal sheaf approaches to symplectic topology, leading to motivations of Problem \ref{intro: problem} and applications of Theorem \ref{thm: main intro} as an answer to it. 

Microlocal sheaf theory was developed by Kashiwara--Schapira \cite{Kashiwara} in the 90s, and it has been substantially applied to the study of symplectic topology, especially making connections to Floer cohomology,  after the work of Nadler--Zaslow \cite{NZ}, \cite{Nadler2} and Tamarkin \cite{Tamarkin1}. 

For any topological space $X$ and a commutative ring $\bk$ (will be generalized to ring spectra later), let $\Shv(X;\bk)$ be the (large) dg-category (equivalently stable $\infty$-category) of sheaves of $\bk$-modules on $X$. A key notion in microlocal sheaf theory is the \emph{singular support} of a sheaf $\cF\in \Shv(X;\bk)$, denoted by $\SS(\cF)$, which is a closed conic subset in $\mathring{T}^*X=T^*X\backslash \{\text{zero-section}\}$, that collects all the covectors along which non-propagation of sections of $\cF$ occurs. A deep result of Kashiwara--Schapira is that the singular support of any sheaf is coisotropic, which is in some sense 
a mathematical shadow of the Heisenberg Uncertainty Principle that one can only observe the position and momentum of a particle along a coisotropic subvariety in the cotangent bundle, but not smaller than that. 

Among the coisotropic subvarieties in $T^*X$, the class of Lagrangian submanifolds, which is having the smallest possible dimension ($=\dim X$), has drawn particular interest in the study of symplectic topology, due to their nice intersection theory through Floer theory, simple moduli property by Arnold's Nearby Lagrangian conjecture for exact Lagrangians and the analysis for special Lagrangians, and easier quantization results (than other coisotropic varieties) by microlocal sheaves or holonomic $D$-modules. 

The approach of microlocal sheaf theory to the study of Lagrangian submanifolds (or more generally subvarieties) goes as follows. For any (exact) Lagrangian $L\subset T^*X$, there is a standard procedure to transform $L$ into a conic Lagrangian $\bL$ inside $T^{*,<0}(X\times \bR_t)$, where $T^{*,<0}(X\times \bR_t)$ is the open locus in $T^*(X\times \bR_t)$ consisting of covectors whose pairing with $\partial_t$ is strictly negative. Now one can study the category of sheaves with singular support contained in $\bL$ (more precisely some localization of it), denoted by $\Shv_{\bL}^{<0}(X\times \bR_t;\bk)$. This can be regarded as sheaf quantizations of $L$. Since it is invariant under Hamiltonian isotopies of $L$, the sheaf category is an interesting symplectic invariant of $L$. It is closely related to the Fukaya category of $T^*X$ (cf. \cite{Tamarkin1},  \cite{Guillermou}). 

The sheaf quantization has a generalization over ring spectra (cf. \cite{JiTr}). It is broadly expected that working over ring spectra and making connections to stable homotopy theory would help in understanding the topology of Lagrangian submanifolds, especially shed new light on Arnold's Nearby Lagrangian conjecture (cf. \cite{CJS}, \cite{AbKr}). 

The first step to understand $\Shv_{\bL}^{<0}(X\times \bR_t;\bk)$ is to understand the sheaves \emph{microlocally} along $L$. Namely, following the pioneering work of Kashiwara--Schapira, one can assign a sheaf of microlocal sheaf categories along $L$ (sometimes referred as Kashiwara--Schapira stack), denoted by $\mu\cShv_L$, whose sections over an open set $\cV\subset L$ form a certain sheaf category that only involves the local geometry of $\cV$ (inside a tubular neighborhood). There is a canonical functor
\begin{align*}
\Shv_{\bL}^{<0}(X\times \bR_t;\bk)\longrightarrow \Gamma(L,\mu\cShv_L)
\end{align*}
that can be roughly interpreted as taking microlocal stalks. 

If $L$ is smooth, then $\mu\cShv_L$ is locally constant with fiber equivalent to $\Mod(\bk)$, and this is the focus of this paper as it appears in Problem \ref{intro: problem}. In nice situations, e.g. $L$ is a closed (embedded) exact Lagrangian submanifold, then knowing $\Gamma(L,\mu\cShv_L)$ is enough to recover $\Shv_{\bL}^{<0}(X\times \bR_t;\bk)$ (cf. \cite{Guillermou}, \cite{JiTr}). So to understand $\Shv_{\bL}^{<0}(X\times \bR_t;\bk)$, it is crucial to first understand $\mu\cShv_L$, captured by the classifying map (\ref{eq: intro classify}). 

The understanding of the classifying map for $\mu\cShv_L$ is also crucial in the definition of a microlocal sheaf category for a Lagrangian skeleton $\Lambda$ (e.g. inside a Weinstein manifold). In general, there is no canonical symplectic framing as in the cotangent bundle case, so basically one first takes the (stable) symplectic frame bundle (or a thickening of it) of the exact symplectic neighborhood of $\Lambda$, which records the auxiliary data of how one can view local pieces of $\Lambda$ inside cotangent bundles, and define a sheaf of microlocal categories there, denoted by $\mu\cShv_{\widetilde{\Lambda}}$. The stable symplectic group, which is homotopy equivalent to the stable unitary group $U$, acts on the symplectic frame bundle, inducing a convolution action 
\begin{align*}
\mu\cShv_U\boxtimes \mu\cShv_{\widetilde{\Lambda}}\longrightarrow \mu\cShv_{\widetilde{\Lambda}}
\end{align*}
where $U$ is viewed as the stable version of its action Lagrangian graph inside $\varinjlim\limits_N T^*U(N)\times (T^*\bR^N)^-\times T^*\bR^N$. This exhibits $\mu\cShv_{\widetilde{\Lambda}}$ as a twisted equivariant sheaf of categories with respect to the character $U\longrightarrow B\Pic(\bk)$ that classifies $\mu\cShv_U$. This character is easy to write down using Theorem \ref{thm: main intro}, and this enables one to define a (possibly) twisted sheaf of microlocal categories on $\Lambda$ through a (non-canonical choice of) twisted equivariant descent. More details of this approach will appear in a forthcoming paper. There is an alternative construction of $\mu\cShv_\Lambda$ using h-principle \cite{Shende}, \cite{NaSh}.

\subsection{Sketch of the proof of Theorem \ref{thm: main intro}}
In this subsection, we give a sketch of the proof of the main theorem, and highlight the underlying geometry and key tools. Here for simplicity, we will assume $X=\bR^N$. For a general manifold $X$, everything can be reduced to the Euclidean case by embedding $X$ into $\bR^N, N\gg 0$ (see Subsection \ref{subsec: general X}).  

The first ingredient is the class of correspondences, which we call (local) Morse transformations, that gives local trivializations of $\mu\cShv_L$.  Roughly speaking, the Morse transformations for a given small open set $\cV\subset L$ is the class of \emph{local} symplectomorphisms of $T^*\bR^N$ such that 
\begin{itemize}
\item its graph in $(T^*\bR^N)^-\times T^*\bR^N$ admits a generating function in the base variables $(\bq_0,\bq_1)$ of $\bR^N\times \bR^N$ (i.e. its projection to the base is a local isomorphism),  

\item it transforms $\cV$ to a Lagrangian graph in $T^*\bR^N$ (i.e. admits a generating function in $\bq_1$). 
\end{itemize}
The Morse transformations have appeared as a generic class of contact transformations in \cite{Kashiwara}, and they are closely related to stratified Morse functions as in stratified Morse theory developed by Goresky--MacPherson \cite{stratified-morse-theory}. For more discussions in this aspect, see \cite{Jin}.

The second ingredient is the connection between Morse transformations for a germ of smooth Lagrangian $(L,x)$ (modulo equivalences given by post-composing with diffeomorphisms on the base manifold) and the space of paths from $\ell_x=:T_xL$ as an affine Lagrangian in $T^*\bR^N$ to the zero-section. The upshot is that these two spaces, after stabilizations, are homotopy equivalent up to a ``group completion". 

More explicitly, at the tangent space level, a Morse transformation is doing a Fourier transform, denoted by $\FT$, then followed by the time-1 map of a quadratic Hamiltonian depending only on $\bp$ (the momentum coordinates) that transforms $\FT(\ell_x)$ into a graph-type affine Lagrangian. If we view the open contractible locus $\cB$ of graph-type Lagrangians in the stable (affine) Lagrangian Grassmannian as a ``fat" base point, then each Morse transformation of $(L,x)$ can be assigned a path starting from $\ell_x$ and ending at $\cB$ as the concatenation of a path $\FT_t(\ell_x), 0\leq t\leq \frac{\pi}{2}$ (where one can take the Hamiltonian flow of $\frac{1}{2}(|\bq|^2+|\bp|^2)$ for $\FT_t$) and the image of $\FT(\ell_x)$ under the flow of a certain quadratic Hamiltonian in $\bp$. From this point of view, the space of such quadratic Hamiltonians involved (after stabilization) is homotopy equivalent to a torsor of 
\begin{align}\label{eq: intro nonneg}
\coprod\limits_{n}\{\text{nonnegative quadratic forms in } \bp\text{ of rank }n\}\simeq \coprod\limits_nBO(n), 
\end{align}
as a topological monoid. The Hamiltonian flows of the latter give based loops of $U/O$ at $\cB$, and this is an incidence of Bott periodicity. 

Now we can explain the key point of the proof. Let $\cP_L\rightarrow L$ be the pullback of the universal principal $\Omega(U/O)\simeq \bZ\times BO$-bundle over $U/O$ through the stable Gauss map $\gamma: L\rightarrow U/O$. Each Morse transformation for an open set $\cV\subset L$ gives a correspondence 
\begin{align*}
\xymatrix{&H\ar[dl]_{p_1}\ar[dr]^{p_0}&\\
\bR_{\bq_1}^N\times \bR_{t_1}&&\bR_{\bq_0}^N\times \bR_{t_0}
,}
\end{align*}
where $H\subset \bR^N_{\bq_0}\times \bR_{t_0}\times \bR^N_{\bq_1}\times  \bR_{t_1}$ is a smooth hypersurface, that establishes an equivalence of categories
\begin{align}\label{eq: intro trivialization}
(p_1)_*p_0^!: \mu\cShv_L(\cV;\bk)\overset{\sim}{\longrightarrow} \Loc(\cV;\bk),
\end{align}
i.e. a trivialization of $\mu\cShv_L$ on $\cV$. Now for simplicity, let us take $\cV$ sufficiently small and contractible around a point $x\in L$ (so that $\Loc(\cV;\bk)\simeq \Mod(\bk)$), then from the above remarks, we can (almost) think of $\cP_L|_{\cV}$ as parametrizing the Morse transformations, hence we get a family of trivializations (\ref{eq: intro trivialization}). Now for each object $\cF\in\mu\cShv_L(\cV;\bk)$, we get a family of $\bk$-modules, assembled into a local system on $\cP_L|_{\cV}$. The key proposition that we establish is the following.  
\begin{prop}[see Proposition \ref{prop: muShv, J-equiv} and Proposition \ref{prop: proof of thm}
]\label{prop: intro key}
For $\bk=\bS$ (the sphere spectrum), there is a natural equivalence of categories 
\begin{align*}
\mu\cShv_L(\cV;\bS)\overset{\sim}{\longrightarrow}  \Loc(\cP_L|_{\cV};\bS)^{J\text{-equiv}},
\end{align*}
that is compatible with restrictions along open inclusions $\cV'\hookrightarrow \cV$. 
\end{prop} 
Here $\cP_L|_{\cV}$ is viewed as a torsor over $\Omega(U/O)$ in the $\infty$-category of spaces $\Spc$, and the $J$-homomorphism $J: \Omega(U/O)\longrightarrow \Pic(\bS)$ is viewed as a character so that one can talk about $J$-equivariant local systems on $\cP_L|_{\cV}$. Then it is easy to see that Proposition \ref{prop: intro key} immediately implies our main result. 

The proof of Proposition \ref{prop: intro key} is essentially contained in \cite{Jin}, though the part of compatibility with restrictions under open inclusions was not discussed there (due to the non-flexibility of the model representing (\ref{eq: intro nonneg}) that we chose there). The idea is to quantize the Hamiltonian action of the space of nonnegative quadratic Hamiltonians depending only on $\bp$ as in (\ref{eq: intro nonneg}), which we referred as the Hamiltonian $\coprod\limits_n BO(n)$-action in \emph{loc. cit.}. Then the tautological vector bundle on $\coprod\limits_n BO(n)$ shows up, which gives rise to the $J$-homomorphism. 

Since there is a nontrivial amount of structures and functorialities involved in Proposition \ref{prop: intro key}, one needs to adopt the right machinery to write down a clean argument. First, the foundations on sheaves of spectra are contained in \cite{higher-topos} and \cite{higher-algebra}. Second, we employ the $(\infty,2)$-category of correspondences developed by Gaitsgory--Rozenblyum \cite{GaRo}, which is designed to treat the six-functor formalism in a clean and functorial way. With our inputs on concrete constructions of (commutative) algebra objects and their modules together with (right-lax) morphisms between them in the category of correspondences $\bCorr(\cC)$, for a Cartesian symmetric monoidal $(\infty,1)$-category $\cC$, in \cite{Jin} (further upgraded in Appendix \ref{sec: Appendix}), we are able to choose the appropriate models (explained in Section \ref{sec: proof}) to achieve the functorial results.

\subsection{Organization} 
The organization of the paper goes as follows. In Section \ref{sec: Bott}, we give some account of the geometry of $U/O$ and review one step of Bott periodicity $\Omega(U/O)\simeq \bZ\times BO$. The point of view is essentially following \cite{Harris}. In Section \ref{sec: microlocal}, we briefly review microlocal sheaf theory over ring spectra and define the sheaf of microlocal categories $\mu\Shv_L$ on a smooth Lagrangian $L$ in $T^*X$. We present two equivalent versions of $\mu\Shv_L$ with the latter version (denoted by $\mu\cShv_L$) adopted in later arguments, for it is convenient to run arguments with correspondences using this version. Then we recall Morse transformations and relate them to paths in $U/O$, which illustrates the geometry underlying the proof of the main theorem. We finish the section by recalling the category of correspondences (\cite{GaRo}) and the functor $\Shv\Sp_*^!$. In Section \ref{sec: proof}, we give the proof of the main result and the application Corollary \ref{cor: AbKr}. As explained before, the key idea of quantizing a Hamiltonian $\coprod\limits_n BO(n)$-action is already contained in \cite{Jin}. The additional ingredient here is the construction of an $\infty$-category $\QHam_L(U/O)$ over $\Open(L)^{op}$ that roughly plays the role of a ``principal $\coprod\limits_nBO(n)$-bundle" to approximate the pullback of the universal principal $\Omega(U/O)$-bundle along the Gauss map. Finally, in Appendix \ref{sec: Appendix}, we give the construction of a canonical functor from a correspondence category of $\Fin_*$-objects of a Cartesian symmetric monoidal $\infty$-category $\cC$ to the category of commutative algebra objects in $\bCorr(\cC)$ with right-lax homomorphisms (and also a module version). This is an upgraded version of the results that we have obtained in \cite{Jin} about constructions of (commutative) algebra and their modules in $\bCorr(\cC)$, and it is frequently (sometimes implicitly) used in the arguments throughout Section \ref{sec: proof}. We also include a list of notations and conventions for categories in Appendix \ref{sec: notations}.

\subsection{Acknowledgement}
I am grateful to David Nadler and Dima Tamarkin for many inspiring discussions and for valuable feedbacks on this work. I would like to thank David  Treumann for useful comments on an earlier draft. I have benefited from discussions with St\'ephane Guillermou that inspired the proof of Corollary \ref{cor: AbKr}. This work is partially supported by an NSF grant DMS-1854232. 

\section{Geometry of $U/O$ and Bott periodicity}\label{sec: Bott}

\subsection{Spectral decomposition for linear Lagrangians}\label{subsec: spectral linear}
Let $\omega_0$ be the standard symplectic form on $T^*\bR^N$, and let $J_0$ be the standard complex structure on $T^*\bR^N$ compatible with $\omega_0$ determined by $g_0(-,-)=\omega_0(-,J_0(-))$ is the Euclidean metric on $\bR^{2N}$.  
For any linear Lagrangian $\ell\in T^*\bR^N$, it determines a quadratic form $A_\ell$ on its projection to the zero-section, given by $A_\ell(u,v)=\widetilde{u}(v)$, where $\widetilde{u}$ is any lifting of $u$ in $\ell$ under the projection to be base. Note that $A_\ell$ is \emph{twice} of a primitive of $\ell$. Let $\ell_{[\infty]}$ be the cotangent fiber at 0. 
\begin{lemma}\label{lemma: ell_infty perp}
For any linear Lagrangian $\ell\subset T^*\bR^N$, the subspace $J_0(\ell\cap \ell_{[\infty]})$ is the orthogonal complement to $\proj_{\bR^N}(\ell)$. 
\end{lemma}
\begin{proof}
For any $u\in \proj_{\bR^N}\ell$ and $v\in \ell\cap \ell_{[\infty]}$, we have 
\begin{align*}
g_0(J_0v, u)=\omega_0(v,u)=\omega_0(v, \widetilde{u})=0,
\end{align*}
where $\widetilde{u}$ is any lifting of $u$ in $\ell$ as before. So we see that $J_0(\ell\cap \ell_{[\infty]})$ is orthogonal to $\proj_{\bR^N}(\ell)$. By dimension reason, they are orthogonal complements to each other. 
\end{proof}
By Lemma \ref{lemma: ell_infty perp}, any linear Lagrangian $\ell\subset T^*\bR^N$ is determined by a subspace $J_0(\ell\cap \ell_{[\infty]})\subset \bR^N$ and a linear Lagrangian graph 
\begin{equation}\label{eq: A_ell}
\ell_{fin}=\Gamma(\frac{1}{2}dA_\ell)\subset T^*(J_0(\ell\cap \ell_{[\infty]}))^\perp,
\end{equation}
and we have $\ell=J_0(\ell\cap \ell_{[\infty]})\oplus \ell_{fin}$ under the orthogonal splitting $T^*\bR^N=T^*(J_0(\ell\cap \ell_{[\infty]}))\oplus T^*(J_0(\ell\cap \ell_{[\infty]}))^\perp$. By this observation, we can view $\ell$ as a \emph{generalized quadratic form} in the sense that it has eigenvalues range in $\bR\cup \{\infty\}\cong S^1$, and its spectral decomposition is equal to the sum of $A_\ell$ and the eigenspace of $\infty$ given by $J_0(\ell\cap \ell_{[\infty]})$.  Such a point of view will be very useful to us. First, it will give a convenient way to describe and prove one step of Bott periodicity: $\Omega(U/O)\simeq \bZ\times BO$. Second, we will often use it to give an easy description of a small neighborhood of any $\ell\in U(N)/O(N)$, namely, those linear Lagrangians given by small deformations of the spectral decomposition of $\ell$. 

Regarding the stabilization process $U/O=\varinjlim\limits_{N}U(N)/O(N)$, where the inclusions $U(N)/O(N)\hookrightarrow U(N+1)/O(N+1)$ are given by direct sum with the zero-section in the extra dimension, a \emph{stable} linear Lagrangian has spectral decomposition given by a finite collection of nonzero (possibly including $\infty$) distinct eigenvalues $\lambda_1,\cdots, \lambda_k$, and finite dimensional eigenspaces $V_{\lambda_1},\cdots, V_{\lambda_k}$, and the orthogonal complement of $\bigoplus\limits_{j=1}^k V_{\lambda_j}$ (which is infinite dimensional) serves as the eigenspace of $0$.

\subsection{One step of Bott periodicity: $\Omega(U/O)\simeq \bZ\times BO$}\label{subsec: Bott}

Recall the ordinary 1-category $\Gamma$ defined in \cite{Segal2}, which is canonically equivalent to $\Fin_*^{op}$, and a $\Gamma$-space is equivalent to a commutative algebra object in $\Spc$. To be consistent with notations used in the rest of the paper, we will use $\Fin_*^{op}$ instead of $\Gamma$. 
If we have fixed an identification $\lng n\rng\cong \{1,\cdots, n,*\}$, then one has a natural functor  $\Delta\rightarrow \Fin_*^{op}$, taking $[n]\in \Delta$ to $\lng n\rng$ and a morphism $f: [n]\rightarrow [m]$ to $\hat{f}: \lng m\rng\rightarrow \lng n\rng$ determined by requiring $\hat{f}^{-1}(k)=\{f(k-1)+1,\cdots, f(k)\}, k\in \lng n\rng^\circ$ (here $\lng n\rng^\circ=\lng n\rng\backslash\{*\}$).  Using this functor, one can turn a $\Fin_*$-space into a simplicial space,  and one takes its geometric realization as the \emph{geometric realization} of the original $\Fin_*$-space.  

Recall the definition of the delooping functor  in \emph{loc. cit.}
\begin{align*}
\sfB: \CAlg(\Spc)&\longrightarrow \CAlg(\Spc)\\
A&\mapsto \sfB A
\end{align*}
where 
$\sfB A: N(\Fin_*)\rightarrow \Spc$ takes $\lng n\rng$ to the geometric realization of the $\Fin_*$-object $\lng m\rng\mapsto A(\lng m\rng\wedge \lng n\rng)$, induced from the canonical functor 
\begin{align*}
N(\Fin_*)\times N(\Fin_*)&\longrightarrow N(\Fin_*)\\
(\lng n\rng,\lng m\rng)&\mapsto \lng n\rng\wedge \lng m\rng\cong(\lng n\rng^\circ\times \lng m\rng^\circ)\cup\{*\}.
\end{align*}
Since $\CAlg(\Spc)$ is presentable and $\sfB$ preserves colimits, it admits a right adjoint $\sf\Omega$ which is the based loop functor. The natural map $A\rightarrow \sfOmega\sfB A$ from adjunction, for any $A\in \CAlg(\Spc)$,  serves as the group completion for $A$ (cf. \cite{Quillen}).

Now following the construction in \cite{Harris}, let $G^\bullet$ be the $\Fin_*$-space defined as follows:
\begin{align*}
&G^{\lng 0\rng}=pt, \\
&G^{\lng n\rng}:=\{(E_j)_{j\in \lng n\rng^\circ}: E_j\overset{\mathrm{linear}}{\subset} \varinjlim\limits_N\bR^N, \dim E_j<\infty, E_j\perp E_k\text{ for }j\neq k\};
\end{align*}
for any $f: \lng n\rng\rightarrow \lng m\rng$, 
\begin{align*}
f_*:& G^{\lng n\rng}\longrightarrow G^{\lng m\rng}\\
&(E_j)_{j\in \lng n\rng}\mapsto (\bigoplus\limits_{j\in f^{-1}(k)}E_{j})_{k\in \lng m\rng^\circ}. 
\end{align*}
There is a natural homeomorphism from $\sfB G^{\lng 1\rng}$ to $U/O$ induced from  
\begin{align}
\nonumber&\coprod\limits_n(G^{\lng n\rng}\times |\Delta^n|)\longrightarrow U/O\\
\label{eq: homeo U/O}&((E_j)_{1\leq j\leq n}, (0\leq t_1\leq \cdots\leq t_n\leq 1)) \mapsto (E_j, -\tan(\pi t_j))_{1\leq j\leq n}. 
\end{align}
Here we have fixed an identification between $\lng n\rng$ and $\{1,\cdots, n,*\}$, and  $(E_j, -\tan(\pi t_j))_{1\leq j\leq n}$ denotes for the spectral decomposition of the image Lagrangian. Here we adopt the convention that  the orthogonal complement of $\bigoplus\limits_{j=1}^nE_j$ has eigenvalue 0. Since the map (\ref{eq: homeo U/O}) is compatible with the gluing map for the geometric realization, it factors through $\sfB G^{\lng 1\rng}$ and gives the desired homeomorphism.  The group completion $G^{\lng 1\rng}\rightarrow \sfOmega\sfB G^{\lng 1\rng}\simeq \Omega(U/O)$ takes $E\in G^{\lng 1\rng}$ to the based loop $\ell(t),t\in [0,1]$, where $\ell(t)$ has spectral decomposition $(E, -\tan(\pi t))$. 

To see the equivalence $B(\bZ\times BO)\simeq U/O$, one considers the following $\Fin_*$-topological categories (the objects form a space rather than just a set)
\begin{align*}
&\widetilde{G}^{\lng 0\rng}=pt, \\
&\widetilde{G}^{\lng n\rng}:=\begin{cases}&\text{Objects}: 
(E_j, E_j')_{j\in \lng n\rng^\circ}: E_j, E_j'\subset\varinjlim\limits_N\bR^N \text{ finite dimensional}, E_j\perp E_k,  E'_j\perp E'_k\text{ for }j\neq k\\
&Hom((E_j, E_j')_j, (F_j, F_j')_j)=\begin{cases}*, \text{ if }\exists W_j\perp E_j, E_j' \text{ with } E_j\oplus W_j=F_j, E_j'\oplus W_j=F_j'\\
\emptyset, \text{ otherwise}
\end{cases}
\end{cases}.
\end{align*}
Taking the geometric realization of the $\Fin_*$-space $|\widetilde{G}^{\lng \bullet\rng}|$, where each $|\widetilde{G}^{\lng n\rng}|$ means the \emph{classifying space} of the topological category in the sense of \cite{Segal}, i.e. the geometric realization of its nerve, gives $B|\widetilde{G}^{\lng 1\rng}|\simeq B(\bZ\times BO)$. There is a homemorphism from the geometric realization of $|\widetilde{G}^{\lng \bullet\rng}|$ to $(U/O\times U/O)/U/O\simeq U/O$, where $U/O$ acts on the two factors by the diagonal action of taking direct sum of Lagrangian subspaces.

\section{Microlocal sheaf categories for Lagrangian submanifolds}\label{sec: microlocal}
In this section, we first recall the construction of the sheaf of microlocal categories on a smooth Lagrangian following \cite{Kashiwara}, which generalizes to the ring spectra setting without much change (for more details see \cite{JiTr}), then we recall the notion of Morse transformations \cite{Jin} and make connections to the space of paths in $U/O$ ending at the zero-section. Lastly, we briefly review the category of correspondences and collect the results that we will need. 

\subsection{Basics in microlocal sheaf theory}
For any smooth manifold $Y$ and a commutative (or at least $\bE_2$) ring spectrum $\bk$, let $\Shv(Y;\bk)$ be the stable $\infty$-category of all sheaves valued in $\bk$-spectra on $Y$. The notion of singular support for any sheaf over an ordinary ring defined in \cite{Kashiwara} can be extended in the spectra setting without change.
\begin{definition}\cite[Definition 5.1.2]{Kashiwara} For any sheaf $\cF\in\Shv(Y;\bk)$, a covector $(x,\xi)\in\mathring{T}^*Y=T^*Y\backslash T^*_YY$ is \emph{not} in the \emph{singular support} of $\cF$, which is denoted by $\SS(\cF)$, if there exists an open neighborhood $U$ of $(x,\xi)$ such that for any covector $(\tilde{x},\tilde{\xi})\in U$ and any smooth function $f: X\rightarrow \bR$ satisfying $f(\tilde{x})=0, df(\tilde{x})=\tilde{\xi}$, we have 
\begin{align*}
\Gamma_{\{f\geq 0\}}(\cF)|_{\tilde{x}}\simeq 0. 
\end{align*}
\end{definition}
It is clear from the definition that $\SS(\cF)$ is a closed conic subset of $\mathring{T}^*Y$. A deep theorem of Kashiwara--Schapira \cite{Kashiwara} asserts that $\SS(\cF)$ is always coisotropic. For any closed conic (coisotropic) subset $C\subset \mathring{T}^*Y$, define
\begin{align*}
\Shv_C(Y;\bk)=:\{\cF\in \Shv(Y;\bk): \SS(\cF)\subset C\}\subset\Shv(Y;\bk), 
\end{align*}
which is a full stable subcategory of $\Shv(Y;\bk)$ closed under infinite direct sums. 

We remark that the parts of microlocal sheaf theory that we will rely on throughout the paper are the functorial properties of singular support under the six functors, which are contained in \cite[Section 5.4]{Kashiwara} (see also \cite[Appendix 2]{Tamarkin1}). Although the original proofs are for sheaves over ordinary rings and for the bounded derived category, all the arguments can be carried out for the spectra setting without essential change. In the case when $C=\Lambda$ is a conic Lagrangian, $\Shv_\Lambda(Y;\bk)$ is relative easy to understand, which only has weakly constructible sheaves, and this is the situation that we will mostly focus on.

\subsection{The sheaf of microlocal categories on a smooth Lagrangian}

Let $X$ be a smooth manifold. The cotangent bundle $T^*X$ is an exact symplectic manifold, in the sense that the symplectic form $\omega=d\bq\wedge d\bp=\sum\limits_jdq_j\wedge dp_j$ has a primitive $\alpha=\bp d\bq$ (up to a negative sign), where $(\bq,\bp)$ is any local Darboux coordinate. 
For a possibly non-conic \emph{exact} Lagrangian $L\subset T^*X$, i.e. $\alpha|_L$ is an exact 1-form, there is a standard way to turn it into a conic Lagrangian in a larger cotangent bundle as follows. 

Let $T^{*,<0}(X\times \bR_t)$ denote for the negative half of the cotangent bundle $T^*(X\times \bR_t)$ consisting of covectors that have strictly negative coefficients in the $dt$-factor. 
For a smooth exact Lagrangian $L\subset T^*\bR^N$, fix a primitive $f_L$ for $\alpha|_L$, i.e. $\alpha|_L=df_L$. Let $\bL$ be the associated \emph{conic Lagrangian lifting} in $T^{*,<0}(\bR^N\times \bR_t)$ defined by 
\begin{align*}
\bL=:\{(\bq,\bp; t,\tau): (\bq,-\bp/\tau)\in L, t=f_L(\bq,\bp), \tau<0\}\subset T^*\bR^N\times \bR_t\times \bR_{\tau<0}. 
\end{align*}
It is usually convenient to view $\bL$ as the cone over the Legendrian $\bL_{\tau=-1}$ in the $1$-jet space $(T^*\bR^N\times \bR_t,-dt+\alpha)$. For any open subset $\cU\subset L$, we will use $\Cone(\cU)$ to denote the conic lifting of $\cU$, as a conic open subset in $\bL$. 

For simplicity, in the following discussions, we will restrict to the case where $X=\bR^N$ (or denoted by $V$ as a vector space). Eventually, the proof of our main theorem reduces to the $\bR^N$ case, since for any smooth manifold $X$, we can embed it into $\bR^N, N\gg 0$ and the case for $X$ follows easily then (see Subsection \ref{subsec: general X} for the details).

Following the work \cite{Kashiwara}, one can define a sheaf of microlocal sheaf categories on $L$ (over some fixed ring spectrum $\bk$) as follows. For any open subset $\cU\subset L$, set
\begin{align}\label{eq: muShv_L}
\mu\Shv_{L}^{pre}(\cU;\bk):=\varinjlim\limits_{\substack{\Omega\in \ConicOpen^{op}(T^{*}(\bR^N\times \bR_t))\\
\Cone(\cU)\overset{\text{closed}}{\subset} \Omega}}\Shv_{\Cone(\cU)\cup\Omega^c}(\bR^N\times\bR_t;\bk)/\Shv_{\Omega^c}(\bR^N\times\bR_t;\bk),
\end{align}
and the restriction maps are the natural ones. By the invariance of microlocal sheaf categories under quantized contact transformations, for sufficiently small $\cU\subset \bL$, we have a non-canonical equivalence $\mu\Shv_{L}^{pre}(\cU;\bk)\simeq \Loc(\cU;\bk)$ compatible with restrictions (see Section \ref{subsec: Morse transf} for more details). 
This implies that the sheafification of $\mu\Shv_{L}^{pre}$, denoted by $\mu\Shv_L$, is a locally constant sheaf of categories $\mu\Shv_{L}$ on $L$ with fiber equivalent to $\Mod(\bk)$.

Note that since the construction of $\mu\Shv_{L}$ comes from local data, one can drop the exact condition on the Lagrangian, and the definition also makes sense for a smooth Lagrangian immersion.

\subsection{An alternative construction of $\mu\Shv_L$}\label{subsec: alternative}
We present an equivalent construction of $\mu\Shv_L$ for a smooth Lagrangian $L\subset T^*\bR^N$ based on weighted blow up, which has the advantage of avoiding the quotient by $\Shv_{\Omega^c}(\bR^N\times\bR_t;\bk)$ in (\ref{eq: muShv_L}), and which later gives us a convenient and precise setting to perform Morse transformations (cf. Section \ref{subsec: Morse transf}). The approach presented here is largely inspired by a course given by D. Tarmarkin in 2016. 

Let $V=\bR^N$. Let 
\begin{align*}
&\cV_L=L\times V\times \bR_t,\text{ and } \\
&\pi_L: \cV_L\longrightarrow L,\ \pi_{V\times\bR_t}: \cV_L\longrightarrow V\times\bR_t
\end{align*}
be the obvious projections. Assume as before that $L$ is exact and fix a primitive $f_L$, and let 
\begin{align*}
\iota_{f_L}: L&\hookrightarrow \cV_L\\
(\bq, \bp)&\mapsto (\bq,\bp; \bq; f_L(\bq,\bp))
\end{align*}
be the ``diagonal" embedding. Now we do a version of deformation to the normal cone with respect to the following $\bR_+$-action on the normal bundle of $\iota_{f_L}$, whose fiber at $\iota_{f_L}(\bq_0,\bp_0)$ is canonically identified with $V\times \bR_t$ with shifted center at $(\bq_0, f_L(\bq_0,\bp_0))$:
\begin{align*}
&s\cdot (\bq-\bq_0, t-f_L(\bq_0,\bp_0))=(s(\bq-\bq_0), s\bp_0\cdot (\bq-\bq_0)+s^2(t-f_L(\bq_0, \bp_0)-\bp_0\cdot (\bq-\bq_0)), \\
&\ \ \ \ \ \ \text{ for }(\bq,t)\in V\times \bR_t,\ s\in \bR_+.
\end{align*}
The outcome is 
\begin{align*}
\widetilde{\cV}_L=\Bl^{\textsf{weighted}}_{\iota_{f_L}(L)\times \{0\}}(\cV_L\times \bR_{z\geq 0})-\overline{(\cV_L-\iota_{f_L}(L))}\times \{0\}\longrightarrow \bR_{z\geq0}
\end{align*}
with an open embedding 
\begin{align*}
j: \cV_L\times \bR_{z>0}\hookrightarrow \widetilde{\cV}_L,
\end{align*}
and a closed embedding 
\begin{align*}
i_0: \cV_L\hookrightarrow \widetilde{\cV}_L
\end{align*}
induced from the natural identification between the central fiber and $\cV_L$. 

Taking the tangent space of $L$ for each $x\in L$ gives the (affine) Gauss map from $L$ to the affine Lagrangian Grassmannian
\begin{align*}
L&\longrightarrow \LagGr_{\aff}(T^*V)\\
x&\mapsto \ell_x=T_xL.
\end{align*}
By choosing the primitive of $\alpha|_{\ell_x}$ such that its value at $x$ equals $f_L(x)$, we get a family of Legendrian liftings of the affine Lagrangians $\ell_x, x\in L$. Let 
\begin{align*}
\bL_{\Gauss}\subset T^{*,<0}(L\times V\times \bR_t)
\end{align*}
be the cone over the smooth \emph{closed} Legendrian in the 1-jet bundle $T^*L\times T^*V\times \bR_t$ from this family of Legendrian liftings. 

For any smooth manifold $Y$ and closed conic set $C\subset T^{*,<0}(Y\times\bR_t)$, following \cite{Tamarkin1}, define 
\begin{align*}
&\Shv^{\geq 0}(Y\times\bR_t;\bk)=:\Shv_{\mathring{T}^*(Y\times\bR_t)\backslash T^{*,<0}(Y\times\bR_t)}(Y\times\bR_t;\bk),\\
&\Shv^{<0}(Y\times \bR_t;\bk)=: \Shv(Y\times\bR_t;\bk)/\Shv^{\geq 0}(Y\times\bR_t;\bk),\\
&\Shv_C^{<0}(Y\times \bR_t;\bk)=: \{\cF\in \Shv^{<0}(Y\times \bR_t;\bk): \SS(\cF)\cap T^{*,<0}(Y\times\bR_t)\subset C \}\overset{\text{full}}{\subset}\Shv^{<0}(Y\times \bR_t;\bk). 
\end{align*}
Here we set $\Shv^{<0}(Y\times \bR_t;\bk)$ as the left orthogonal complement of $\Shv^{\geq 0}(Y\times\bR_t;\bk)$ (see Theorem \ref{thm: variant Tamarkin}).

Let $\pi_{V\times \bR_t}: \cV_L\times \bR_{z>0}\rightarrow V\times\bR_t$ be the projection to the $V\times \bR_t$ factor. 
\begin{lemma}\label{lemma: equiv muShv}
For each open $\cU\subset L$, there is an equivalence
\begin{align}\label{lemma eq: def normal}
i_0^*j_*\pi_{V\times \bR_t}^*: \mu\Shv_{L}(\cU;\bk)\longrightarrow \Shv_{\bL_\Gauss}^{<0}(\cU\times V\times \bR_t;\bk).
\end{align}
\end{lemma}
\begin{proof}
First, both sides are invariant under compactly supported Hamiltonian isotopies in $T^*V$ (cf. \cite{GKS}), so without loss of generality, we may assume that $L$ is in general position, i.e. its projection to $V$ is finite-to-one. Then we can cut out small open balls $\Omega_\alpha$ in $\cU\subset L$ by neighborhoods of the form $B_\alpha\times I_\alpha\subset V\times \bR_t$, where $B_\alpha\subset V$ and $I_\alpha\subset \bR_t$ are open balls and intervals respectively (cf. \cite{JiTr}). Then we have 
\begin{align*}
\mu\Shv_L(\Omega_\alpha;\bk)\simeq \Shv_{\Omega_\alpha}(B_\alpha\times I_\alpha;\bk)/\Loc(B_\alpha\times I_\alpha;\bk). 
\end{align*}
Using the above identification, the functor (\ref{lemma eq: def normal}) with $\cU$ replaced by $\Omega_\alpha$ is an equivalence of categories (with both sides non-canonically equivalent to $\Mod(\bk)$). Since $\cU$ has a complete cover by such $\Omega_\alpha$, and the equivalences are compatible with restrictions for any inclusion $\Omega_\beta\subset \Omega_\alpha$, it induces an equivalence of categories for $\cU$. 
\end{proof}

Let 
\begin{align*}
\mu\cShv_L(\cU;\bk)=:\Shv_{\bL_{\Gauss}}^{<0}(\cU\times V\times \bR_t;\bk),
\end{align*}
for any open $\cU\subset L$, which is clearly a sheaf of categories on $L$. 
Then Lemma \ref{lemma: equiv muShv} assures that $\mu\Shv_L\simeq \mu\cShv_L$. It is also immediate from the construction and the $\bR$-equivariance of $\Shv_{\bL_\Gauss}^{<0}(\cU\times V\times \bR_t;\bk)$ under any shift of the primitive $f_L|_\cU$ by a constant, that $\mu\cShv_L(\cU;\bk)$ is also well defined for any smooth (not necessarily exact) Lagrangian immersion. In the following, we will constantly use $\mu\cShv_L$ as our definition for the sheaf of microlocal categories on $L$.

\subsection{Morse transformations and paths in $U/O$}\label{subsec: Morse transf}
Recall the notion of Morse transformations (cf. \cite[Section 5.1]{Jin}). Thanks to the construction in Subsection \ref{subsec: alternative}, we only need to focus on Morse transformations for any conic lifting of an affine Lagrangian $\ell\subset T^*\bR^N$ that are the negative conormal bundles of smooth hypersurfaces in $\bR_{\bq_0}^N\times\bR_{t_0}\times \bR^N_{\bq_1}\times\bR_{t_1}$ of the form (up to some shifts in coordinates)
\begin{align*}
t_1-t_0-(Q+\frac{1}{2}A_\ell)[\bq_0]-\bq_0\cdot \bq_1=0, 
\end{align*}
for some nondegenerate quadratic form $Q$ on $\pi(\ell)$ and $A_\ell$ is as in Subsection \ref{subsec: spectral linear}  (after parallel shifting $\ell$ to be a linear Lagrangian). Here the negative conormal bundle means the part of conormal bundle of the hypersurface inside $(T^{*,<0}(\bR_{\bq_0}^N\times\bR_{t_0}))^-\times T^{*,<0}(\bR_{\bq_1}^N\times\bR_{t_1})$. One should think of each Morse transformation corresponding to a smooth Lagrangian in 
$T^*(\bR_{\bq_0}^N\times \bR_{\bq_1}^N)$ of the form
\begin{align}\label{eq: generating function}
\Graph(d((Q+\frac{1}{2}A_\ell)[\bq_0]+\bq_0\cdot\bq_1)),
\end{align}
which is a symplectomorphism that has a generating function in $\bq_0,\bq_1$. We stick to the convention that for a quadratic form $Q$, thought of as a symmetric matrix, $Q[\bq]=\bq^T\cdot Q\bq$ means the value of the quadratic function, that is to distinguish from the matrix-vector product $Q\bq$ (which will also show up).

The key point of considering Morse transformations is that they provide local trivializations of $\mu\cShv_L$. For any open $\cU\subset \bL$, if there is a Morse transformation $\bL_{01}$ for $\cU$ with underlying smooth hypersurface $H$, then the correspondence 
\begin{align}\label{diagram: corr H}
\xymatrix{&\cU\times H\ar[dr]^{p_0}\ar[dl]_{p_1}&\\
\cU\times \bR_{\bq_1}^N\times \bR_{t_1}& &\cU\times \bR_{\bq_0}^N\times \bR_{t_0}
}
\end{align}
induces an equivalence of categories 
\begin{align}\label{eq: corr H cU}
(p_1)_*p_0^!: \mu\cShv_{L}(\cU;\bk)\longrightarrow \Shv^{<0}_{((\Delta T^*\cU)\times\bL_{01})\circ \bL_{\Gauss}}(\cU\times \bR_{\bq_1}^N\times \bR_{t_1};\bk).
\end{align}
Since $((\Delta T^*\cU)\times\bL_{01})\circ \bL_{\Gauss}|_{\cU}$ is the negative conormal bundle of a smooth hypersurface whose projection to $\cU$ is a trivial fibration, the right-hand-side of  (\ref{eq: corr H cU}) is naturally identified with $\Loc(\cU;\bk)$.

Now we will explain some close relation between Morse transformations and paths in $U/O$ ending at the base point (the stable zero-section). Technically, we should use the stable affine Lagrangian Grassmannian instead of $U/O$, but since the difference is trivial at the topological level, we will always ignore such issues. It is clear from (\ref{eq: generating function}) that the underlying symplectomorphism of a Morse transformation is the composition of the Fourier transform, denoted by $\FT$, and the time-1 Hamiltonian map of the quadratic Hamiltonian function purely in $\bp$, given by $(Q+\frac{1}{2}A_\ell)[\bp]$.

Let 
\begin{align}\label{eq: FT_t}
\FT_t=\varphi_{\frac{1}{2}(|\bp|^2+|\bq|^2)}^{\frac{\pi}{2}t}, 0\leq t\leq 1
\end{align} 
be the path of linear symplectomorphism with $\FT_0=id$ and $\FT_1=\FT$, given by the Hamiltonian flow of the function $\frac{1}{2}(|\bp|^2+|\bq|^2)$. If we think of the contractible locus of graph like linear Lagrangians in $U/O$ as a ``fat" base point, then we can assign each Morse transformation the following path in $U(N)/O(N)\subset U/O$ starting from $\ell$ and ending at the base point:
\begin{align*}
(\varphi_{Q+\frac{1}{2}A_\ell}^t\bullet \FT_t)(\ell):=\begin{cases}&\FT_{2t},\ 0\leq t\leq \frac{1}{2}\\
&\varphi_{Q+\frac{1}{2}A_\ell}^{2t-1}\circ\FT,\ \frac{1}{2}\leq t\leq 1.
\end{cases}
\end{align*} 
This point of view is a key ingredient in our proof for the main theorem.

\subsection{The category of correspondences and the functor $\ShvSp_*^!$}\label{subsec: ShvSp}
In this subsection, we recall some basic facts about the correspondence category defined in \cite{GaRo}, and the symmetric monoidal functor $\ShvSp^!_*$ that we will use later in the proof of the main theorem. 
In \cite{Jin}, we proved based on \cite{GaRo}, \cite{higher-topos} and \cite{higher-algebra} that there is a canonically defined  symmetric monoidal functor 
\begin{align*}
\nonumber\ShvSp_!^*:\bCorr(\Slch)_{\all,\all}^{\propmap}&\longrightarrow (\bPrstL)^{\twoop}\\
\nonumber X&\mapsto \Shv(X;\Sp)\\
\Big(\xymatrix{Z\ar[r]^{p}\ar[d]_{q}&X\\
Y&
}
\Big)&\mapsto \Big(q_!p^*: \Shv(X;\Sp)\rightarrow \Shv(Y;\Sp)\Big).
\end{align*}
By the same approach, one gets the canonical symmetric monoidal functor 
\begin{align*}
\ShvSp_*^!: &\bCorr(\Slch)^{\propmap}_{\all,\all}\rightarrow \bPrstR\\
\nonumber X&\mapsto \Shv(X;\Sp)\\
\Big(\xymatrix{Z\ar[r]^{p}\ar[d]_{q}&X\\
Y&
}
\Big)&\mapsto \Big(q_*p^!: \Shv(X;\Sp)\rightarrow \Shv(Y;\Sp)\Big).
\end{align*}
The second functor is the one that we will use later. 

Recall that in \cite[Theorem 2.6, Theorem 2.11, Theorem 2.21, Proposition 2.22]{Jin}, we provided a list of constructions of 
commutative algebra objects, their modules and (right-lax) homomorphisms among them in $\bCorr(\cC)$, for $\cC$ a Cartesian symmetric monoidal $\infty$-category, out of concrete data including $\Fin_*$-objects, $\Fin_{*,\dagg}$-objects and correspondences among them. We will frequently use these, through the functoriality of $\ShvSp_*^!$, to present symmetric monoidal structures on certain microlocal sheaf categories and (right-lax) homomorphisms among them. Note that in this setting, tensor product in a sheaf category should be taken as $\overset{!}{\otimes}$.  On the other hand, (almost) all the sheaf categories that we are interested in will be some full subcategories with singular support condition in conic Lagrangians, restricted to which the functor $q_*p^!$ are continuous, thus (almost) all the categories and functors involved will be lying in  $\bPrstL$.

\section{Proof of the main theorem}\label{sec: proof}
First, let us give a more precise statement of the main theorem and make some remarks. 

\begin{thm}\label{thm: sec proof}
For any smooth immersed Lagrangian $L$ in $T^*X$, the classifying map for the sheaf of microlocal categories $\mu\cShv_L$ over a ring spectrum $\bk$ (commutative or at least $\bE_2$) is homotopic to the composition
\begin{align}\label{eq: thm sec proof}
L\overset{\vartheta\circ\gamma}{\longrightarrow} U/O\overset{BJ}{\longrightarrow} B\Pic(\bS)\longrightarrow B\Pic(\bk),
\end{align}
where $\gamma$ is the stable Gauss map, $\vartheta$ is the canonical involution on $U/O$, and $BJ$ is the the delooping of the $J$-homomorphism. 
\end{thm}

\begin{remark}
There is a twisted version of Theorem \ref{thm: sec proof}. Namely, one can consider the category of twisted sheaves on $X$, determined (up to homotopy) by a map $tw: X\rightarrow B\Pic(\bk)$, and the induced sheaf of microlocal categories $\mu\cShv_L^{tw}$ on $L$. By functoriality, the resulting classifying map for $\mu\cShv_L^{tw}$ is just the twist of (\ref{eq: thm sec proof}) by the composition $L\rightarrow X\overset{tw}{\rightarrow}B\Pic(\bk)$. This is analogous to the dependence on the choice of background class in the definition of Fukaya categories (cf. \cite{Seidel}). 
\end{remark}

Now let us start the proof of the theorem. For any affine Lagrangian $\ell\subset T^*\bR^{N}$,  
we fix a Legendrian lifting of $\ell$ in $T^*\bR^{N}\times \bR$, denoted by $\Lambda_\ell$.  
For any quadratic form $Q$ on $T^*\bR^N$ depending only on $\bp$,  such that $\varphi_{Q}(\FT(\ell))$ is graph like, it determines a Morse transformation for $\Lambda_{\ell}$ given by the negative conormal bundle (negative in $dt_1$) of the smooth hypersurface 
\begin{align*}
-t_1+t_0+Q[\bq_0]+\bq_0\cdot\bq_1=0
\end{align*}
in $\bR^{N}_{\bq_0}\times \bR^{N}_{\bq_1}\times \bR_{t_0}\times \bR_{t_1}$.  Let $\bL^N_{01, Q}$ denote for the corresponding Morse transformation.

\subsection{A commutative monoid $G_{Q_\flat,\ell}$ arising from a Morse transformation $\bL_{01,Q_\flat}$}

For any (stable) affine Lagrangian $\ell$, let $\Spectr(\ell)$ be the spectral decomposition of $\ell$ (after a parallel shifting to be a linear Lagrangian), i.e. the sum of that of $A_\ell$ and the eigenspace of $\infty$, and let $\Spectr^-(\ell)$ (resp. $\Spectr^+(\ell)$, $\Spectr^\infty(\ell)$, $\Spectr^I(\ell)$) be the negative (resp. positive, $\infty$, $I\subset \bR\cup\{\infty\}$) part of the spectral decomposition. The same notation (except for $\Spectr^\infty$) applies to (usual) quadratic forms, i.e. for a quadratic form $Q$, $\Spectr^-(Q)$ (resp. $\Spectr^{+}(Q)$) denotes for its negative (resp. positive) spectral part. 
For a \emph{nonnegative} quadratic form $Q$, we also call the underlying vector space of $\Spectr^+(Q)$ the \emph{support} of $Q$, and denote it by $\Supp(Q)$. For the corresponding underlying vector space of a spectral part, we replace $\Spectr$ in the notation by $\underline{\Spectr}$. 

We fix the following notions. We say a quadratic form $Q_\flat$ is negatively (resp. nonpositively, etc.) \emph{stable} on $\varinjlim\limits_N\bR^N\cong \varinjlim\limits_N(\bR^N)^*$ (under the Euclidean metric), if there exists $-\epsilon$ with $\epsilon>0$(resp. $\epsilon\geq 0$, etc.) such that for $N\gg 0$, 
\begin{align*}
Q_\flat|_{\bR^{N+k}}=Q_\flat|_{\bR^N}\oplus (-\epsilon)I_{(\bR^N)^{\perp}}, \forall k\geq 0.
\end{align*} 
For any \emph{negatively} stable $Q_\flat$ and any affine Lagrangian $\ell$ such that $\varphi_{Q_\flat}\circ\FT(\ell)$ is graph like, we assign the stable Morse transformation $\bL_{01,Q_\flat}$ for $\Lambda_\ell$,  whose underlying symplectomorphism is $\varphi_{Q_\flat}\circ\FT$. Here for any quadratic function $Q$ on $\varinjlim\limits_N\bR^N\cong\varinjlim\limits_N(\bR^N)^*$, 
\begin{align*}
\varphi_{Q}=\text{ the time-1 map of the Hamiltonian flow of }Q[\bp]. 
\end{align*}
The following constructions can be thought of as a more formal treatment of the relation between Morse transformations and paths in $U/O$, as discussed in Subsection \ref{subsec: Morse transf}.

For any stable Morse transformation $\bL_{01,Q_\flat}$ for $\Lambda_\ell$, let 
\begin{align*}
G_{Q_\flat,\ell; N}^{\lng 0\rng}=pt,&\\
G_{Q_\flat,\ell; N}^{\lng n\rng}=:\{ &(Q_i)_{i\in\lng n\rng^\circ}: \text{ each }Q_i\text{ is a nonnegative quadratic form supported on $\bR^N$};\\
&\rank (\Spectr^-(
\varphi_{\sum\limits_{i\in \lng n\rng^\circ}Q_i}\circ\varphi_{Q_\flat}(\FT(\ell))))=\sum\limits_{i\in\lng n\rng^\circ}\rank(Q_i)+\rank(\Spectr^-(\varphi_{Q_\flat}(\FT(\ell)))\},
\end{align*}
and
\begin{align*}
&VG_{Q_\flat,\ell;N}^{\lng 0\rng}=pt,\\
&VG_{Q_\flat,\ell;N}^{\lng n\rng}=\{(Q_i, \bp_i)_{i\in\lng n\rng^\circ}: \bp_i\in \Supp(Q_i)\}\subset G_{Q_\flat;N}^{\lng n\rng}\times \bR^N.
\end{align*}
Note that for any nonnegative quadratic form $Q_1$, we have 
\begin{align}\label{ineq: varphi_Q}
\rank(\Spectr^-(\varphi_{Q_1}\circ\varphi_{Q_\flat}(\FT(\ell))))\leq \rank(Q_1)+\rank(\Spectr^-(\varphi_{Q_\flat}(\FT(\ell)))). 
\end{align}
This in particular implies that for any $(Q_i)_{i\in\lng n\rng^\circ}\in G_{Q_\flat,\ell; N}^{\lng n\rng}$, we have 
\begin{align}\label{eq: rank equality Q_-}
\rank(\Spectr^-(\varphi_{\sum\limits_{i\in S}Q_i}\circ\varphi_{Q_\flat}(\FT(\ell))))=\sum\limits_{i\in S}\rank(Q_i)+\rank(\Spectr^-(\varphi_{Q_\flat}(\FT(\ell)))), 
 \end{align}
for any subset $S\subset \lng n\rng^\circ $. So we can organize $G_{Q_\flat,\ell; N}^{\lng n\rng},\ n \in \bZ_{\geq 0}$ (resp. $VG_{Q_\flat,\ell; N}^{\lng n\rng},\ n\in \bZ_{\geq 0}$) into a $\Fin_*$-object in the way that for any $f: \lng n\rng\rightarrow \lng m\rng$, we associate 
\begin{align*}
G_{Q_\flat,\ell; N}^{\lng n\rng}&\longrightarrow G_{Q_\flat,\ell; N}^{\lng m\rng}\\
(Q_j)_{j\in\lng n\rng^\circ}&\mapsto  (\sum\limits_{j\in f^{-1}(i)}Q_j)_{i\in\lng m\rng^\circ}
\end{align*}
and similarly for $VG_{Q_\flat,\ell; N}^{\lng n\rng},\ n\in \bZ_{\geq 0}$.

\begin{lemma}\label{lemma: decompose Q flag}
\item[(a)] For any nonnegative quadratic form $Q$ on $\bR^N$ of rank $k$ and for any partition $(k_1,\cdots, k_n)$ of $k$ of size $n$, the space of $n$-tuples of nonnegative quadratic forms $(Q_1,\cdots, Q_n)$ satisfying that 
\begin{align}\label{eq: lemma decompose Q flag}
\sum\limits_{i=1}^nQ_i=Q,\ \rank(Q_i)=k_i, i=1,\cdots,n,
\end{align}
is isomorphic to the partial flag variety $\Fl_{k_1,k_1+k_2,\cdots}(\Supp(Q))$. 

\item[(b)] For any active map $\lng n\rng\rightarrow \lng m\rng$, the map $G^{\lng n\rng}_{Q_\flat,\ell;N}\rightarrow G^{\lng m\rng}_{Q_\flat,\ell;N}$ is a proper fiber bundle over a subset of connected components of $G^{\lng m\rng}_{Q_\flat,\ell;N}$, and the fiber over any $(Q_i)\in G^{\lng m\rng}_{Q_\flat,\ell;N}$ in the image is diffeomorphic to the product of disjoint unions of some partial flag varieties of the underlying subspace $\Supp(Q_i)$.  
\end{lemma}
\begin{proof}
(a) Up to congruent relation, we may assume that $Q$ has eigenvalues 0 and 1, and is supported on $E\subset \bR^N$. Then it is easy to see that the decomposition in (\ref{eq: lemma decompose Q flag}) is the same as decomposing $E$ into mutually orthogonal subspaces $E_i$ with $\dim E_i=k_i, 1\leq i\leq n$. So the lemma follows. 

(b) follows directly from (a). 
\end{proof}

An immediate corollary of Lemma \ref{lemma: decompose Q flag} and \cite[Theorem 2.6]{Jin} is 
\begin{prop}
The $\Fin_*$-objects $G_{Q_\flat,\ell;N}^\bullet$ and  $VG_{Q_\flat,\ell;N}^\bullet$ represent commutative algebra objects in $\Corr(\Slch)_{\fib,\all}$. 
\end{prop}

Recall the following notions from \cite{Jin}:
\begin{align*}
&\lng n\rng_\dagger=\lng n\rng\sqcup\{\dagg\};
\end{align*}
a morphism 
\begin{align*}
f: \lng n\rng_\dagg\rightarrow \lng m\rng_\dagg
\end{align*} 
is a map satisfying $f(*)=*, f(\dagg)=\dagg$; these together defines an ordinary category $\Fin_{*,\dagg}$ with objects $\lng n\rng_\dagg, n\geq 0$ and morphisms as above. 

Now we endow $G_{Q_\flat,\ell;N}\times \bR^{N}_{\widetilde{\bq}_1}\times\bR_{\widetilde{t}_1}$, with the following module structure over $VG_{Q_\flat,\ell;N}$ through \cite[Theorem 2.21]{Jin}:
\begin{align}\label{eq: module G_{Q_-,ell;N}}
(G_{Q_\flat,\ell;N}\times \bR^{M+k}_{\widetilde{\bq}_1}\times\bR_{\widetilde{t}_1})^{\lng n\rng_\dagg}=\{&(Q_1,\bp_1,\cdots, Q_n,\bp_n; Q_\dagg, \widetilde{\bq}_1,\widetilde{t}_1): (Q_\bullet)_{\bullet\in \lng n \rng^\circ_\dagg}\in G_{Q_\flat,\ell; N}^{\lng n+1\rng},\\
&\bp_i\in \Supp(Q_i), i\in \lng n\rng^\circ\},
\end{align}
and for any $f: \lng n\rng_\dagg\rightarrow \lng m\rng_\dagg$, 
\begin{align}\label{eq: module G_univ R M}
\nonumber(G_{Q_\flat,\ell;N}\times \bR^{M+k}_{\widetilde{\bq}_1}\times\bR_{\widetilde{t}_1})^{\lng n\rng_\dagg}\longrightarrow &(G_{Q_\flat,\ell;N}\times \bR^{M+k}_{\widetilde{\bq}_1}\times\bR_{\widetilde{t}_1})^{\lng m\rng_\dagg}\\
((Q_1,\bp_1,\cdots, Q_n,\bp_n; Q_\dagg, \widetilde{\bq}_1,\widetilde{t}_1)\mapsto &((\sum\limits_{j\in f^{-1}(i)}Q_j, \sum\limits_{j\in f^{-1}(i)}\bp_j)_{i\in\lng m\rng^\circ}; Q_\dagg+\sum\limits_{j\in f^{-1}(\dagg)\backslash\{\dagg\}}Q_j, \widetilde{\bq}_1-2\sum\limits_{j\in f^{-1}(\dagg)\backslash\{\dagg\}}\bp_j, \\
\nonumber&\widetilde{t}_1-\sum\limits_{j\in f^{-1}(\dagg)}Q_j[\bp_j'],
\end{align}
where $\bp_j'$ is any vector satisfying that $Q_j\bp_j'=\bp_j$. 
By the assumption on $Q_j, j\in\lng n\rng^\circ$, there exists a vector $\bp'$ satisfying 
$Q_j\bp'=\bp_j, j\in f^{-1}(\dagg)\backslash\{\dagg\}$, hence 
we have $\sum\limits_{j\in f^{-1}(\dagg)}Q_j[\bp_j']=(\sum\limits_{j\in f^{-1}(\dagg)}Q_j)[\bp']$. This shows that 
the module structure (\ref{eq: module G_univ R M}) is well defined.

For any $G_N\subset G_{Q_\flat,\ell;N}$ (note that $G_N$ here is not specific to the particular model of $\bigsqcup\limits_{n\leq N}BO(n)$ as in \cite{Jin}), 
if for any inert map $\lng n\rng\rightarrow \lng k\rng$ in $\Fin_*$, the induced map
\begin{align}\label{eq: cond G_N}
G_N\underset{G_{Q_\flat,\ell;N}}{\times} G_{Q_\flat,\ell;N}^{\lng n\rng}\rightarrow G_{Q_\flat,\ell;N}^{\lng k\rng}\text{ has image in }G_N\underset{G_{Q_\flat,\ell;N}}{\times} G_{Q_\flat,\ell;N}^{\lng k\rng}, 
\end{align}
where the map $G_{Q_\flat,\ell;N}^{\lng n\rng}\rightarrow G_{Q_\flat,\ell;N}$ is through the unique active map $\lng n\rng\rightarrow\lng 1\rng$
then 
we can define the $\Fin_*$-object 
\begin{align}\label{eq: G_N requirement}
G_N^{\bullet}=G_N\underset{G_{Q_\flat,\ell;N}}{\times} G_{Q_\flat,\ell;N}^{\bullet}, 
\end{align}
which  is automatically a commutative algebra object in $\Corr(\Slch)_{\fib,\all}$, equipped with an algebra homomorphism $G_N\rightarrow G_{Q_\flat,\ell;N}$. Similarly, the $\Fin_*$-object $VG_N^\bullet=G_N\underset{G_{Q_\flat,\ell;N}}\times VG_{Q_\flat,\ell;N}^\bullet$ is a commutative algebra object in $\Corr(\Slch)_{\fib,\all}$, and we can uniquely associate a module structure on $G_N\times \bR_{\widetilde{\bq}_1}^{M+k}\times \bR_{\widetilde{t}_1}$ over $VG_N$ as above in $\Corr(\Slch)$. 

If we choose $Q_\flat$ to be a negatively stable quadratic Hamiltonian, then we let
\begin{align*}
&G_{Q_\flat,\ell}^\bullet=:\bigcup\limits_{N}G^\bullet_{Q_\flat,\ell;N},\ VG_{Q_\flat,\ell}^\bullet=:\bigcup\limits_{N}VG^\bullet_{Q_\flat,\ell;N}.
\end{align*}
Though these no longer give objects in $\Corr(\Slch)$ (they are like ind-objects), a standard way to view and study them is through the truncations at finite levels $N$.

\subsection{The fibrant simplicial category $\QHam(U/O)$}
\begin{definition}\label{def: QHam(U/O)}
Consider the following fibrant simplicial category, denoted by $\QHam(U/O)$:
\begin{itemize}
\item[(i)] The objects are $(\cU, Q_{\flat}, G^\bullet, M^{\bullet,\dagg})$, where $\cU\overset{\text{open}}{\subset }U/O$, $Q_\flat$ is a negatively stable quadratic Hamiltonian
 depending only on $\bp$, and 
 \begin{align*}
 (G^\bullet=\bigcup\limits_{N\geq 0} G_N^\bullet, M^{\bullet,\dagg}=\bigcup\limits_{N\geq 0} M_N^{\bullet,\dagg})
 \end{align*}
 is a pair of $\Fin_*$ and $\Fin_{*,\dagg}$ objects in the ordinary category of paracompact Hausdorff spaces, with a filtration given by  $\Fin_*$ and $\Fin_{*,\dagg}$ objects $(G_N^\bullet, M_N^{\bullet, \dagg})$ in $\Slch$, that determine $\Fun(\bZ^{op}_{\geq 0}, \Mod^{N(\Fin_*)}(\Corr(\Slch)))$. 
 
They need to satisfy the following conditions.\\

\noindent ia) The Hamiltonian map $\varphi_{Q_\flat}$ takes the Fourier transform of all elements in $\cU$ to graph type linear Lagrangians.\\

\noindent ib) We have $G^\bullet\subset \bigcap\limits_{\ell\in\cU}G_{Q_\flat,\ell}^\bullet$, $G_N^\bullet=G^\bullet\cap  \bigcap\limits_{\ell\in\cU}G_{Q_\flat,\ell;N}^\bullet$ and the diagram
\begin{align*}
\xymatrix{G^{\lng n\rng}\ar[r]\ar[d]&G^{\lng n\rng}_{Q_\flat,\ell}\ar[d]\\
G^{\lng 1\rng}\ar[r]&G^{\lng 1\rng}_{{Q_\flat},\ell}
}
\end{align*}
is Cartesian for every $n$ and $\ell\in \cU$. Equivalently, this means the embeddings $G_N^\bullet\subset G_{{Q_\flat},\ell;N}^\bullet$ induce commutative algebra homomorphisms from the latter to the former in $\CAlg(\Corr(\Slch))$. 
We require that $G\hookrightarrow G_{Q_\flat,\ell}$ is a homotopy equivalence and the positive eigenvalues of elements in $G$ have a uniform finite upper bound, and the above implies that
\begin{align*}
\prod\limits_{j\in\lng n\rng^\circ} G^{\{j,*\}}\overset{h.e.}{\simeq }G^{\lng n\rng}\overset{h.e.}{\simeq} G^{\lng n\rng}_{{Q_\flat},\ell}\overset{h.e.}{\simeq} \prod\limits_{j\in\lng n\rng^\circ} G^{\{j,*\}}_{{Q_\flat},\ell}.
\end{align*}

\noindent ic) For any $\ell\in \cU$, let $V^-_{\ell, Q_\flat}=\Spectr^-(\varphi_{Q_\flat}(\FT(\ell)))$. We require that for any subspace $E\subset V^-_{\ell, Q_\flat} $ the space of subspaces that can be realized as graphs of linear functions from $E$ to $E^\perp$ (in $\varinjlim\limits_N\bR^N$) and that intersect trivially with the support of every element in $G$ is contractible.\\

\noindent id) $M^{\bullet, \dagg}$ is a free module of $G^\bullet$ generated by $Q_\flat$, i.e. $M^{\lng n\rng_\dagg}=\{((Q_j)_{j\in\lng n\rng^0}, Q_\dagg+Q_\flat): (Q_\bullet)_{\bullet\in {\lng n\rng}_\dagg^\circ}\in G^{\lng n+1\rng}\}$. 
\begin{align*}
&\dim\Spectr^-(\varphi_{Q_\dagg+Q_\flat}(\FT(\ell))))=\dim\Spectr^-(\varphi_{Q_\flat}(\FT(\ell)))+\rank(Q_\dagg).
\end{align*}
holds for all $\ell\in \cU$. 
This implies that $\varphi_{Q_\dagg+Q_\flat}(\FT(\ell))$ is of graph type, and with (ib), this implies that for any $((Q_j)_{j\in\lng n\rng^0}, Q_\dagg+Q_\flat)\in M^{\lng n\rng_\dagg}$, the rank equality 
\begin{align}\label{eq: rank, QHam}
&\dim\Spectr^-(\varphi_{\sum\limits_{\bullet\in\lng n\rng^0_\dagg}Q_\bullet}(\varphi_{Q_\flat}(\FT(\ell))))=\\
&\dim\Spectr^-(\varphi_{Q_\flat}(\FT(\ell)))+\sum\limits_{\bullet\in\lng n\rng_\dagg^\circ}\rank(Q_\bullet)
\end{align}
holds for all $\ell\in \cU$. Moreover, we require that the spectral decomposition of $\FT(\varphi_{Q_\dagg+Q_\flat}(\FT(\ell)))$ to have eigenvalues away from a fixed open interval containing $0$, for all $\ell\in \cU$ and $Q_\dagg+Q_\flat\in M$. \\

\item[(ii)] A morphism $\fj_{Q_{12}}: (\cU_1,Q_\flat^{(1)}, G_1^\bullet, M_1^{\bullet,\dagg})\rightarrow (\cU_2,Q_\flat^{(2)}, G_2^\bullet, M_2^{\bullet,\dagg})$, depending on a nonnegative quadratic form $Q_{12}\in G_2$, consists of the following data.\\

\noindent iia) $\cU_2\subset \cU_1$,  $G_1^\bullet$ is included in $G_2^\bullet$, and $\fj_{Q_{12}}$ determines a morphism $(G_1^\bullet, M_1^{\bullet,\dagg})\rightarrow (G_2^\bullet, M_2^{\bullet,\dagg})$ in $\Mod^{N(\Fin_*)}(\Corr(\Slch))$, sending $Q_\flat^{(1)}$ to $Q_\flat^{(2)}+Q_{12}$.\\

\noindent  iib) we have 
\begin{align*}
\widetilde{Q}=Q_\flat^{(2)}+Q_{12}-Q_\flat^{(1)},
\end{align*}
is nonnegative and it satisfies 
\begin{align}\label{eq: no rank contribution}
&\dim\Spectr^-(\varphi_{\widetilde{Q}+Q_\flat^{(1)}}(\FT(\ell))))=\dim \Spectr^-(\varphi_{Q_\flat^{(1)}}(\FT(\ell))),
\end{align}
for any $\ell\in \cU_2$.  Note that the nonnegativity of $\widetilde{Q}$ implies that it has \emph{no contribution} to any rank equality, i.e. 
\begin{align*}
&\dim\Spectr^-(\varphi_{\widetilde{Q}}(\varphi_{Q_\flat^{(1)}+Q_1}(\FT(\ell))))=\dim \Spectr^-(\varphi_{Q_\flat^{(1)}+Q_1}(\FT(\ell)))\\
=&\dim \Spectr^-(\varphi_{Q_\flat^{(1)}}(\FT(\ell)))+\rank(Q_1),\ \forall Q_1\in G_1,\ \ell\in \cU_2. 
\end{align*}
From this, we see that $\fj_{Q_{12}}$ (uniquely) exists if and only if $G_2\supset G_1\cup (G_1+Q_{12})$, where $G+Q_{12}=\{Q_1+Q_{12}: Q_1\in G\}$. \\

\item[(iii)]  For any two objects $(\cU_1,Q_\flat^{(1)}, G_1^\bullet, M_1^{\bullet, \dagg})$ and $(\cU_2,Q_\flat^{(2)}, G_2^\bullet, M_2^{\bullet, \dagg})$, we take the space of $\fj_{Q_{12}}$ endowed with the induced topology from the space of quadratic forms in $\bp$.  The composition of morphisms are defined in the most natural way: for any $\fj_{Q_{i,i+1}}: (\cU_i,Q_\flat^{(i)}, G_i^\bullet, M_i^{\bullet,\dagg})\rightarrow (\cU_{i+1},Q_\flat^{({i+1})}, G_{i+1}^\bullet, M_{i+1}^{\bullet,\dagg})$, $i=1,2$, the composition $\fj_{Q_{23}}\circ\fj_{Q_{12}}$ is equal to $\fj_{Q_{12}+Q_{23}}$. It is not hard to see that $\fj_{Q_{12}+Q_{23}}$ is well defined. Let $\widetilde{Q}_i=Q_\flat^{(i+1)}+Q_{i,i+1}-Q_\flat^{(i)}$. Then $\widetilde{Q}_1+\widetilde{Q}_2$ satisfies condition (\ref{eq: no rank contribution}). Indeed, 
\begin{align*}
&\dim\Spectr^-\varphi_{\widetilde{Q}_1+\widetilde{Q}_2+Q_\flat^{(1)}}(\FT(\ell)))=\dim \Spectr^-(\varphi_{Q_\flat^{(2)}+Q_{12}+\widetilde{Q}_2}(\FT(\ell)))\\
=&\dim \Spectr^-(\varphi_{Q_\flat^{(2)}+Q_{12}}(\FT(\ell)))=\dim\Spectr^-\varphi_{\widetilde{Q}_1+Q_\flat^{(1)}}(\FT(\ell)))= \dim\Spectr^-\varphi_{Q_\flat^{(1)}}(\FT(\ell))).
\end{align*}
Taking the singular simplicial complex of each mapping space defines the mapping simplicial sets and the composition rules for them. 
 \end{itemize}
 \end{definition}

In the following, unless otherwise specified, we will equally view $\QHam(U/O)$ as an $\infty$-category through the coherent nerve functor.

\begin{lemma}\label{lemma: QHam objects}
For any $\ell_0\in U/O$, there exists an object $(\cU,Q_\flat,G^\bullet, M^{\bullet,\dagg})$ in $\QHam(U/O)$ such that $\cU$ contains $\ell_0$. 
\end{lemma}
\begin{proof}
Given any $\ell_0\in U/O$, choose a negatively stable $Q_\flat$ such that $\varphi_{Q_\flat}(\FT(\ell_0))$ is of graph type and has no negative spectrum part. Up to congruent equivalences, we may assume that $\varphi_{Q_\flat}(\FT(\ell_0))$ has only eigenvalues $0$ and/or $1$. The Fourier transform of $\varphi_{Q_\flat}(\FT(\ell_0))$ has spectrum concentrated in $-1$ and $\infty$. Take an open neighborhood $\widetilde{\cU}$ of $\FT(\varphi_{Q_\flat}(\FT(\ell_0)))$ in $U/O$ by choosing (i) a small neighborhood $I_{-1}$ and $I_\infty$ of $-1$ and $\infty$ respectively on the spectrum circle in Figure \ref{Figure: spectrum circle}; (ii) an open neighborhood $\cV$ of $b_{\ell_0}^\perp$ in $\Gr(\dim b_{\ell_0}^\perp, \infty)$--then $\widetilde{\cU}$ consists of linear Lagrangians $\widetilde{\ell}$ with $\underline{\Spectr}^{I_\infty} \widetilde{\ell}\in \cV$ and $\Spectr(\widetilde{\ell})=\Spectr^{I_{-1}}(\widetilde{\ell})\oplus \Spectr^{I_\infty}(\widetilde{\ell})$.  Let $\cU=\FT\circ\varphi_{-Q_\flat}\circ \FT(\widetilde{\cU})$.
\begin{figure}[h]
\begin{tikzpicture}
\draw (0,0) circle [radius=1.5];
\filldraw (1.5,0) circle (1pt) node [right] {$\infty$};
\draw (1.5,0) node [below] {${\substack{\\ \\ \uparrow\\ b_{\ell_0}^\perp}}$};
\draw[cyan] (1.5,0) circle (6pt);
\filldraw ({1.5* cos(0.75*pi r) }, {1.5* sin(0.75*pi r) }) circle (1pt) node [left] {$-1$};
\draw ({1.5* cos(0.75*pi r) }, {1.5* sin(0.75*pi r) }) node [above] {${\substack{b_{\ell_0}\\ \downarrow\\ \\ }}$};
\draw[cyan] ({1.5* cos(0.75*pi r) }, {1.5* sin(0.75*pi r) }) circle (6pt); 
 \filldraw ({1.5* cos(-0.75*pi r) }, {1.5* sin(-0.75*pi r) }) circle (1pt) node [left] {$1$};
 \filldraw({-1.5}, {0}) circle (1pt) node[left] {$0$};
\end{tikzpicture}
\caption{The spectrum of $\FT(\varphi_{Q_\flat}(\FT(\ell_0)))$}\label{Figure: spectrum circle}
\end{figure}

For any fixed $K>0$,  $\epsilon>0$, let 
\begin{align}\label{eq: G_{K,epsilon, R, b_ell}}
G_{K, \epsilon, R, b_{\ell_0}}=\{&Q_1\in G_{\varphi_{Q_\flat},\ell_0}: \forall v\in \Spectr^+(Q_1),\ |\proj_{b_{\ell_0}^\perp} v|\leq K |\proj_{b_{\ell_0}} v|;\\
 \nonumber&Q_1|_{b_{\ell_0}}\geq (1/2+\epsilon)I_{\proj_{b_{\ell_0}}(\Supp(Q_1))}, Q_1\leq R\cdot I\}.
\end{align} 
Since $G_{K, \epsilon, R, b_{\ell_0}}$ satisfies (\ref{eq: cond G_N}), it uniquely determines a commutative algebra (ind-)object in $\Corr(\Slch)$. Let $M_{K,\epsilon, R, b_{\ell_0}}^{\bullet,\dagg}$ be the free module of $G_{K, \epsilon, R, b_{\ell_0}}^\bullet$  generated by the element $Q_\flat$. 
We show that for appropriate choices of $I_{-1}, I_\infty$ and $\cV$, on the above defined $
\cU$ containing $\ell_0$, $\varphi_{Q_1+Q_\flat}(\FT(\ell)))$ satisfies the rank equality
\begin{align*}
\dim \Spectr^-(\varphi_{Q_1+Q_\flat}(\FT(\ell))))=\dim\Spectr^-(\varphi_{Q_\flat}(\FT(\ell)))+\rank(Q_1). 
\end{align*}
 for all $Q_1\in G_{K,\epsilon,R, b_{\ell_0}}$ and $\ell\in \cU$. Moreover, the eigenvalues of $\FT(\varphi_{Q_1+Q_\flat}(\FT(\ell)))$ are away from a fixed open interval containing $0$.

 For any $\ell\in \cU$, we have
\begin{align*}
\Spectr (\FT(\varphi_{Q_1+Q_\flat}(\FT(\ell)))))=\Spectr(\widetilde{\ell}+2Q_1),
\end{align*}
 where $\widetilde{\ell}=\FT\circ \varphi_{Q_\flat}\circ\FT(\ell)$ and $\widetilde{\ell}+2Q_1$ means the sum as generalized quadratic forms. 
 We just need to show that for sufficiently small $I_{-1}$, $I_\infty$ and $\cV$, we can make 
\begin{align}
\label{eq: proof spectr ell, Q_dagg, -}&\dim \Spectr^+(\widetilde{\ell}+2Q_1)\geq \dim \Spectr^+(\widetilde{\ell}) +\rank(Q_1),\\
\label{eq: proof spectr ell, Q_dagg, +}&\dim \Spectr^-(\widetilde{\ell}+2Q_1)\geq \dim \Spectr^-(\widetilde{\ell})-\rank(Q_1). 
\end{align}
Note that consequently, the above inequalities are equalities. Consider the subspace $\underline{\Spectr}^+(\widetilde{\ell})\oplus \proj_{\underline{\Spectr}^{I_{-1}}(\widetilde{\ell})}(\underline{\Spectr}^+(Q_1))$. By shrinking $\cV$ and $I_{-1}$ if necessary, we can find $\delta>0$ such that 
the restriction of $Q_1\in G_{K,\epsilon,R, b_{\ell_0}}$ to $ \proj_{\underline{\Spectr}^{I_{-1}}(\widetilde{\ell})}(\underline{\Spectr}^+(Q_1)),\ \widetilde{\ell}\in \widetilde{\cU}$, has eigenvalues all greater than $1/2+\delta$, and then $\widetilde{\ell}+2Q_1|_{\proj_{\underline{\Spectr}^{I_{-1}}(\widetilde{\ell})}(\underline{\Spectr}^+(Q_1))}$ has eigenvalues greater than $2\delta$. Since for any $u\in \underline{\Spectr}^{+}(\widetilde{\ell})$ and $v\in \proj_{\underline{\Spectr}^{I_{-1}}(\widetilde{\ell})}(\underline{\Spectr}^+(Q_1))$, $\widetilde{\ell}(u,v)=0$ and $|Q_1(u,v)|<R|u||v|$, we have  
\begin{align*}
(\widetilde{\ell}+2Q_1)(u+v, u+v)=&(\widetilde{\ell}+2Q_1)(u,u)+(\widetilde{\ell}+2Q_1)(v,v)+2(\widetilde{\ell}+2Q_1)(u,v)\\
\geq &R'|u|^2+2\delta|v|^2-4R|u||v|>0,
\end{align*}
for sufficiently large $R'$ achievable by shrinking $\cV$, $I_1$ and $I_\infty$. This shows that $(\widetilde{\ell}+2Q_1)$ is postive definite on $\underline{\Spectr}^{+}(\widetilde{\ell})\oplus \proj_{\underline{\Spectr}^{I_{-1}}(\widetilde{\ell})}(\underline{\Spectr}^+(Q_1))$, and this confirms (\ref{eq: proof spectr ell, Q_dagg, -}). In a similar way, we can show that $\widetilde{\ell}+2Q_1$ is negative definite on $\underline{\Spectr}^{(I_\infty\backslash\{\infty\})^-}(\widetilde{\ell})\oplus (\underline{\Spectr}^{I_{-1}}(\widetilde{\ell})\cap \underline{\Spectr}^+(Q_1|_{\underline{\Spectr}^{I_{-1}}(\widetilde{\ell})})^{\perp})$, 
and so (\ref{eq: proof spectr ell, Q_dagg, +}) is established as well. Along the way, we also see that the eigenvalues of $\FT(\varphi_{Q_1+Q_\flat}(\FT(\ell)))$ are away from the open interval $(-2\delta, 2\delta)$. The above confirms that $(\cU,Q_\flat, G^\bullet_{K,\epsilon, R, b_{\ell_0}}, M_{K, \epsilon, R, b_{\ell_0}}^{\bullet, \dagg})$ satisfies ia), ib) and id).

\begin{figure}[h]
\begin{tikzpicture}
\draw[thick] (-6,0)--(-2,0);
\draw[thick] (0,0)--(6,0) node [below] {$b_{\ell_0}$};
\draw[thick, dashed] (-2,0)--(0,0);
\draw[thick] (-2, 2)--(-2,-4)--(2,-2)--(2,4)--(-2,2);
\draw (1.7, 3.4) node  {$b_{\ell_0}^\perp$}; 
 \draw (6,0) ellipse (20pt and 30pt);
  \draw[dashed] (-6, 1.08)--(0,0);
 \draw (0,0)--(6, -1.08);
 \draw (-6, -1.08)--(6, 1.08);
  \draw(-6,1.08)  arc[x radius =20pt , y radius = 30pt, start angle= 90, end angle= 270]; 
   \draw[dashed](-6,-1.08)  arc[x radius =20pt , y radius = 30pt, start angle= 270, end angle= 450]; 
   \draw (-6,-1.08)  arc[x radius =20pt , y radius = 30pt, start angle= 270, end angle= 340];
 \draw[blue, thick] (0,0)--(3, 3) node[right] {$E$};
 \draw[blue, thick, dashed] (0,0)--(-2,-2);
 \draw[blue, thick] (-2,-2)--(-3,-3);
 \draw[blue, dashed] (3,3)--(1.8,3);
 \draw[blue] (5.3, 0.3)--(6.5, -0.8);
 \draw[blue](-6.5, 0.8)--(-5.3, -0.3);
 \draw[blue]  (5.3, 0.3)--(0,0);
 \draw[blue, dashed] (0,0)--(-5.3, -0.3);
 \draw[blue] (-6.5, 0.8)--(0,0);
 \draw[blue, dashed] (0,0)--(6.5, -0.8);
 \draw[blue] (5.3-1.2, 0.3+1.1)--(6.5+1.2, -0.8-1.1);
  \draw[blue](-6.5-1.2, 0.8+1.1)--(-5.3+1.2, -0.3-1.1);
  \draw[blue ](-6.5-1.2, 0.8+1.1)-- (5.3-1.2, 0.3+1.1);
  \draw[blue] (6.5+1.2, -0.8-1.1)--(0,-1.9+7.7/23.8);
  \draw[blue,dashed ](0,-1.9+7.7/23.8)--(-2,-1.9+9.7/23.8);
  \draw[blue] (-2,-1.9+9.7/23.8)--(-5.3+1.2, -0.3-1.1);
  \draw[blue] (3,-1.2) node {$E^\perp$}; 
  \draw[green, thick] (6, -0.3)--(-6,0.3) node [left]{$\proj_{E^\perp} b_{\ell_0}$};
  \draw[dashed] (-6,0)--(-6,0.3);
  \draw[->, red, thick] (1,1)--(0.5,0.5);
  \draw[->,red, dashed, thick] (-1,-1)--(-0.5,-0.5);
  \draw[->, red, thick] (2.6,2.6)--(1.5,1.5);
    \draw[->, red,dashed,  thick] (-2,-2)--(-1.5,-1.5);
    \draw[ red,thick] (-2.6,-2.6)--(-2,-2);
    \draw[->, dashed, red, thick]   (5.3*0.8, 0.3*0.8)-- (6*0.8, -0.3*0.8);
    \draw[->, dashed, red, thick] (6.5*0.8, -0.8*0.8)--(6*0.8,-0.3*0.8); 
     \draw[->, red, dashed, thick]   (-5.3*0.8, -0.3*0.8)-- (-6*0.8, 0.3*0.8);
    \draw[->,  red, dashed,  thick] (-6.5*0.8, 0.8*0.8)--(-6*0.8,0.3*0.8); 
    \draw[->, red, thick] (0.7*0.8, -0.6*0.8)--(0,0);
    \draw[->,red, thick] (-0.7*0.8,0.6*0.8)--(0,0);
    \draw[->, red, thick] (-0.7*2.5,0.6*2.5)--(-0.7, 0.6); 
    \draw[->, red, thick] (0.7*2.5,-0.6*2.5)--(0.7, -0.6); 
     \draw[->, red, thick] (-0.7*2.5-6*0.8,0.6*2.5+0.3*0.8)--(-0.7-6*0.8, 0.6+0.3*0.8); 
   \draw[->, red, thick] (0.7*2.5-6*0.8,-0.6*2.5+0.3*0.8)--(0.7-6*0.8, -0.6+0.3*0.8); 
     \draw[->, red, thick] (-0.7*2.5+4.8,0.6*2.5-0.24)--(-0.7+4.8, 0.6-0.24); 
    \draw[->, red, thick] (0.7*2.5+4.8,-0.6*2.5-0.24)--(0.7+4.8, -0.6-0.24); 
    \end{tikzpicture}
\caption{The red vector field determines a 1-parameter subgroup $\phi_t$ in $GL(\bR^N)$, compatible with stabilization on $N$, that deformation retracts $B_K(b_{\ell_0},b_{\ell_0}^\perp)\cup E^\perp$ onto $E^\perp$. This induces a homotopy equivalence between $\Gr(\dim E;\rel\ G_{K,\epsilon,R,  b_{\ell_0}})$ and the open locus $\Hom(E, E^\perp)\subset \Gr(\dim E;\infty)$. }\label{figure: ic}
\end{figure}
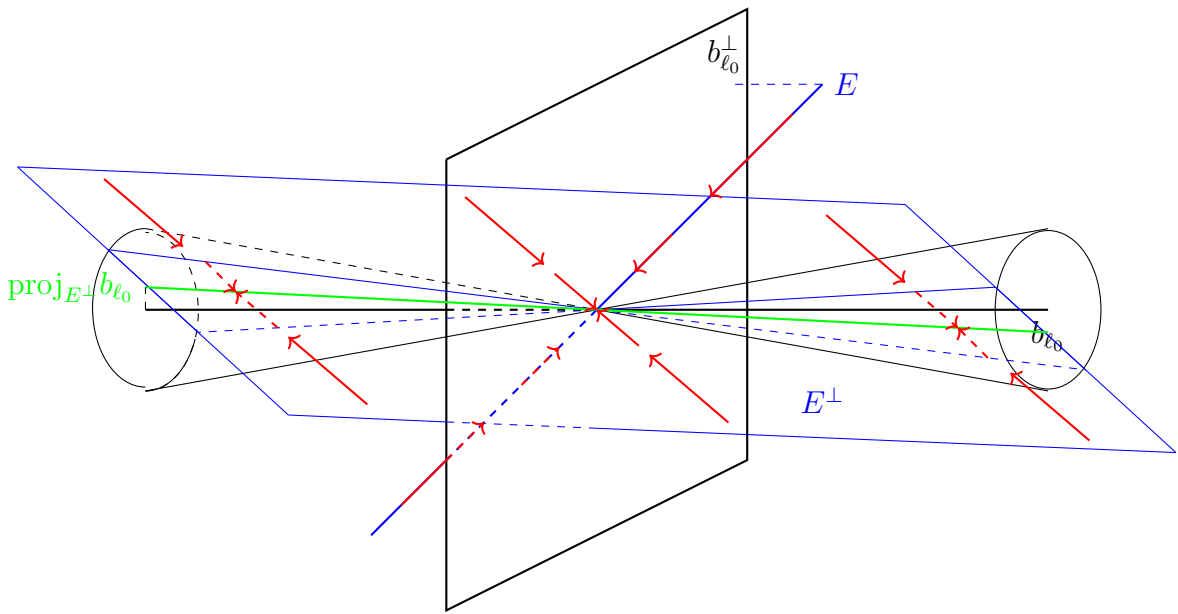

 Now we show that condition ic) holds for $(\cU, Q_\flat, G_{K, \epsilon,R, b_{\ell_0}}^\bullet, M_{K, \epsilon, R,b_{\ell_0}}^{\bullet, \dagg})$, provided that $\cU$ has been chosen very small. Let $B_K(b_{\ell_0}, b_{\ell_0}^\perp)$ be the space of vectors $v$ satisfying $|\proj_{b_{\ell_0}^\perp} v|\leq K|\proj_{b_{\ell_0}}v|$, i.e. it is the union of the vector subspaces in $G_{K,\epsilon, R, b_{\ell_0}}$.  Given any finite dimensional vector subspace $E$ very close to $b_{\ell_0}^\perp$, at any finite level (i.e. $\dim b_{\ell_0}<\infty$), we have the following natural diffeomorphisms
 \begin{align*}
 &B_K(b_{\ell_0},b_{\ell_0}^\perp)\cong \Cone(S^{\dim b_{\ell_0}-1}\times D_K^{\dim b_{\ell_0}^\perp})\\
 &E^\perp\cap B_K(b_{\ell_0},b_{\ell_0}^\perp)\cong \Cone(S^{\dim b_{\ell_0}-1}\times D_{K'}^{\dim E^\perp-\dim b_{\ell_0}}),
 \end{align*}
 where one can identify $D_{K'}^{\dim E^\perp-\dim b_{\ell_0}}$ as the intersection of an affine subspace in $b_{\ell_0}^\perp$ with $D_K^{\dim b_{\ell_0}^\perp}$ (see the illustration in Figure \ref{figure: ic}). Now to show that ic) holds, we just need to show that there is a homotopy equivalence between the space of vector subspaces of the same dimension as $E$ that intersect every element in $G_{K,\epsilon, R, b_{\ell_0}}$ trivially, denoted by $\Gr(\dim E;\rel\ G_{K,\epsilon,R,  b_{\ell_0}})$, and the open locus $\Hom(E, E^\perp)\subset \Gr(\dim E;\infty)$. This is obvious by defining a contracting vector field on $E\oplus (\proj_{E^\perp} b_{\ell_0})^{\perp E^\perp}$ that generates a one-parameter subgroup of $GL(\bR^N)$.

Therefore, for $\cU$ sufficiently small, we have $(\cU, Q_\flat, G_{K, \epsilon, R, b_{\ell_0}}^\bullet, M_{K,\epsilon, R, b_{\ell_0}}^{\bullet, \dagg})$ an object in $\QHam(U/O)$.     
\end{proof}

The above lemma shows that $\QHam(U/O)$ is nonempty and its projection to $\Open(U/O)^{op}$ form a basis for  the topology of $U/O$.

\begin{lemma}\label{lemma: negative Q_flat dominates}
For any $\ell_0\in U/O$ and any object $(\cU,Q_\flat, G^\bullet, M^{\bullet, \dagg})\in \QHam(U/O)$ with $\ell_0\in \cU$, there exists an object $(\cU_1, Q_\flat^{(1)}, G_1^\bullet, M_1^{\bullet, \dagg})$ and a morphism $\fj_{Q_{12}}$ from the former to the latter,  such that $\ell_0\in \cU_1$ and $\varphi_{Q_\flat^{(1)}}(\FT(\ell_0))$ has no negative spectral part. 
\end{lemma}
\begin{proof}
Using congruent equivalences, we may assume that $\FT(\varphi_{Q_\flat}(\FT(\ell_0)))$ has spectral decomposition concentrated in $\pm 1$ and $\infty$. Let $V_{1}$ (resp. $V_{-1}, V_\infty$) be the eigenspace of $1$ (resp. $-1$ and $\infty$) of $\FT(\varphi_{Q_\flat}(\FT(\ell_0)))$. Our first observation about $G$ is that for any $Q_1\in G$, the restriction of $Q_1|_{V_{-1}}$ must have positive eigenvalues bounded below by $1/2+\epsilon$ for some fixed $\epsilon>0$, followed from assumption id). 
Therefore, by the boundedness of the eigenvalues of elements in $G$, we must have the support of all $Q_1\in G$ contained in the region 
\begin{align*}
B_{K}(V_{-1}, V_{1}\oplus V_\infty)=\{v\in \bR^N: |\proj_{V_{1}\oplus V_\infty} v|\leq K |\proj_{V_{-1}} v|\}, 
\end{align*}
for some $K>0$ (shown as the blue cone in Figure \ref{figure: blue cone}). Now take $Q_{12}$ to be the nonnegative quadratic form supported on $V_{1}$ and with eigenvalue $1$, then $\varphi_{Q_\flat-Q_{12}}(\FT(\ell_0))$ is graph like and  has spectral decomposition concentrated at $1$ and $\infty$.

\begin{figure}[h]
\begin{tikzpicture}
\draw[thick] (-3,0)--(3,0) node[below] {$V_{-1}$};
\draw[thick] (0,-2.5)--(0,2.5) node[right] {$V_\infty$};
\draw[thick] (2,2) --(-2,-2) node[below] {$V_{1}$}; 
\draw[cyan] (3, -1.55) arc[x radius =30pt , y radius = 45pt, start angle= -90, end angle= 270];
\draw[cyan] (-3, -1.55) arc[x radius =30pt , y radius = 45pt, start angle= -90, end angle= 270];
\draw[cyan] (2.8, -1.55)--(-2.7, 1.55);
\draw[cyan] (-2.8, -1.55)--(2.7, 1.55);
\draw[red] (2.6, 1.54) arc[x radius=75pt, y radius= 29pt, start angle=-10, end angle=350];
\draw[red] (2.6, -1.84) arc[x radius=75pt, y radius= 29pt, start angle=-10, end angle=350];
\draw[red, dashed] (-1,0.75)--(1,-0.75);
\draw[red, dashed] (1, 2.65)--(-1, -2.65);
\draw[red] (-1,0.75)--(-1, -2.65);
\draw[red] (1,-0.75)--(1, 2.65);
\filldraw[red, opacity=0.2] (-1,0.75)--(0,0)--(-1, -2.65);
\filldraw[red, opacity=0.2] (1,-0.75)--(0,0)--(1, 2.65);
\end{tikzpicture}
\caption{The blue cone is $B_K(V_{-1}, V_{1}\oplus V_\infty)$, and the red cone, with the filled region showing one slice along $V_{1}\oplus V_\infty$, is $B_K(V_{-1}\oplus V_{1}, V_\infty)$.}\label{figure: blue cone}
\end{figure}
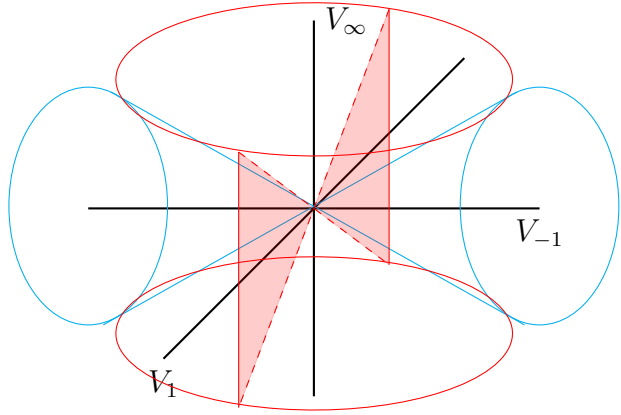

Consider the sum $G+Q_{12}=\{Q+Q_{12}: Q\in G\}$. Since the support of $Q_{12}$ intersects $\Supp(Q), Q\in G$ trivially, we have the rank equality holds $\rank(Q+Q_{12})=\rank(Q)+\rank(Q_{12})$ for all $Q\in G$. On the other hand, the forms in $G+Q_{12}$ have eigenvalues bounded above by some $R>0$, and their support are contained in the cone $B_{K}(V_{-1}\oplus V_1, V_\infty)$ (shown as the red cone in Figure \ref{figure: blue cone}). 

By the relation $b_{\ell_0}=V_1\oplus V_{-1}$, choose sufficiently small $\cU_1$ containing $\ell_0$, and let $(\cU_1, Q_\flat^{(1)}, G_1^\bullet, M_1^{\bullet,\dagg})$ be $(\cU_1, Q_\flat-Q_{12}, G^\bullet_{K,\epsilon, R, b_{\ell_0}}, M^{\bullet,\dagg}_{K,\epsilon,, R, b_{\ell_0}})$ as defined in the proof of Lemma \ref{lemma: QHam objects} (more specifically (\ref{eq: G_{K,epsilon, R, b_ell}})), then the morphism
\begin{align*}
\fj_{Q_{12}}: (\cU,Q_\flat, G^\bullet, M^{\bullet, \dagg})\rightarrow (\cU_1, Q_\flat^{(1)}, G_1^\bullet, M_1^{\bullet,\dagg})
\end{align*}
satisfies all the conditions in Definition \ref{def: QHam(U/O)} 
(ii). This completes the proof. 
  
\end{proof}

For any $L\subset T^*\bR^N$, let $\gamma_L: L\rightarrow U/O$ be the stable Gauss map. Let $\QHam_L(U/O)$ be the fibrant simplicial category, whose 
\begin{itemize}
\item[(i)] Objects are $(\cU\overset{\text{open}}{\subset}U/O, \cV\overset{\text{open}}{\subset} L, Q_\flat, G^\bullet, M^{\bullet,\dagg})$ subject to the condition that $\gamma_L(\cV)\subset \cU$ and $(\cU, Q_\flat, G^\bullet, M^{\bullet,\dagg})\in \QHam(U/O)$;

\item[(ii)] Mapping simplicial sets are defined by 
\begin{align}\label{eq: Maps, QHam_L}
&\Maps_{\QHam_L(U/O)}((\cU_1, \cV_1, Q_\flat^{(1)}, G_1^\bullet, M_1^{\bullet,\dagg}),  (\cU_2, \cV_2, Q_\flat^{(2)}, G_2^\bullet, M_2^{\bullet,\dagg}))\\
=&\begin{cases}&\Maps_{\QHam(U/O)}((\cU_1, Q_\flat^{(1)}, G_1^\bullet, M_1^{\bullet,\dagg}),  (\cU_2,Q_\flat^{(2)}, G_2^\bullet, M_2^{\bullet,\dagg})),\ \cV_1\supset \cV_2;\\
&\emptyset,\ \cV_1\not\supset \cV_2. 
\end{cases}
\end{align}
\end{itemize}
We also view $\QHam_L(U/O)$ as an $\infty$-category. 
There is an obvious functor $p_L: \QHam_L(U/O)^{op}\rightarrow \Open(L)$ by the projection to the $\cV$ factor. It is a Cartesian fibration: for any object $(\cU, \cV,Q_\flat, G^\bullet, M^{\bullet,\dagg})$ in $\QHam_L(U/O)^{op}$ and any morphism $\cV'\hookrightarrow \cV$ in $\Open(L)$, the Cartesian morphism over the inclusion is the unique morphism $(\cU, \cV',Q_\flat, G^\bullet, M^{\bullet, \dagg})\rightarrow (\cU, \cV,Q_\flat, G^\bullet, M^{\bullet,\dagg})$.

\subsection{A canonical functor $F_L: \QHam_L(U/O)^{op}\rightarrow \Fun(\Delta^1, \PrstL)$}\label{subsec: F_L}

We will construct a canonical functor $F_L: \QHam_L(U/O)^{op}\rightarrow \Fun(\Delta^1, \PrstL)$ through a set of correspondences. The construction will heavily rely on the main results in \cite[Section 2]{Jin} and  Appendix \ref{sec: Appendix}. We remark that in the following correspondences, the Hamiltonian flow for each quadratic form $Q$ is performed in the opposite direction compared to those in the previous sections, i.e. $\varphi_{Q}\leadsto \varphi_{-Q}$.  In particular, the loops/paths that they induce are the reversed ones.

\subsubsection{The first set of correspondences}\label{subsubsec: first corr}

For any $(\cU, \cV, Q_\flat^{(1)}, G^{\bullet}, M^{\bullet,\dagg})$ in $\QHam_L(U/O)$, let 
\begin{align}\label{eq: cH_M_N}
\cH_{M_N}=&\{(\bq_0,\bq_1,t_0,t_1, Q_\flat+Q_\dagg):  t_1-t_0+(Q_\flat+Q_\dagg)[\bq_0]+\bq_1\cdot \bq_0\geq 0\}\\
\nonumber&\subset \bR^N\times \bR^N\times \bR\times \bR\times M_N. 
\end{align}

Then $\cH_{M_N}$ is naturally endowed with a $VG_N$-module structure:
\begin{itemize}
\item 
\begin{align*}
\cH_{M_N}^{\lng n\rng_\dagg}=&\{((Q_i, \bp_i)_{i\in\lng n\rng^\circ}; (\bq_0,\bq_1,t_0,t_1, Q_\flat+Q_\dagg)): ((Q_i)_{i\in\lng n\rng^\circ}, Q_\flat+Q_\dagg)\in M_N^{\lng n\rng_\dagg}\}\\
&\subset VG_N^{\lng n\rng}\times \cH_{M_N}
\end{align*}

\item For any $f: \lng m\rng_\dagg\rightarrow \lng n\rng_\dagg$ in $N(\Fin_{*,\dagg})$, 
\begin{align*}
\cH_{M_N}^{\lng m\rng_\dagg}\longrightarrow& \cH_{M_N}^{\lng n\rng_\dagg}\\
((Q_i, \bp_i)_{i\in\lng m\rng^\circ}; (\bq_0,\bq_1,t_0,t_1, Q_\flat+Q_\dagg))\mapsto& ((\sum\limits_{i\in f^{-1}(j)}Q_i, \sum\limits_{i\in f^{-1}(j)}\bp_i)_{j\in\lng n\rng^\circ}; (\bq_0,\bq_1+2\sum\limits_{i\in f^{-1}(\dagg)\backslash\{\dagg\}}\bp_i,\\
&t_0,t_1+\sum\limits_{i\in f^{-1}(\dagg)\backslash\{\dagg\}}Q_i[\bp'_i], Q_\flat+Q_\dagg+\sum\limits_{i\in f^{-1}(\dagg)\backslash\{\dagg\}}Q_i)),
\end{align*}
where $\bp_i'$ is any solution solving $Q_i\bp_i'=\bp_i$. It is easy to check that the map is well defined. 

\item The morphism $\cH_{G_N, M_N}^{\bullet, \dagg}\rightarrow VG^\bullet\circ\pi_\dagg$ is the natural one induced from the projection of $\cH_{M_N}^{\lng n\rng_\dagg}$ to $VG_N^{\lng n\rng}$ for each $n$. 
\end{itemize}
It is straightforward to check that $(VG^\bullet, \cH_{M_N}^{\bullet,\dagg})$ satisfies the conditions in \cite[Theorem 2.21]{Jin}, so it is naturally an object in $\cMod^{N(\Fin_*)}(\Slch)$.

For any $\cY=(\cU,\cV,Q_\flat, G,M)\in \QHam_L(U/O)$, let $\cS_{\cY}$ be the space of \emph{stable} nonnegative quadratic form $\widetilde{Q}$ such that $\varphi_{Q_\flat+Q_\dagg+\widetilde{Q}}(\FT(\ell))$ is graph like and satisfies that
\begin{align*}
\rank(\Spectr^-(\varphi_{Q_\flat+Q_\dagg+\widetilde{Q}}(\FT(\ell))))=\rank(\Spectr^-(\varphi_{Q_\flat+Q_\dagg}(\FT(\ell))))
\end{align*}
for all $\ell\in \cU$ and $Q_\flat+Q_\dagg\in M$. It is clear that $\cS_{\cY}$ is contractible, for it admits a deformation retraction to the point $\{\widetilde{Q}=0\}$ through sending $\widetilde{Q}$ to $t\widetilde{Q}, 0\leq t\leq 1$. In the following and later subsections, the contractibility of $\cS_{\cY}$ will play a role in the definition of the canonical functor $F_L$ (\ref{eq: F_L, sketch}). 

For any $m$-simplex $\tau_m$ in $\cS_{\cY}$, which is a family $\widetilde{Q}_v$ parametrized by $v\in |\Delta^m|$, we have the following correspondence of pairs in $\Mod^{N(\Fin_*)}(\Corr(\Slch))$ for any $N\gg 0$
\begin{align}\label{diagram: corr base}
\xymatrix{&(VG_N, \cV\times\cH_{M_N}\times |\Delta^m|)\ar[dr]^{p_0}\ar[dl]_{p_{1, \tau_m}}&\\
(VG_N, \cV\times M_N\times \bR_{\bq_1}^{N}\times \bR_{t_1}\times |\Delta^m|)&&(pt, \cV\times\bR^{N}_{\bq_0}\times \bR_{t_0}), 
}
\end{align}
where $p_0$ is the obvious projection, and $p_{1,\tau_m}$ is the map that does the identity on the factor $VG_N$ and on $\cV\times \cH_{M_N}\times |\Delta^m|$ it sends
\begin{align*}
(x; Q_\flat+Q_\dagg, \bq_0,\bq_1,t_0,t_1; v)\mapsto (x; Q_\flat+Q_\dagg, \bq_1+2\widetilde{Q}_v\bq_0, t_1+\widetilde{Q}_v[\bq_0];v)
\end{align*}

\begin{lemma}\label{eq: corresp cH_M_N}
The correspondence (\ref{diagram: corr base}) canonically determines a morphism from $(pt, \cV\times\bR^{N}_{\bq_0}\times \bR_{t_0})$ to $(VG_N, \cV\times M_N\times \bR_{\bq_1}^{N}\times \bR_{t_1}\times |\Delta^m|)$ in  $\Mod^{N(\Fin_*)}(\bCorr(\Slch)^{\propmap}_{\all,\all})^{\rightlax}$.  
\end{lemma}
\begin{proof}
Since everything is constant on $\cV$, we will omit it in the bulk of the proof. According to \cite[Proposition 2.22]{Jin}, we just need to check that diagram (\ref{diagram: p_1,tau_m}) is Cartesian, and the map (\ref{eq: cH_{M_N}, prop}) for the unique active map $\lng m\rng_\dagg\rightarrow\lng 0\rng_\dagg$ is proper (making the 2-morphism lying in $\propmap$).
\begin{align}\label{diagram: p_1,tau_m}
\xymatrix{
\cH_{M_N}^{\lng k\rng_\dagg}\times |\Delta^m|\ar[r]\ar[d]_{p_{1,\tau_m}}&\prod\limits_{j\in \lng k\rng^\circ} VG^{\{j,*\}}\times \cH_{M_N}^{\{\dagg,*\}}\times |\Delta^m|\ar[d]^{(id, p_{1,\tau_m})}\\
M_N^{\lng k\rng_\dagg}\times \bR_{\bq_1}^N\times \bR_{t_1}\times |\Delta^m|\ar[r]&\prod\limits_{j\in\lng k\rng^\circ} VG^{\{j,*\}}\times M_N^{\{\dagg,*\}}\times  \bR_{\bq_1}^N\times \bR_{t_1}\times |\Delta^m|
}
\end{align}

\begin{align}\label{eq: cH_{M_N}, prop}
\pi_k: \cH_{M_N}^{\lng k\rng_\dagg}&\longrightarrow \cH_{M_N}\\
\nonumber((Q_j,\bp_j)_{j\in\lng k\rng^\circ}; (\bq_0,\bq_1,t_0,t_1,Q_\flat+Q_\dagg)&\mapsto (\bq_0, \bq_1+2\sum\limits_{j\in \lng k\rng^\circ}\bp_j, t_0, t_1+\sum\limits_{j\in \lng k\rng^\circ}Q_j[\bp_j'], Q_\flat+Q_\dagg+\sum\limits_{j\in \lng k\rng^\circ}Q_j)
\end{align}
The first one is obvious. Regarding the second one,
for any $(\bq_0, \widetilde{\bq}_1,t_0,\widetilde{t}_1, Q_\flat+\hat{Q}_\dagg)\in \cH_{M_N}$, which by assumption satisfies 
\begin{align*}
\widetilde{\mu}:=\widetilde{t}_1-t_0+(Q_\flat+\hat{Q}_\dagg)[\bq_0]+\widetilde{\bq}_1\cdot\bq_0\geq 0
\end{align*}
the fiber $\pi_k^{-1}(\bq_0, \widetilde{\bq}_1,t_0,\widetilde{t}_1, Q_\flat+\hat{Q}_\dagg)$ admits a proper map to the space of splittings of $\hat{Q}_\dagg$ into $k+1$ ones in $G_N^{\lng k\rng\cup \{\dagg\}}$ with fiber at $((Q_j)_{j\in \lng k\rng^\circ}, Q_\dagg)$ identified with the space of $(\bp_j)_{j\in\lng k\rng^\circ}$ satisfying
\begin{align*}
\sum\limits_{j\in\lng k\rng^\circ} Q_j[\bp_j']+2\sum\limits_{j\in\lng k\rng^\circ}\bp_j\cdot \bq_0+\sum\limits_{j\in\lng k\rng^\circ}Q_j[\bq_0]=\sum\limits_{j\in\lng k\rng^\circ} Q_j[\bp_j'+\bq_0]\leq \widetilde{\mu},
\end{align*}
hence $\pi_k^{-1}(\bq_0, \widetilde{\bq}_1,t_0,\widetilde{t}_1, Q_\flat+\hat{Q}_\dagg)$ is compact. On the other hand $\pi_k$ is clearly closed, so it is proper as desired. 
\end{proof}

\subsubsection{The second set of correspondences}\label{subsubsec: second corr}
Again, for any $(\cU, \cV, Q_\flat^{(1)}, G^{\bullet}, M^{\bullet,\dagg})$ in $\QHam_L(U/O)$, we consider the following correspondence(s) for localizing sheaves on $\cV\times M_N\times \bR^N_{\bq_1}\times \bR_{t_1}$ along the full subcategory $\Shv^{\geq 0}(\cV\times M_N\times \bR^N_{\bq_1}\times \bR_{t_1})$: 
\begin{align}\label{eq: projector, cV}
\xymatrix@=0.5em{&(VG_N, \cV\times M_N\times \bR_{\bq_1}^N\times\bR_{t_1}\times (-\infty,0])\ar[dl]^{a_{(-\infty, 0]}}\ar[dr]_{p_{(-\infty,0]}}&\\
(VG_N, \cV\times M_N\times \bR_{\bq_1}^N\times\bR_{t_1})&&(VG_N, \cV\times M_N\times \bR_{\bq_1}^N\times\bR_{t_1})
},
\end{align}  
where $p_{(-\infty,0]}$ is the obvious projection and $a_{(-\infty,0]}$ is the addition map on the last two factors $\bR_{t_1}\times (-\infty,0]$. The functor $(a_{(-\infty,0]})_!p_{(-\infty,0]}^*$ on $\Shv(\cV\times M_N\times \bR_{\bq_1}\times\bR_{t_1};\bk)$ is the convolution with $\bk_{(-\infty,0]}$, and will be denoted as $*\bk_{(-\infty,0]}$. Similarly, we have 
\begin{align}\label{eq: projector, cV2}
\xymatrix@=0.5em{&(VG_N, \cV\times M_N\times \bR_{\bq_1}^N\times\bR_{t_1}\times [0,\infty))\ar[dl]^{a_{[0,\infty)}}\ar[dr]_{p_{[0,\infty)}}&\\
(VG_N, \cV\times M_N\times \bR_{\bq_1}^N\times\bR_{t_1})&&(VG_N, \cV\times M_N\times \bR_{\bq_1}^N\times\bR_{t_1})
},
\end{align}  
and we denote the resulting convolution functor $(a_{[0,\infty)})_*p_{[0,\infty)}^!$ by $\overset{!}{*}\omega_{[0,\infty)}$. 

We have \cite[Proposition 2.1 and 2.2]{Tamarkin1} and their analogue. 
\begin{thm}[\cite{Tamarkin1}]\label{thm: variant Tamarkin}
For any smooth manifold $X$, the convolution functor 
\begin{align*}
*\bk_{(-\infty,0]}: \Shv(X\times\bR_t; \bk)\longrightarrow \Shv(X\times\bR_t; \bk)
\end{align*}
(resp. 
\begin{align*}
\overset{!}{*}\omega_{[0,\infty)}: \Shv(X\times\bR_t; \bk)\longrightarrow \Shv(X\times\bR_t; \bk)
\end{align*}
)

gives a projector 
\begin{align*}
\Shv(X\times\bR_t; \bk)\longrightarrow \Shv(X\times\bR_t; \bk)/\Shv^{\geq 0}(X\times\bR_t; \bk):=\Shv^{<0}(X\times\bR_t;\bk),
\end{align*}
with the right-hand-side identified with the left (resp. right) orthogonal complement of $\Shv^{\geq 0}(X\times\bR_t; \bk)$. 
\end{thm}

The pair $(VG_N, \cV\times M_N\times\bR_{\bq_1}^N\times\bR_{t_1})$ determines an object $(\Shv(VG_N;\Sp), \Shv( \cV\times M_N\times\bR_{\bq_1}^N\times\bR_{t_1});\Sp)$ in $\Mod^{N(\Fin_*)}(\PrstR)$ through the symmetric monoidal functor $\ShvSp_*^!$. 

\begin{lemma}\label{lemma: Loc, left orthog}
The convolution action of $\Shv(VG_N;\Sp)$ on $\Shv( \cV\times M_N\times\bR_{\bq_1}^N\times\bR_{t_1});\Sp)$ preserves the full subcategory $\Shv^{\geq 0}( \cV\times M_N\times\bR_{\bq_1}^N\times\bR_{t_1});\Sp)$. In particular, passing to the left orthogonal complement, we have a well defined object $(\Shv(VG_N;\Sp), \Shv^{<0}( \cV\times M_N\times\bR_{\bq_1}^N\times\bR_{t_1});\Sp))$ in $\Mod^{N(\Fin_*)}(\PrstR)$. 
\end{lemma}
\begin{proof}
It suffices to show that under the correspondence
\begin{align*}
\xymatrix{\cV\times (M_N\times \bR^N_{\bq_1}\times\bR_{t_1})_{VG_N}^{\lng 1\rng_\dagg}\ar[r]^{\frj\ \ \ \ \ }\ar[d]_{a}&VG_N\times \cV\times M_N\times\bR^N_{\bq_1}\times\bR_{t_1}\\
\cV\times M_N\times\bR^N_{\bq_1}\times\bR_{t_1}
},
\end{align*}
as in the $VG_N$-module structure of $\cV\times M_N\times\bR^N_{\bq_1}\times\bR_{t_1}$, for any $\cG\in \Shv(VG_N;\Sp)$ and $\cF\in \Shv^{\geq 0}(VG_N\times\cV\times M_N\times \bR_{\bq_1}^N\times\bR_{t_1};\Sp)$, we have 
\begin{align*}
\SS(a_*\frj^!(\cG\boxtimes \cF))\subset T^{*,\geq 0}(\cV\times M_N\times\bR^N_{\bq_1}\times\bR_{t_1}).
\end{align*} 
Now apply the convolution functor $\overset{!}{*}\omega_{[0,\infty)}$ to the above diagram. By the base change formula, we get 
\begin{align*}
(\cG\boxtimes \cF)\overset{!}{*}\omega_{[0,\infty)}\simeq 0\Rightarrow (a_*\frj^!(\cG\boxtimes \cF)) \overset{!}{*}\omega_{[0,\infty)}\simeq 0.
\end{align*}
Hence the lemma follows. 
\end{proof}

Let 
\begin{align}
\label{eq: H_cV,M_N}
\bH_{\cV, M_N}^-=&\{(x, Q_\flat+Q_\dagg, \bq_1,t_1): t_1-\frac{1}{2}(A_{\ell_x}-2Q_\flat-2Q_\dagg)^{\FT}[\bq_1]<0\}\\
\nonumber\subset&\cV\times M_N\times \bR^N_{\bq_1}\times\bR_{t_1},\\
\label{eq: H_cV,M_N_0}\bH_{\cV, M_N}^0=&\{(x, Q_\flat+Q_\dagg, \bq_1,t_1): t_1-\frac{1}{2}(A_{\ell_x}-2Q_\flat-2Q_\dagg)^{\FT}[\bq_1]=0\}\\
\nonumber\subset&\cV\times M_N\times \bR^N_{\bq_1}\times\bR_{t_1}
\end{align}
where $\ell_x$ is the tangent space $T_xL$ as an affine Lagrangian in $T^*\bR^N$ and for any generalized symmetric matrix $Q$, $Q^{\FT}$ denotes for its Fourier transform, i.e. changing eigenvalues $\lambda$ to $-1/\lambda$ but keeping the eigenspaces unchanged. For any closed smooth hypersurface $H$ in $\bR^N\times\bR_t$, let $\Lambda_H$ denote the negative conormal bundle of $H$, i.e. the part of the conormal bundle inside $T^{*,<0}(\bR^N\times\bR_t)$. 
Note that by the definition of  $(\cU, \cV, Q_\flat^{(1)}, G^{\bullet}, M^{\bullet,\dagg})$, $(A_{\ell_x}-2Q_\flat-2Q_\dagg)^{\FT}$ in (\ref{eq: H_cV,M_N}) is a usual symmetric matrix. Similarly to $\cH_{M_N}$ (\ref{eq: cH_M_N}), one can upgrade $\bH_{\cV, M_N}^-$ to a $VG_N$ module in $\Corr(\Slch)$ (see Lemma \ref{lemma: H_V,M_N} below). 

\begin{remark}\label{remark: H_V,M_N}
By Lemma \ref{lemma: Loc, left orthog}, the correspondence (\ref{diagram: corr base}) gives a right-lax functor
\begin{align*}
(\Sp^\otimes, \Shv^{<0}_{\bL_\Gauss}(\cV\times \bR_{\bq_0}^N\times\bR_{t_0};\Sp))\longrightarrow (\Loc(VG_N;\Sp), \Shv^{<0}_{\bL'_\Gauss}(\cV\times M_N\times\bR_{\bq_1}^N\times\bR_{t_1}\times |\Delta^m|;\Sp),
\end{align*}
where $L'=\varphi_{-(Q_\flat+Q_\dagg)\circ \FT(L)}$ and $\bL'_\Gauss$ is the conic Lagrangian associated to $L'$. Since $L'$ is a graph type Lagrangian by assumption, its tangent space has Legendrian lifting equal to the negative conormal bundle of a smooth hypersurface, which are exactly $\bH_{\cV, M_N}^0$ restricted to each $x\in \cV$ as above. Therefore, $\Shv^{<0}_{\bL'_\Gauss}(\cV\times M_N\times\bR_{\bq_1}^N\times\bR_{t_1}\times |\Delta^m|;\Sp)$ can be identified with $\Loc(\bH^-_{\cV,M_N}; \Sp)$ (modulo the factor $|\Delta^m|$), and the pair $(\Loc(VG_N;\Sp), \Loc(\bH_{\cV,M_N}^-;\Sp))$ can be extended to $(\Shv(VG_N;\Sp), \Shv(\bH_{\cV,M_N}^-;\Sp))$ in $\Mod^{N(\Fin_*)}(\PrstL)$ (see Lemma \ref{lemma: H_V,M_N} below). 
\end{remark}

Consider the following collection of correspondences
\begin{align}
\label{eq: corr_H_cV, M_N} &\xymatrix{&(VG_N, \bH_{\cV, M_N}^-)\ar@{^{(}->}[dr]^{\fj_{\cV,M_N}}\ar[dl]_{id}&\\
(VG_N,  \bH_{\cV, M_N}^-)&&(VG_N, \cV\times M_N\times\bR^N_{\bq_1}\times\bR_{t_1})
},\\
\label{eq: corr_H_cV, M_N, G} 
&\xymatrix{&(VG_N, \bH_{\cV, M_N}^-)\ar[dr]^{p_\bH}\ar[dl]_{id}&\\
(VG_N, \bH_{\cV, M_N}^-)&&(VG_N, \cV\times M_N)
},\\
\label{eq: corr_M_N, G} 
&\xymatrix{&(VG_N, \cV\times M_N)\ar[dr]^{id}\ar[dl]_{p_{VG_N}}&\\
(G_N, \cV\times M_N)&&(VG_N, \cV\times M_N)
},
\end{align}
where $\fj_{\cV,M_N}$ is the natural inclusion, $p_{\bH}$ is the projection to the first two factors, $p_{VG_N}$ is the natural projection from the $VG_N$-module $\cV\times M_N$, inherited from the free $G_N$-module structure under the algebra homomorphism $VG_N\rightarrow G_N$ (so the factor $\bp$ acts trivially), to the free $G_N$-module $\cV\times M_N$. In the followings, $(\cV\times M_N)_{VG_N}^{\lng \bullet\rng_\dagg}$ (and similar notations) is to distinguish its $VG_N$-module structure from its free $G_N$-module structure. 

\begin{lemma}\label{lemma: H_V,M_N}
\begin{itemize}
\item[(a)] One can upgrade $\bH_{\cV, M_N}^-$ to a $VG_N$-module in $\Corr(\Slch)_{\fib,\all}$; \\

\item[(b)]
All three correspondence (\ref{eq: corr_H_cV, M_N}), (\ref{eq: corr_H_cV, M_N, G}) (resp. (\ref{eq: corr_M_N, G}) )
 canonically induce right-lax morphisms (resp. morphisms) in $\cMod^{N(\Fin_*)}(\bCorr(\Slch)_{\fib,\all}^{\openmap})^{\rightlax}$ (resp. $\cMod^{N(\Fin_*)}(\bCorr(\Slch)_{\fib,\all})$). Moreover, they induce a canonical functor
\begin{align}\label{eq: lemma b, H_V,M_N}
&(p_{VG_N})_*\circ (p_\bH^{!})^{-1}\circ \fj_{\cV,M_N}^!:\\
\nonumber&(\Loc(VG_N;\Sp), \Shv^{<0}_{\Lambda_{\bH^0_{\cV,M_N}}}(\cV\times M_N\times\bR^N_{\bq_1}\times\bR_{t_1};\Sp))\rightarrow (\Loc(G_N;\Sp), \Loc(\cV\times M_N;\Sp))
\end{align}
in $\cMod^{N(\Fin_*)}(\PrstL)$, where $(p_{\bH}^!)^{-1}$ means the inverse of 
\begin{align}\label{lemma eq: p_bH}
p_\bH^{!}: (\Loc(VG_N;\Sp), \Loc(\bH_{\cV,M_N}^-;\Sp))\overset{\sim}{\longrightarrow} (\Loc(VG_N;\Sp), \Loc(\cV\times M_N;\Sp)).
\end{align}
\end{itemize}
\end{lemma}
\begin{proof}
For (a), we set 
\begin{align*}
(\bH^-_{\cV,M_N})^{\lng n\rng_\dagg}=(M_N^{\lng n\rng_\dagg})_{VG_N}\underset{(M_N^{\lng 0\rng_\dagg})_{VG_N}}{\times} \bH_{\cV,M_N}^-,
\end{align*}
where $(M_N^{\lng n\rng_\dagg})_{VG_N}\rightarrow (M_N^{\lng0\rng_\dagg})_{VG_N}$ is the map corresponding to the unique inert map $\lng n\rng_\dagg\rightarrow \lng 0\rng_\dagg$. We just need to prove that for any active map $\lng n\rng_\dagg\rightarrow \lng 0\rng_\dagg$, the map
\begin{align*}
(\bH^-_{\cV,M_N})^{\lng n\rng_\dagg}&\longrightarrow(\bH^-_{\cV,M_N})^{\lng 0\rng_\dagg}\\
((Q_j,\bp_j)_{j\in\lng n\rng^\circ};x, Q_\flat+Q_\dagg, \bq_1,t_1)&\mapsto (x, Q_\flat+Q_\dagg+\sum\limits_{j}Q_j, \bq_1+2\sum\limits_j\bp_j, t_1+\sum\limits_jQ_j[\bp_j'])
\end{align*}
is well defined, i.e. the image satisfies the condition (\ref{eq: H_cV,M_N}). Note that it suffices to show the case for $n=1$. 

Spelling things out, we need to show the following inequality
\begin{align}\label{ineq: t_1, Q_1}
t_1+Q_1[\bp_1']-\frac{1}{2}(A_{\ell_x}-2Q_\flat-2Q_\dagg-2Q_1)^{\FT}[\bq_1+2\bp_1]<0
\end{align}
under the condition 
\begin{align*}
t_1-\frac{1}{2}(A_{\ell_x}-2Q_\flat-2Q_\dagg)^{\FT}[\bq_1]<0.
\end{align*}

Let $P: b_{\ell_x}\hookrightarrow\bR^N$ be the embedding of the projection of $\ell_x$ and $P^T:\bR^N\rightarrow b_{\ell_x}$ be the orthogonal projection. Since $A_{\ell_x}-2Q_\flat-2Q_\dagg-2Q_1$ is the same as $A_{\ell_x}-2Q_\flat-2Q_\dagg-2PP^TQ_1PP^T$, we may assume without loss of generality that $Q_1$ is supported on $b_{\ell_x}$ and the geometry is happening on $T^*b_{\ell_x}$ (that is we forget about the $\infty$-spectral part of $A_{\ell_x}$). 

Fixing $\widetilde{\bq}_1=\bq_1+2\bp_1$, and let $\bp_1$ vary in $\Supp(Q_1)$, we have 
\begin{align*}
&\partial_{\bp_1}(\frac{1}{2}(A_{\ell_x}-2Q_\flat-2Q_\dagg)^\FT[\bq_1]+Q_1[\bp_1'])\\
&=(A_{\ell_x}-2Q_\flat-2Q_\dagg)^\FT(-\frac{1}{2}\widetilde{\bq}_1+\bp_1)-(2Q_1)^{\FT}\bp_1=\mathbf{0}\\
\end{align*}
if and only if 
\begin{align*}
\varphi^1_{Q_1}(\widetilde{\bq}_1, (A_{\ell_x}-2Q_\flat-2Q_\dagg-2Q_1)^\FT\widetilde{\bq}_1)=(\bq_1,\bp_1'). 
\end{align*}
Moreover, the Hessian
\begin{align*}
(A_{\ell_x}-2Q_\flat-2Q_\dagg)^\FT-(2Q_1)^{\FT}|_{\Supp(Q_1)}
\end{align*}
is strictly negative, which follows from the rank equality (\ref{eq: rank, QHam}).
Therefore, 
\begin{align*}
\frac{1}{2}(A_{\ell_x}-2Q_\flat-2Q_\dagg)^\FT[\bq_1]+Q_1[\bp_1']
\end{align*}
obtains a maximum at $\varphi^1_{Q_1}(\widetilde{\bq}_1, (A_{\ell_x}-2Q_\flat-2Q_\dagg-2Q_1)^\FT\widetilde{\bq}_1)$. In that case, 
\begin{align*}
\frac{1}{2}(A_{\ell_x}-2Q_\flat-2Q_\dagg)^\FT[\bq_1]+Q_1[\bp_1']-\frac{1}{2}(A_{\ell_x}-2Q_\flat-2Q_\dagg-2Q_1)^{\FT}[\bq_1+2\bp_1]=0,
\end{align*}
and (\ref{ineq: t_1, Q_1}) follows. 

Furthermore, we see from the above that the fiber of $(\bH_{\cV, M_N}^{-})^{\lng 1\rng_\dagg}\rightarrow (\bH_{\cV, M_N}^{-})^{\lng 0\rng_\dagg}$ over a fixed $h=(x, Q_\flat+\hat{Q}_\dagg, \widetilde{\bq}_1,\widetilde{t}_1)$ is isomorphic to the region in 
\begin{align*}
\{(Q_1, \bp_1;x, Q_\flat+Q_\dagg): Q_\dagg+Q_1=\hat{Q}_\dagg\} 
\end{align*}
cut out by 
\begin{align*}
&t_1-\frac{1}{2}(A_{\ell_x}-2Q_\flat-2Q_\dagg)^{\FT}[\bq_1]<0\\
&\Leftrightarrow\\
&\frac{1}{2}(A_{\ell_x}-2Q_\flat-2Q_\dagg)^\FT[\widetilde{\bq}_1-2\bp_1]+Q_1[\bp_1']>\widetilde{t}_1.
\end{align*}
So the fiber is isomorphic to an open disc bundle over the space of $Q_1$. This shows that the vertical arrows in defining $\bH^-_{\cV,M_N}$ as a $VG_N$-module in $\Corr(\Slch)$ are all locally trivial fibrations. 

For part (b), one can show the three correspondences all induce right-lax morphism in $\cMod^{N(\Fin_*)}(\bCorr(\Slch)_{\fib,\all}^{\openmap})^{\rightlax}$ in the same way as the proof of Lemma \ref{eq: corresp cH_M_N} (the last one induces a genuine morphism). In particular, they respectively induce left-lax functors  
\begin{align*}
&\fj_{\cV,M_N}^!: (\Shv(VG_N;\Sp), \Shv(\cV\times M_N\times\bR^N_{\bq_1}\times\bR_{t_1};\Sp))\rightarrow (\Shv(VG_N;\Sp), \Shv(\bH_{\cV,M_N}^-;\Sp))\\
&p_\bH^!:  (\Shv(VG_N;\Sp), \Shv(\cV\times M_N;\Sp))\rightarrow (\Shv(VG_N;\Sp), \Shv(\bH_{\cV,M_N}^-;\Sp)),
\end{align*}
and a functor
\begin{align*}
(p_{VG_N})_*: (\Shv(VG_N;\Sp), \Shv(\cV\times M_N;\Sp))\rightarrow (\Shv(G_N;\Sp), \Shv(\cV\times M_N;\Sp)). 
\end{align*}
Now it suffices to show that $\fj_{\cV,M_N}^!$, $p_\bH^!$ and $(p_{VG_N})_*$ restricted to the relevant full subcategories as in (\ref{eq: lemma b, H_V,M_N}) and (\ref{lemma eq: p_bH}) are equivalences. 
This is almost evident, as 
\begin{itemize}
\item
$\Shv^{<0}_{\Lambda_{\bH^0_{\cV,M_N}}}(\cV\times M_N\times\bR^N_{\bq_1}\times\bR_{t_1})$ can be identified with the essential image of 
\begin{align*}
(\fj_{\cV,M_N})_*: \Loc(\bH_{\cV,M_N}^-;\Sp)\rightarrow \Shv(\cV\times M_N\times\bR^N_{\bq_1}\times \bR_{t_1};\Sp),
\end{align*}
\item
for any active map $\lng n\rng_\dagg\rightarrow\lng m\rng_\dagg$ in $N(\Fin_{*,\dagg})$
\begin{align*}
\xymatrix{
(\bH_{\cV,M_N}^-)^{\lng n\rng_\dagg}\ar[dr]^{f}\ar@/_1pc/ [ddr] \ar@/^1pc/[drr]&&\\
&(\bH_{\cV,M_N}^-)^{\lng m\rng_\dagg}\underset{(\cV\times M_N)_{VG_N}^{\lng m\rng_\dagg}}{\times} (\cV\times M_N)_{VG_N}^{\lng n\rng_\dagg}\ar[r]\ar[d]&(\cV\times M_N)_{VG_N}^{\lng n\rng_\dagg}\ar[d]\\
&(\bH_{\cV,M_N}^-)^{\lng m\rng_\dagg}\ar[r]&(\cV\times M_N)_{VG_N}^{\lng m\rng_\dagg}
}
\end{align*}
the map $f$ is an open embedding that induces a homotopy equivalence between locally trivial fibrations over $(\bH_{\cV,M_N}^-)^{\lng m\rng_\dagg}$.
\end{itemize}

 \end{proof}

The composition of the first set (\ref{diagram: corr base}) and the second set (\ref{eq: projector, cV}),  (\ref{eq: corr_H_cV, M_N}), (\ref{eq: corr_H_cV, M_N, G}), (\ref{eq: corr_M_N, G}) of correspondences\footnote{More precisely, one should product the second set of correspondences everywhere by $|\Delta^m|$ and modify $\bH^+_{\cV, M_N}\times |\Delta^m|$ by changing $Q_\flat+Q_\dagg$ to $Q_{\flat}+Q_\dagg+\tau_m(u), u\in |\Delta^m|, \tau_m\in \cS_{\cY}$.} for all $N\gg 0$, together with Remark \ref{remark: H_V,M_N}, Lemma \ref{lemma: H_V,M_N} and results in Appendix \ref{sec: Appendix}, yield the following. 
\begin{prop}\label{prop: muShv, J-equiv}
Given any $\cY=(\cU, \cV,Q_\flat, G, M)\in \QHam_L(U/O)$, and any $m$-simplex $\tau_m$ in $\cS_\cY$, the composition of the first and the second set of correspondences induces a right-lax functor
\begin{align}\label{eq: prop, Sp, Loc, G}
(\Sp^\otimes, \varprojlim\limits_N\Shv^{<0}_{\bL_{\Gauss}}(\cV\times \bR_{\bq_0}^N\times\bR_{t_0};\Sp))\longrightarrow (\varprojlim\limits_N\Loc(G_N;\Sp)^\otimes, \varprojlim\limits_N\Loc(\cV\times M_N\times |\Delta^m|;\Sp)),
\end{align}
that further induces an equivalence
\begin{align}\label{eq prop: Shv, J-equiv}
F_{\tau_m}:\varprojlim\limits_N\Shv^{<0}_{\bL_{\Gauss}}(\cV\times \bR_{\bq_0}^N\times\bR_{t_0};\Sp)&\overset{\sim}{\longrightarrow} \varprojlim\limits_N\Loc(\cV\times M_N\times |\Delta^m|;\Sp)^{J\text{-equiv}}\\
\nonumber&\overset{\sim}{\longrightarrow} \Loc(\cV\times M;\Sp)^{J\text{-equiv}},
\end{align}
where the last equivalence is from $!$-pushforward along the projection to $\cV\times M$. 
\end{prop}
Here recall from \cite{Jin} that 
\begin{align*}
J: \coprod\limits_n BO(n)\simeq G\longrightarrow \Pic(\bS)
\end{align*}
is realized as $\pi_*\omega_{VG}$, which is a commutative algebra object in $\Loc(G;\Sp)^{\otimes}$. Here 
\begin{align*}
\pi_*: \varprojlim\limits_N\Loc(VG_N;\Sp)^{\otimes}\longrightarrow  \varprojlim\limits_N\Loc(G_N;\Sp)^\otimes
\end{align*}
is induced from the family of commutative algebra homomorphisms $VG_N^\bullet\rightarrow G_N^\bullet$ in $\Corr(\Slch)_{\fib,\all}$ and an application of Theorem \ref{thm: appendix}. 
(see \cite[Section 4]{Jin} for more details). In particular, under (\ref{eq: prop, Sp, Loc, G}), $\bS\in \CAlg(\Sp^\otimes)$ goes to $\pi_*\omega_{VG}\in \CAlg(\varprojlim\limits_N\Loc(G_N;\Sp)^\otimes
)\simeq \CAlg(\Loc(G;\Sp)^{\otimes})$. Note that because of the properness of the map $G^{\lng n\rng}\rightarrow G^{\lng m\rng}$ for any active morphism $\lng n\rng\rightarrow \lng m\rng$, the symmetric monoidal structure on $\Loc(G;\Sp)^{\otimes}$ coincides with that from viewing $G$ as a commutative monoid in $\Spc$ (with $!$-pullback and $!$-pushforward structure maps).

\subsubsection{Construction of $F_L$}
Given objects $\cY_i=(\cU_i,\cV_i, Q_\flat^{(i)}, G^{\bullet}_{(i)}, M_{(i)}^{\bullet,\dagg}), i=1,2$,  for any $n$-simplex $\sigma_n$ in 
$\Maps_{\QHam_L(U/O)}(\cY_1, \cY_2)$
defined by a family of $Q_{12, u}\in G_{(2)}, u\in |\Delta^n|$ (the geometric realization of any $n$-simplex will be identified with $\sum\limits_{j=1}^{n+1}t_j=1$ in $\bR^{n+1}$) such that $\widetilde{Q}_u=Q_\flat^{(2)}+Q_{12,u}-Q_\flat^{(1)}\geq 0$, and an $n$-simplex $\tau_{n}: |\Delta^n|\rightarrow\cS_{\cY_2}$ (defined in Subsection \ref{subsubsec: first corr}
), consider the following commutative diagram for $N\gg 0$ (which can be viewed as an object in 
\begin{align*}
\Fun(\Delta^1\times \Delta^1, \Corr(\Fun^\diamond(N(\Fin_{*,\dagg}), (\Slch)))_{\inert, \all}),
\end{align*}
with the vertex $(0,0)$ and $(1,1)$ be respectively the bottom right corner and the top left corner in the diagram)
\begin{align}\label{eq: square in Corr}
\xymatrix@=0.6em{
&(VG_{(1), N}, \cV_{1}\times\cH_{M_{(1), N}}\times |\Delta^n|)\ar[dl]_{p_{1,\widetilde{\tau}_n}^{(1)}\ \ \ }\ar[dr]^{\ \ \ \ \ \ p_2^{(1)}}&\\
(VG_{(1),N}, {\substack{\cV_{1}\times M_{(1),N}\times \bR_{\bq_1}^N\\
\times \bR_{t_1}
\times |\Delta^n|}})&&(pt, \cV_{1}\times\bR^{N}_{\bq_0}
\times\bR_{t_0})\\
&(VG_{(1), N}, \cV_{2}\times\cH_{M_{(1), N}}\times |\Delta^n|)\ar@^{(->}[uu]\ar[dr]^{\ \ \ \ \ \ p^{(1)}_2|_{\cV_{2}}}\ar[dd]^{\fri_{\sigma_n}}\ar[dl]_{p^{(1)}_{1, \widetilde{\tau}_n}|_{\cV_{2}}\ \ \ }&\\
(VG_{(1),N}, {\substack{\cV_{2}\times M_{(1),N}\times \bR_{\bq_1}^N\\
\times \bR_{t_1}
\times |\Delta^n|}})\ar@^{(->}[uu]\ar@^{(->}[dd]_{\frj_{\sigma_n}}&&(pt, \cV_{2}\times\bR^{N}_{\bq_0}
\times\bR_{t_0})\ar[dd]^{id}\ar@^{(->}[uu]\\
&(VG_{(2), N}, \cV_{2}\times \cH_{M_{(2), N}}\times |\Delta^n|)\ar[dr]^{\ \ \ \ p_2^{(2)}}\ar[dl]_{p_{1, \tau_{n}}^{(2)}}&\\
(VG_{(2),N}, {\substack{\cV_{2}\times M_{(2), N}\times \bR_{{\bq}_1}^N\\
\times\bR_{t_1}\times |\Delta^n|}})&&(pt, \cV_{2}\times \bR^{N}_{\bq_0}\times \bR_{t_0})
}
\end{align}
where $\frj_{\sigma_n}$ is the embedding induced from the morphisms $(VG_{(1),N}, M_{(1),N})\rightarrow (VG_{(2),N}, M_{(2),N})$ determined by $\sigma_n$ and the identity map on the other factors, 
\begin{align*}
&\fri_{\sigma_n}: (Q_1,\bp; \bq_0,\bq_1,t_0,t_1,Q_\flat^{(1)}, u)\mapsto (Q_1,\bp;\bq_0,\bq_1+2\widetilde{Q}_u\bq_0,t_0,t_1+\widetilde{Q}_u(\bq_0),Q^{(2)}_\flat+Q_{12,u},u ),
\end{align*}
and $\widetilde{\tau}_n$ is the $n$-simplex in $\cS_{\cY_1}=\cS_{(\cU_1, Q_\flat^{(1)}, G_{(1)}, M_{(1)})}$ defined by 
\begin{align*}
\widetilde{\tau}_{n}(u)=Q_\flat^{(2)}+Q_{12,u}-Q_\flat^{(1)}+\tau_{n}(u).
\end{align*}
The top vertical arrows are all obvious inclusions induced from $\cV_{2}\hookrightarrow \cV_{1}$.

These 
define a natural transformation
\begin{align}\label{diagram: Corr Y_1, Y_2}
\begin{xy}
\xymatrix{\Sing_n(\cS_{\cY_2}\times \Maps_{\QHam_L(U/O)}(\cY_1,\cY_2))
\ar[r]^{pr_1}\ar[d]_{\frg_{n}}
&\Sing_n(\cS_{\cY_2})
\ar[d]^{\frf_{\cY_2,n}}\ar[dl]**\dir{=}
\\
\Sing_n(\cS_{\cY_1})
\ar[r]_{\frf_{\cY_1, n}\ \ \ \ \ \ \ \ \ \ \ \ \ }
&\Fun(\Delta^1\times \bZ^{op}_{\gg 0}, \Corr(\Fun^\diamond(N(\Fin_{*,\dagg}), (\Slch)))_{\inert, \all}) 
}
\end{xy}
\end{align}
where $\bZ_{\gg 0}$ is the indexed category for the dimensions $N$, $\frg_n(\tau_n, \sigma_n)=\widetilde{\tau}_n$,  $f_{\cY_2,n}$ is sending $\tau_{n}\in \Sing_n(\cS_{\cY_2})$ to the bottom correspondence in (\ref{eq: square in Corr}) and $f_{\cY_1,n}$ is sending $\widetilde{\tau}_n\in \Sing_n(\cS_{\cY_1})$ to the top correspondence in (\ref{eq: square in Corr}). 
Ideally, we would like to extend diagram (\ref{diagram: Corr Y_1, Y_2}) by further composing with the functor 
\begin{align}\label{eq: compose muShv}
\Fun(\Delta^1, \Corr(\Fun^\diamond(N(\Fin_{*,\dagg}), (\Slch)))_{\inert, \all})\longrightarrow \Fun(\Delta^1, \PrstL)
\end{align}
from first applying $\ShvSp_*^!$ to $\Fun(\Delta^1, \cMod^{N(\Fin_*)}(\PrstR)^{\rightlax})$(note that the horizontal correspondences in (\ref{diagram: Corr Y_1, Y_2}) are exactly the first set of correspondences in Subsubsection \ref{subsubsec: first corr}
), then applying the second set of correspondences in Subsubsection \ref{subsubsec: second corr} (restricting to the microlocal sheaf categories $\Shv_{\bL_\Gauss}^{<0}(\cV_{i}\times \bR_{\bq_0}^N\times\bR_{t_0}\times|\Delta^n|), i=1,2$), and finally getting to $F_{\tau_n}$ (\ref{eq prop: Shv, J-equiv}) by Proposition \ref{prop: muShv, J-equiv} after taking $\varprojlim\limits_N$ everywhere. The goal is to construct a canonical functor 
\begin{align}\label{eq: F_L, sketch}
F_L: \QHam_L(U/O)^{op}\longrightarrow \Fun(\Delta^1,\PrstL). 
\end{align}
from these data. 
A small problem about doing this is that if we do $*$-pushforward along the upper vertical arrows in (\ref{eq: square in Corr}), we will create more singular support than required (so not well defined). To remedy this, we will do the followings:
\begin{itemize}
\item First, we revert all arrows of the form 
\begin{align}\label{eq: QHam hor}
(\cU_1, \cV_1, Q_\flat, G^\bullet, M^{\bullet,\dagg})\rightarrow (\cU_2, \cV_2, Q_\flat, G^\bullet, M^{\bullet,\dagg})
\end{align}
in $\QHam_L(U/O)$, i.e. $\cU_1\supset \cU_2$ and $\cV_1\supset \cV_2$ but the latter three components are identical and the morphism is identity on those factors. The new $\infty$-category, which has the same object as $\QHam_L(U/O)$ but requiring the reversed containment relation on $\cV_i$ and $\cU_i$ in the definition of morphism spaces (\ref{eq: Maps, QHam_L}), is denoted by $\QHam_L^{\cV\op}(U/O)$.
 
\item We replace $\QHam_L(U/O)$ by $\QHam_L^{\cV\op}(U/O)$ in (\ref{diagram: Corr Y_1, Y_2}) and composing with the functor (\ref{eq: compose muShv}) described above, then we get a natural transformation between two functors 
\begin{align*}
F_{\cY_2, n}\circ pr_1\Rightarrow F_{\cY_1, n}\circ\frg_n: \Sing_n(\cS_{\cY_2}\times \Maps_{\QHam^{\cV\op}_L(U/O)}(\cY_1,\cY_2))\longrightarrow \Fun(\Delta^1,\PrstL)
\end{align*}

\item We will then construct a canonical functor 
\begin{align*}
(\QHam_L^{\cV\op}(U/O))^{op}\longrightarrow \Fun(\Delta^1,\bPrstL). 
\end{align*}
Let $\opincl$ be the class of morphisms in $\QHam_L^{\cV\op}(U/O)$ coming from reversing the class in (\ref{eq: QHam hor}). The above functor sends $\opincl$ to left adjointable functors in $\Fun(\Delta^1,\bPrstL)$ and satisfies the right Beck-Chevalley condition with respect to $vert=\opincl$ (cf. \cite[Chapter 7, Definition 3.1.5]{GaRo}), so we can revert them by taking their left adjoints and get the desired functor (\ref{eq: F_L, sketch}). More explicitly, we can view $\QHam_L(U/O)$ as the correspondence category of $\QHam_L^{\cV\op}(U/O)$ with vertical arrows given by $\opincl$ and horizontal arrows given by the class with $\cU_1=\cU_2, \cV_1=\cV_2$ (every morphism has such a unique factorization), then apply\footnote{In applying the theorem, use $horiz=\all$, then restrict to the so defined subclass of horizontal arrows.} \cite[Chapter 7, Theorem 3.2.2]{GaRo}. 
\end{itemize}

By the compatibility of the diagram (\ref{diagram: Corr Y_1, Y_2}) with respect to degeneration and face maps, the diagrams together with (\ref{eq: compose muShv}) assemble to define a natural transformation
\begin{align*}
\xymatrix{
|\Sing_\bullet (\cS_{\cY_2}\times \Maps_{\QHam^{\cV\op}_L(U/O)}(\cY_1,\cY_2))|\ar[r]^{\ \ \ \ \ \ \ \ \ \ \ \ \ \ \ \ \pr_1}\ar[d]_{\frg}&|\Sing_\bullet(\cS_{\cY_2})|\ar[dl]**\dir{=}\ar[d]^{F_{\cY_2}}\\
|\Sing_\bullet(\cS_{\cY_1})|\ar[r]_{F_{\cY_1}}&\Fun(\Delta^1, \PrstL)
}.
\end{align*}

For simplicity, in the following we will not distinguish a space with its singular simplicial complex (or further taking geometric realization). Let $F_L(\cY_i)$ be the left Kan extension of $F_{\cY_i}$ along the map $\cS_{\cY_i}\rightarrow pt$. Since $\cS_{\cY_i}$ is contractible, $F_L(\cY_i)$ is canonically isomorphic to the object 
\begin{align*}
(F_{\tau_n}: \varprojlim\limits_N\Shv^{<0}_{\bL_\Gauss}(\cV_i\times \bR^N_{\bq_0}\times\bR_{t_0};\Sp)\overset{\sim}{\longrightarrow} \Loc(\cV_i\times M;\Sp)^{J\text{-equiv}}),
\end{align*} 
for any $\tau_n\in \Sing_n(\cS_{\cY_i})$, and this gives a canonical morphism (up to a contractible space of choices)
\begin{align}\label{eq: Y_i, F_i}
F_L(\cY_1,\cY_2): \Maps_{\QHam^{\cV\op}_L(U/O)}(\cY_1, \cY_2)\rightarrow \Maps_{\Fun(\Delta^1, \PrstL)}(F_L(\cY_2), F_L(\cY_1)). 
\end{align}

To see that the morphism (\ref{eq: Y_i, F_i}) is compatible with compositions, note that for three objects $\cY_i, i=1,2,3$, we have the following diagram for every $n\geq 0$:
\begin{align}\label{diagram: Corr Y_1, Y_2, Y_3}
\xymatrix@C=0.6em{
&\Sing_n({\substack{\cS_{\cY_3}\times \Maps(\cY_2, \cY_3)\ar[dl]_{\comp}\\
\times \Maps(\cY_1, \cY_2)}})\ar[r]^{pr_1}\ar[d]_{\frg_{3,2}}&\Sing_n(\cS_{\cY_3})\ar@/^3pc/[dd]^{\frf_{\cY_3}}\ar[dl]**\dir{=}\\
\Sing_n(\cS_{\cY_3}\times \Maps(\cY_1, \cY_3))\ar@/^4pc/[urr]^{pr_1}\ar[dr]_{\frg_{3,1}}&\Sing_n(\cS_{\cY_2}\times \Maps(\cY_1, \cY_2))\ar[d]_{\frg_{2,1}}\ar[r]^{ pr_1}&\Sing_n(\cS_{\cY_2})\ar[d]^{\frf_{\cY_2}}\ar[dl]**\dir{=}\\
&\Sing_n(\cS_{\cY_1})\ar[r]^{\frf_{\cY_1}\ \ \ \ \ \ \ \ \ \ \ \ \ \ \ \ \ }&\Fun({\substack{\Delta^1\times \bZ^{op}_{\gg0}, \Corr(\Fun^\diamond(N(\Fin_{*,\dagg}), \Slch)})_{\inert, \all}})
}
\end{align}
where 
left square is strictly commutative and there is a strict equality between the composite natural transformation
\begin{align*}
\frf_{\cY_3}\circ pr_1\Rightarrow \frf_{\cY_2}\circ pr_1\circ\frg_{3,2}\Rightarrow \frf_{\cY_1}\circ \frg_{2,1}\circ\frg_{3,2}
\end{align*}
and the natural transformation 
\begin{align*}
\frf_{\cY_3}\circ pr_1\Rightarrow \frf_{\cY_1}\circ \frg_{3,1}\circ \comp. 
\end{align*}
Applying $\ShvSp_*^!$ and restricting to the microlocal sheaf categories, then applying Proposition \ref{prop: muShv, J-equiv} and using the compatibility with face and degeneracy maps as above, we get a contractible space of commutative diagrams 
\begin{align}\label{diagram: Y_1, Y_2, Y_3}
\xymatrix{
F_L(\cY_1,\cY_2,\cY_3): &\Maps(\cY_2,\cY_3)\times \Maps(\cY_1,\cY_2)\ar[r]\ar[d]&\Maps(\Delta^2, \Fun(\Delta^1, \PrstL))\ar[d]\\
&\varprojlim\limits_{I\subsetneq \{0,1,2\}}\prod\limits_{i<j\in I}\Maps(\cY_{3-j},\cY_{3-i})\ar[r]^{\substack{(F_L(\cY_{3-i},\cY_{3-j}),\\
 F_L(\cY_k))}}&\Maps(\partial \Delta^2, \Fun(\Delta^1, \PrstL))
}
\end{align}
whose restriction to the second row gives a trivial Kan fibration to the (contractible) space of maps defined by $F_L(\cY_i, \cY_j)$ and $F_L(\cY_k)$  as above. 
Clearly, the space of $F_L(\cY_1,\cY_2, \cY_3)$ and $F_L(\cY_1,\cY_2)$ are compatible with degenerations, e.g. if $\cY_1=\cY_2$, then $F_L(\cY_1, \cY_1, \cY_3)|_{\{id_{\cY_1}\}\times \Maps(\cY_1,\cY_3)}$ canonically factors through $\Maps(\Delta^{\{0,1\}},\Fun(\Delta^1,\PrstL))$. More systematically, one adds the degeneration maps into diagram (\ref{eq: Y_i, F_i}) and (\ref{diagram: Y_1, Y_2, Y_3}) if some adjacent $\cY_j$'s are equal.

By induction on the number of $\cY_i$'s involved, we get a canonical functor (up to a contractible space of  choices) in $\Spc^{\Delta^{op}}$
\begin{align}\label{eq: F_L, Seq}
F_L: (\coprod\limits_{(\cY_1,\cdots, \cY_{\bullet+1})}\Maps(\cY_\bullet, \cY_{\bullet+1})\times\cdots\times\Maps(\cY_1,\cY_2))\overset{F_L(\cY_1,\cdots, \cY_{\bullet+1})}{\longrightarrow} \Maps(\Delta^\bullet, \Fun(\Delta^1, \PrstL)).
\end{align}
Since the left-hand-side is equivalent to $\Seq_\bullet (\QHam_L^{\cV\op}(U/O)^{op})$, the functor $F_L$ (\ref{eq: F_L, Seq}) corresponds to a functor in $\OneCat$, denoted by $F_L$ as well 
\begin{align}\label{eq: modified F_L}
F_L: \QHam_L^{\cV\op}(U/O)^{op}\longrightarrow \Fun(\Delta^1, \PrstL).
\end{align}
As remarked before, by reversing the arrows in $\opincl$ through sending to the left adjoints (for this we need to replace $\Fun(\Delta^1, \PrstL)$ by $\Fun(\Delta^1, \bPrstL)$), and by some abuse of notations, we get the desired functor $F_L$ (\ref{eq: F_L, sketch}).

\subsection{The main diagram of categories and functors}

The main diagram of categories and functors is the following (in which the lower right square and the upper rightmost triangle are naturally commutative):  
\begin{align}\label{diagram: main Kan}
\xymatrix{
&\QHam_L(U/O)^{op}\ar[dd]_{p_L}\ar@/^3pc/[ddrrr]^{F_L}\ar[dl]_{(G,M)}\ar[dr]^{\Loc(-;\Sp)^{J\text{-equiv}}}&\\
\cMod^{N(\Fin_*)}(\Spc)^{op}&&\PrstL\\
&\Open(L)\ar[rrr]_{\RKan_{p_L}(F_L)}\ar[dr]_{\RKan_{p_L}(ev_0\circ F_L)}\ar[ur]_{\ \ \ \  \RKan_{p_L} (\Loc(-;\Sp)^{J\text{-equiv}})}\ar[ul]^{\RKan_{p_L}(G,M)\ \ \ \ }&&&\Fun(\Delta^1,\PrstL)\ar[ull]_{ev_1}\ar[dll]^{ev_0}\\
&&\PrstL
}
\end{align}
where 
\begin{itemize}
\item[(i)]$ (G,M)$ is the functor sending any object $(\cU, \cV\overset{\text{open}}{\subset} L, G^\bullet, M^{\bullet,\dagg})$ to the pair $(G,M)$ viewed as an object in $\cMod^{N(\Fin_*)}(\Spc)^{op}$, and taking any morphism to the corresponding embedding; 

\item[(ii)] the functor $F_L:  \QHam_L(U/O)^{op}\rightarrow \Fun(\Delta^1, \PrstL)$ has been constructed in Section \ref{subsec: F_L}, as a small modification of (\ref{eq: modified F_L}) with some abuse of notations. Its evaluation at $0\in \Delta^1$ gives $\mu\cShv_{L}(-;\Sp)$,  and its evaluation at $1\in \Delta^1$ gives $\Loc(-;\Sp)^{J\text{-equiv}}$;
 
\item[(iii)] We take the right Kan extension of the functor $(G,M)$ (resp. $\Loc(-;\Sp)^{J\text{-equiv}}$, $F_L$ and $ev_0\circ F_L$) along $p_L$, and denote it by $\RKan_{p_L}(G,M)$ (resp. $\RKan_{p_L}(\Loc(-;\Sp)^{J\text{-equiv}})$, $\RKan_{p_L}(F_L)$ and $\RKan_{p_L}(ev_0\circ F_L)$). 
\end{itemize}

For any topological space $X$, let $\PcoShv(X;\PrkL)$ (resp. $\coShv(X;\PrkL)$) be the $(\infty,1)$-category of pre-cosheaves (resp. cosheaves) of (presentable) $\bk$-linear categories on $X$, with continuous (i.e. colimit-preserving) corestriction functors. 
\begin{lemma}\label{lemma: Loc_F^pre}
For any topological space $X$, given a presheaf of spaces $\cF^{pre}: \Open(X)^{op}\rightarrow \Spc$ over $X$, let $\cLoc_{\cF^{pre}}\in \PcoShv(X;\PrkL)$ defined by 
\begin{align*}
\cLoc_{\cF^{pre}}:&\Open(X)\mapsto \PrkL\\
&U\mapsto \Loc(\cF^{pre}(U);\bk)\simeq \Fun(\cF^{pre}(U),\Mod(\bk))
\end{align*}
Let $\cF$ be the sheafification of $\cF^{pre}$. Then $\cLoc_{\cF}$ gives a cosheafification of $\cLoc_{\cF^{pre}}$. Equivalently, this means for any cosheaf $\fG\in \coShv(X;\PrkL)$, we have an isomorphism of spaces ($\infty$-groupoids)
\begin{align*}
\Maps_{\PcoShv(X;\PrkL)}(\fG, \cLoc_{\cF^{pre}})\simeq \Maps_{\coShv(X;\PrkL)}(\fG, \cLoc_{\cF}).
\end{align*} 

\end{lemma}

\begin{proof}
We have the following isomorphism of spaces
\begin{align*}
&\Maps_{\PcoShv(X;\PrkL)}(\fG, \cLoc_{\cF^{pre}})\simeq \Maps_{\PShv(X;\Spc)}(\cF^{pre}, \Fun^L_\bk(\fG(-), \Mod(\bk))^{\Spc})\\
\simeq &\Maps_{\Shv(X;\Spc)}(\cF, \Fun^L_\bk(\fG(-), \Mod(\bk))^{\Spc})\simeq \Maps_{\coShv(X;\PrkL)}(\fG, \cLoc_{\cF}),
\end{align*}
where $\PShv(X;\Spc)$ is the $\infty$-category of presheaves of spaces on $X$, and for any $\infty$-category $\cC$, $\cC^{\Spc}$ is the space obtained from removing all non-invertible morphisms in $\cC$ (equivalently, the image of $\cC$ under the right adjoint of the full embedding $\Spc\hookrightarrow \OneCat$). 

\end{proof}

\begin{lemma}\label{lemma: cofinal D(x)}
For any $x\in L$, let $\fC(x)$ be the full subcategory of $\QHam_L(U/O)$ consisting of $(\cU, \cV, Q_\flat, G^\bullet, M^{\bullet,\dagg}), x\in \cV$, and let $\fD(x)$ be the full subcategory of $\fC(x)$, whose objects satisfy $\varphi_{Q_\flat}(\FT(\ell_x))$ has no negative spectral part. Then 
\begin{itemize}
\item[(a)] $\fD(x)$ is a filtered poset;
\item[(b)] the inclusion $\fD(x)\hookrightarrow \fC(x)$ is cofinal. 
\end{itemize}
\end{lemma}
\begin{proof}
(a). The partial ordering $(\cU_1, \cV_1, Q_\flat^{(1)}, G_{(1)}^\bullet, M_{(1)}^{\bullet,\dagg})\prec(\cU_2, \cV_2, Q_\flat^{(2)}, G_{(2)}^\bullet, M_{(2)}^{\bullet,\dagg})$ in $\fD(x)$ is given by $\cU_1\supset \cU_2, \cV_1\supset \cV_2$, $Q_\flat^{(2)}-Q_\flat^{(1)}$ is nonnegative (recall that the difference would contribute trivially to any rank equality), and $G_{(1)}^\bullet\subset G_{(2)}^{\bullet}$. Given any two objects  $(\cU_i, \cV_i, Q_\flat^{(i)}, G_{(i)}^\bullet, M_{(i)}^{\bullet,\dagg}), i=1,2$ in $\fD(x)$, let $Q'_\flat$ be a (negatively stabilized) quadratic form supported on (the stabilization of) $b_{\ell_x}$ such that $Q'_\flat-Q_\flat^{(i)}|_{b_{\ell_x}}$ are strictly positive for $i=1,2$, and have eigenvalues bounded below by some $\epsilon>0$. Let $Q_\flat''$ be a positive quadratic form supported on $b_{\ell_x}^{\perp}$ such that its eigenvalues are greater than some $R\gg 0$. 

Let $\widehat{Q}_\flat=Q_\flat'+Q_\flat''$.  We claim that for $R$ sufficiently large, $\widehat{Q}_\flat\geq Q_\flat^{(i)}, i=1,2$. Indeed, since $Q_\flat^{(i)}, i=1,2$ both have bounded eigenvalues, for any $u\in b_{\ell_x}$ and $v\in b_{\ell_x}^{\perp}$, we have
\begin{align}\label{eq: difference on u,v}
\nonumber (\widehat{Q}_\flat-Q_\flat^{(i)})(u+v)=&(Q_\flat'(u)-Q_\flat^{(i)}(u))+(Q_\flat''(v)-Q_\flat^{(i)}(v))-2Q_\flat^{(i)}(u,v)\\
\geq &(Q_\flat'(u)-Q_\flat^{(i)}(u))+(Q_\flat''(v)-Q_\flat^{(i)}(v))-2K|u||v|
\end{align}
for some $K>0$ depending only on $Q_\flat^{(i)}, i=1,2$. Thus for $R$ sufficiently large, (\ref{eq: difference on u,v}) is always nonnegative. 
Now we fix such $Q_\flat'$ and $Q_\flat''$ as above such that $\varphi_{Q_\flat'}(\FT(\ell_x))$ is of graph type and has no negative spectral part. For sufficiently large $K, R'> 0$ and sufficiently small $\epsilon'>0$, we can choose $ (\widehat{\cU},\widehat{\cV}, \widehat{Q}_\flat, G_{K, \epsilon', R', b_{\ell_x}}^\bullet, M_{K, \epsilon', R', b_{\ell_x}}^{\bullet,\dagg})$ as defined in Lemma \ref{lemma: QHam objects}, such that it dominates both 
$(\cU_i, \cV_i, Q_\flat^{(i)}, G_i^\bullet, M_i^{\bullet,\dagg}), i=1,2$.

(b). We show that $\fD(x)\hookrightarrow \fC(x)$ is cofinal, which means that for any $(\cU, \cV, Q_\flat, G^\bullet, M^{\bullet,\dagg})\in \fC(x)$, the $\infty$-category $\fD(x)\underset{\fC(x)}{\times}\fC(x)_{(\cU, \cV, Q_\flat, G^\bullet, M^{\bullet,\dagg})/}$ is contractible. Since $\fD(x)\underset{\fC(x)}{\times}\fC(x)_{(\cU, \cV, Q_\flat, G^\bullet, M^{\bullet,\dagg})/}$ is a left fibration over $\fD(x)$ and $\fD(x)$ is a filtered poset, the geometric realization $|\fD(x)\underset{\fC(x)}{\times}\fC(x)_{(\cU, \cV, Q_\flat, G^\bullet, M^{\bullet,\dagg})/}|$ can be calculated by the (homotopy) colimit
\begin{align*}
\varinjlim\limits_{j\rightarrow\infty}\Maps_{\QHam_L(U/O)}((\cU, \cV, Q_\flat, G^\bullet, M^{\bullet,\dagg}),  (\cU_j, \cV_j, Q_\flat^{(j)}, G_j^\bullet, M_j^{\bullet,\dagg}))
\end{align*}
over any cofinal sequence $(\cU_j, \cV_j, Q_\flat^{(j)}, G_j^\bullet, M_j^{\bullet,\dagg})$ in $\fD(x)$ if it exists (we will see the existence below). Since the morphisms in the above sequence are all inclusions, the colimit is taking infinite unions.

Without loss of generality (up to congruent relations), we may assume that $\FT(\varphi_{Q_\flat}(\FT(\ell_x)))$ has spectral decomposition concentrated in $\pm 1$ and $\infty$, and let $V_1$, $V_{-1}$ and $V_\infty$ be the corresponding eigenspaces (note that $V_{-1}$ is the stabilized part).  First, Lemma \ref{lemma: negative Q_flat dominates}
 says that the fiber product $\fD(x)\underset{\fC(x)}{\times}\fC(x)_{(\cU, \cV, Q_\flat, G^\bullet, M^{\bullet,\dagg})/}$ is nonempty. 
Second, let 
\begin{align}
\label{eq: Q hat R}&Q_{\flat,R}'=Q_\flat|_{b_{\ell_x}}-(1/2+1/R)I_{V_{1}}+(1/2-1/R)I_{V_{-1}},\ Q_{\flat,R}''=Q_\flat|_{b_{\ell_x}^\perp}+R\cdot I_{b_{\ell_x}^\perp}\\
\nonumber&\widehat{Q}_{\flat,R}=Q_{\flat,R}'+Q_{\flat,R}'',\ R>0
\end{align}
then $Q_\flat-\widehat{Q}_{\flat,R}$ is of the form 
\begin{center}
\begin{tikzpicture}
\draw (0,0)--(7,0)--(7,7)--(0,7)--(0,0); 
\draw (0,5)--(7,5);
\draw (2,0)--(2,7);
\draw (0,2)--(7,2);
\draw (5,0)--(5,7);
\draw[decorate, decoration={brace,amplitude=10pt}, xshift=-4pt, yshift=0pt] (0,5)--(0,7) node [black, midway, xshift=-0.6cm] {\footnotesize$V_{1}$}; 
\draw[decorate, decoration={brace,amplitude=10pt,  mirror, raise=4pt},  yshift=0pt] (7,2)--(7,5) node [black, midway, xshift=0.8cm] {\footnotesize$V_{-1}$}; 
\draw[decorate, decoration={brace,amplitude=10pt,  mirror, raise=4pt},  yshift=0pt] (7,0)--(7,2) node [black, midway, xshift=0.8cm] {\footnotesize$b_{\ell_x}^\perp$}; 
\draw (1,6) node {$\substack{(1/2+1/R)I_{V_{1}}}$}; 
\draw (3.5, 3.5) node {$(-\frac{1}{2}+\frac{1}{R})I_{V_{-1}}$};
\draw (6,1) node {$-R\cdot I_{b_{\ell_x}^\perp}$};
\draw (1, 3.5) node {$0$};
\draw (3.5, 6) node {$0$};
\draw (1, 1) node{$Y_{13}^T$};
\draw (3.5, 1) node {$Y_{23}^T$};
\draw (6, 3.5) node {$Y_{23}$};
\draw (6, 6) node {$Y_{13}$};
\end{tikzpicture}
\end{center}
where $Y_{13}$ and $Y_{23}$ are independent of $R$ and has finite norm. Let 
\begin{align}\label{eq: cofinal sequence}
(\cU_n, \cV_n, \widehat{Q}_{\flat,n}, G_{(n)}^\bullet, M_{(n)}^{\bullet, \dagg}), n\in\bZ_{\geq0}
\end{align}
 be a sequence of objects in $\fD(x)$, where $ \widehat{Q}_{\flat,n}$ is defined as in (\ref{eq: Q hat R}) with $R=n$; $G_{(n)}=G_{K_n, \epsilon_n, nR_n, b_{\ell_x}}^\bullet$, 
which has been defined in the proof of Lemma \ref{lemma: QHam objects} (\ref{eq: G_{K,epsilon, R, b_ell}}) with $K_n\uparrow+\infty, \epsilon_n\downarrow 0^+, R_n\uparrow +\infty$; $\cU_n$ and $\cV_n$ are decreasing and $\bigcap\limits_n\cU_n=\bigcap\limits_n\cV_n=\{\ell_x\}$. Then we have the following observations
 \begin{itemize}
 \item[(1)] 
The sequence  (\ref{eq: cofinal sequence})  is  cofinal  in $\fD(x)$;

 \item[(2)] For any $\epsilon>0$, consider the space of quadratic forms $Q_{12}$ with $Q_{12}|_{V_{1}}> (1/2+\epsilon) I_{V_{1}}$ of rank $\dim V_{1}$ (so automatically positive) satisfying 
 \begin{itemize}
 \item[2a)] the support intersects trivially with that of every $Q_1\in G$; 
\item[ 2b)] The eigenvalues of $Q_{12}$ are less than some $R'\gg 0$. 
\end{itemize}
Then it satisfies that $Q_{12}\geq Q_\flat-\widehat{Q}_{\flat,n}$ and $G,G+Q_{12}\subset G_{(n)}$  for $n$ large enough, and this space embeds into the mapping space from $(\cU, \cV, Q_\flat, G^\bullet, M^{\bullet, \dagg})$ to $(\cU_n, \cV_n, \widehat{Q}_{\flat,n}, G_{(n)}^\bullet, M_{(n)}^{\bullet, \dagg}))$ as an open subset for $n\gg 0$. 
 
\item[(3)]  
\begin{align}\label{eq: all Q_12}
\nonumber&\bigcup\limits_{n=1}^{+\infty}\Maps_{\QHam_L(U/O)}((\cU, \cV, Q_\flat, G^\bullet, M^{\bullet,\dagg}),  (\cU_n, \cV_n, \widehat{Q}_{\flat,n}, G_{(n)}^\bullet, M_{(n)}^{\bullet,\dagg}))\\
=&\{Q_{12}: Q_{12}|_{V_{1}}>\frac{1}{2}I_{V_{1}}, \ \rank(Q_{12})=\dim V_{1},\ \Supp(Q_{12})\cap \Supp(Q_1)=\{0\}, \forall Q_1\in G\}. 
\end{align}
\end{itemize}

By assumption ic) on $(\cU, Q_\flat, G^\bullet, M^{\bullet, \dagg})$, the sublocus $Z_{V_{1}}$ of $\Hom(V_{1}, V_{-1}\oplus V_\infty)$ as an open subset in $\Gr(\dim {V_{1}},\infty)$ consisting of $E$ with $E\cap \Supp(Q_1)=\{0\}, \forall Q_1\in G$,  is contractible. On the other hand, the space (\ref{eq: all Q_12}) is a trivial fibration over $Z_{V_{1}}$, therefore it is also contractible. 
\end{proof}

Recall that we have the universal principal $\Omega(U/O)$-bundle
\begin{align*}
\xymatrix{\bZ\times BO\simeq\Omega(U/O)\ar[r]&E(\bZ\times BO)\simeq pt\ar[d]^{\eta_{U/O}}\\
&B(\bZ\times BO)\simeq U/O
},
\end{align*}
and the Gauss map\footnote{As remarked after Theorem \ref{thm: sec proof}, technically we should write $\vartheta\circ\gamma$ instead, but to simplify the notation, we will omit $\vartheta$.} $\gamma: L\longrightarrow U/O$. A classical realization for $E(\bZ\times BO)$ is $\text{Path}_*(U/O)$, the path space on $U/O$ with ending point at the monoidal unit $*\in U/O$. 

\begin{prop}\label{prop: proof of thm}
\item[(a)] The functor $\LKan_{p_L^{op}}(G,M): \Open(L)^{op}\rightarrow \cMod^{N(\Fin_*)}(\Spc)$ determines a pair of presheaves of spaces $(\Phi^{pre}_\alg, \Phi^{pre}_{\sfmod})\in \cMod^{N(\Fin_*)}(\cP\Shv(L;\Spc))$, whose  sheafification $(\Phi_{\alg}, \Phi_{\sfmod})$ has its group completion isomorphic to the pair $((\bZ\times BO)_L, \gamma^{-1}(\eta_{U/O}))$, where $(\bZ\times BO)_L$ is the constant sheaf of commutative topological monoid $\bZ\times BO$.\\

\item[(b)] 
Let $\fG_L$ denote for the locally constant sheaf of  (stable $\infty$-)categories on $L$ defined by
\begin{align*}
&\Open(L)^{op}\longrightarrow \PrstL\\
&\cV\mapsto \Loc(\gamma^{-1}(\eta_{U/O})|_{\cV};\Sp)^{J\text{-equiv}},
\end{align*} 
where $\gamma^{-1}(\eta_{U/O})|_{\cV}$ means the total space as a module of $\bZ\times BO$. There is a canonical equivalence of locally constant sheaf of categories $\mu\cShv_L\simeq \fG_L$.

\end{prop}
\begin{proof}
\item[(a)] First, let $\cF^{pre}_{(G,M)}$ denote for the presheaf on $\QHam_L(U/O)$ defined by the functor $(G,M)$ in (\ref{diagram: main Kan}). For each object $(\cU,\cV,Q_\flat,G^\bullet, M^{\bullet,\dagg})\in \QHam_L(U/O)$, the pair $(G,M)$ specifies a commutative submonoid of $\Omega(U/O)$ acting on a space of paths in $U/O$ starting from any point in $\cU$ to the base point, i.e. a space of based loops acting on a space of sections of the fibration $\text{Path}_*(U/O)\rightarrow U/O$ over $\cU$.  This induces a map of presheaves on $\QHam_L(U/O)^{op}$
\begin{align}\label{eq: tilde tau}
\widetilde{\tau}: \cF_{(G,M)}^{pre}\longrightarrow (p_L^{op})^{-1}((\bZ\times BO)_L, \gamma^{-1}(\eta_{U/O})).
\end{align}
More explicitly, for each morphism 
\begin{align*}
\frj_{Q_{12}}: (\cU_1,\cV_1,Q^{(1)}_\flat,G^\bullet_{(1)}, M^{\bullet,\dagg}_{(1)})\rightarrow (\cU_2,\cV_2,Q^{(2)}_\flat,G^\bullet_{(2)}, M^{\bullet,\dagg}_{(2)}),
\end{align*}
in $\QHam_L(U/O)$, the following diagram
\begin{align*}
\xymatrix{(G_{(1)}, M_{(1)})\ar[r]\ar[d]_{\frj_{Q_{12}}}&(\Maps(\cV_1, \bZ\times BO), \Gamma(\cV_1,  \gamma^{-1}(\eta_{U/O})))\ar[d]\\
(G_{(2)}, M_{(2)})\ar[r]&(\Maps(\cV_2, \bZ\times BO), \Gamma(\cV_2,  \gamma^{-1}(\eta_{U/O})))
}
\end{align*}
is commutative up to a contractible space of homotopies between concatenated paths $\varphi^t_{Q^{(1)}_\flat}\bullet\FT_t(\ell_x)$ and $\varphi^t_{Q_{12}}\bullet\varphi^t_{Q_\flat^{(2)}}\bullet \FT_t(\ell_x)$ for each $x\in \cV$, relative to their endpoints. We require the space of homotopies happen in the space of time-dependent Hamiltonian flows of nonpositively stabilized quadratic functions in $\bp$, denoted by $\varphi_{Q_t}^t, 0\leq t\leq 1$ (so each gives a path $\varphi^t_{Q_t}\bullet \FT_t(\ell_x)$ with the endpoint graph like), such that the difference $\Delta Q=:\int_0^1Q_tdt-Q_\flat^{(1)}\geq 0$ and it satisfies the equality (\ref{eq: no rank contribution}) replacing $\widetilde{Q}$ (which means it has no contribution to any rank equalities as explained in Definition \ref{def: QHam(U/O)}), hence contractible.  It is clear that the space of homotopies are compatible with composition of morphisms in $\QHam_L(U/O)$. 

Now by the functoriality of left Kan extension along $p_L^{op}$, 
$\widetilde{\tau}$ (\ref{eq: tilde tau}) determines a morphism of presheaves 
\begin{align*}
\tau: (\Phi_{\alg}^{pre}, \Phi_{\sfmod}^{pre})\longrightarrow ((\bZ\times BO)_L,\gamma^{-1}(\eta_{U/O})). 
\end{align*}
To see that $\tau$ induces an isomorphism of pairs after sheafification and group completion, we just need to show that $\tau$ induces such isomorphisms at the stalk level. For each $x\in L$, the stalk of $(\Phi_{\alg}^{pre}, \Phi_{\sfmod}^{pre})$ at $x$ is isomorphic to 
\begin{align*}
\varinjlim\limits_{(\cU, \cV, Q_\flat, G^\bullet, M^{\bullet,\dagg})\in \fD(x)}(G, M)
\end{align*} 
where $\fD(x)$ is the same as in Lemma \ref{lemma: cofinal D(x)}, and it can be calculated by a cofinal sequence in $\fD(x)$ (cf. the proof of Lemma \ref{lemma: cofinal D(x)}). 
It is then easy to see that this is isomorphic to the pair $\coprod\limits_{n}BO(n)$ together with its torsor generated by the concatenated path $\rho_{(Q_\flat,\ell_x)}:=\varphi_{Q_\flat}^t\bullet \FT_t(\ell_x)$, for any $Q_\flat$ satisfying $\varphi_{Q_\flat}(\FT(\ell_x))$ is of graph type and has no negative spectral part (here $\FT_t$ is as in (\ref{eq: FT_t})), and the induced morphism
\begin{align*}
\tau_x: &\varinjlim\limits_{(\cU, \cV, Q_\flat, G^\bullet, M^{\bullet,\dagg})\in \fD(x)}(G,M)\simeq (\coprod\limits_{n}BO(n), \coprod\limits_{n}BO(n)\lng\rho_{(Q_\flat, \ell_x)}\rng)\\
&\longrightarrow (\bZ\times BO, \eta_{U/O}|_{\ell_x})
\end{align*} 
becomes a weak homotopy equivalence after taking group completion on the left-hand-side. \\

\item[(b)]In the following, we view $\fG_L$ as a cosheaf by taking the left adjoint of the restriction functors. 
Let $\fG_L^{pre}$ be the pre-cosheaf of categories on $L$ defined by
\begin{align*}
&\Open(L)\longrightarrow \PrstL\\
&\cV\mapsto \Loc(\Gamma(\cV, \gamma^{-1}(\eta_{U/O}));\Sp)^{J\text{-equiv}}.
\end{align*}
Then it is clear that $\fG_L$ is a cosheafification of $\fG_L^{pre}$, and in particular there is a canonical functor of pre-cosheaves $\fG_L\rightarrow \fG_L^{pre}$. 
Let $\fF^{pre}_L$ denote for the pre-cosheaf on $\QHam_L(U/O)^{op}$ defined by the functor $\Loc(-;\Sp)^{J\text{-equiv}}$ in (\ref{diagram: main Kan}). 
From part (a), we see that there is a canonical functor of pre-cosheaves of categories 
\begin{align*}
\fG_L\longrightarrow \fG_L^{pre} \longrightarrow (p_L)_*\fF^{pre}_L\simeq (\RKan_{p_L}(\Loc(-;\Sp)^{J\text{-equiv}})).
\end{align*}

In the following, we will view $\mu\cShv_L$ as a cosheaf of categories. The functor $F_L$ in (\ref{diagram: main Kan}) determines a canonical equivalence 
\begin{align*}
p_L^{-1}\mu\cShv_L\overset{\sim}{\longrightarrow} \fF_L^{pre}. 
\end{align*}
and therefore we get a functor of pre-cosheaves
\begin{align}\label{eq: G_L, RKan}
\fG_L\longrightarrow (p_L)_*(\fF_L^{pre})\overset{\sim}{\longrightarrow} (p_L)_*p_L^{-1}(\mu\cShv_L)\simeq \RKan_{p_L}(ev_0\circ F_L). 
\end{align}

Let $pt_{\QHam_L(U/O)}$ denote for the constant functor $\QHam_L(U/O)\rightarrow \Spc$ that maps every object to $pt$. The left Kan extension $\LKan_{p_L^{op}}(pt_{\QHam_L(U/O)})$ determines a presheaf of spaces on $L$, denoted by $(p_L)_!(pt_{\QHam_L(U/O)})^{pre}$. Let $(p_L)_!pt_{\QHam_L(U/O)}$ be the sheafification of $((p_L)_!pt_{\QHam_L(U/O)})^{pre}$. Applying Lemma \ref{lemma: Loc_F^pre} and the functor in (\ref{eq: G_L, RKan}), we get a functor of cosheaves: 
\begin{align}\label{eq: muSh tensor}
\fG_L\longrightarrow \mu\cShv_L\otimes \cLoc_{((p_L)_!pt_{\QHam_L(U/O)})}, 
\end{align}
 where $\cLoc_{((p_L)_!pt_{\QHam_L(U/O)})}$ is the cosheaf of categories on $L$ defined as in Lemma \ref{lemma: Loc_F^pre}. Since $\mu\cShv_L$ is locally constant with cofiber equivalent to $\Sp$, we can tensor with the dual 
 local system of categories of $\mu\cShv_L$, denoted by $\mu\cShv_L^{\vee}$ on both sides of (\ref{eq: muSh tensor}), and get a functor of cosheaves
 \begin{align}\label{eq: muSh tensor 2}
\fG_L\otimes \mu\cShv_L^{\vee}\longrightarrow \cLoc_{((p_L)_!pt_{\QHam_L(U/O)})}.
 \end{align}
Since the stalk of $((p_L)_!pt_{\QHam_L(U/O)})^{pre}$ at every $x\in L$ is contractible, $(p_L)_!pt_{\QHam_L(U/O)}$ is isomorphic to the constant sheaf on $L$ with fibers weakly homotopy equivalent to $pt$. Therefore, the right-hand-side in (\ref{eq: muSh tensor 2}) is equivalent to $\cLoc_L$. By comparing the costalks of both sides of (\ref{eq: muSh tensor 2}), as locally constant cosheaves, we see that (\ref{eq: muSh tensor 2}) is an equivalence, and this finishes the proof. 
\end{proof}

\subsection{The case for a general base manifold $X$}\label{subsec: general X}

For a general smooth manifold $X$, there is a standard treatment to reduce it to the $\bR^N$ case. One can choose a smooth embedding  $X\hookrightarrow \bR^N$ for some large $N$ and denote the associated embedding $X\times \bR_t\hookrightarrow \bR^N\times \bR_t$ by $\iota$. For any smooth (immersed) Lagrangian $L$ in $T^*X$, let $\widetilde{L}$ be its image under 
the canonical correspondence
\begin{align}\label{eq: f_d, f_pi}
T^*X\overset{f_d}{\longleftarrow} T^*\bR^N\underset{\bR^N}{\times}X\overset{f_\pi}{\longrightarrow} T^*\bR^N,
\end{align} 
which is a smooth (immersed) Lagrangian in $T^*\bR^N$. There is an obvious fibration 
\begin{align*}
\pi_{\widetilde{L}}: \widetilde{L}\rightarrow L, 
\end{align*}
whose fiber at $(x,\xi)\in L$ is a torsor over the fiber of the conormal bundle of $X$ in $\bR^N$ at $x$.

The  pushforward functor $\iota_*$ on microlocal sheaf categories induces an equivalence of sheaves of categories  
\begin{align*}
\pi_{\widetilde{L}}^{-1}\mu\cShv_L\longrightarrow\mu\cShv_{\widetilde{L}}.
\end{align*} 
So the classifying map for $\mu\cShv_L$ is homotopic to 
\begin{align*}
L\simeq \widetilde{L}\overset{\gamma_{\widetilde{L}}}{\longrightarrow} U/O\longrightarrow B\Pic(\bk). 
\end{align*}
It is not hard to see that the composition of the inverse of the homotopy equivalence $\pi_{\widetilde{L}}$ and $\gamma_{\widetilde{L}}$ is homotopic to the stable Gauss map $L\overset{\gamma_L}{\rightarrow} U/O$. Recall that the stable Gauss map for $L$ is induced from the trivialization of the bundle of stable Lagrangian Grassmannian over $T^*X$ by the section $\sigma_X$ of taking tangent spaces of cotangent fibers. 
Then we have a commutative diagram
\begin{align*}
\xymatrix{(\pi_{\widetilde{L}}^{-1}(\cLagGr_{T^*X}|_L), \pi_{\widetilde{L}}^{-1}(\sigma_X|_L))\ar[rr]^{\ \ \ \ \ \ Tf_\pi\circ(Tf_d)^{-1}}_\sim\ar[dr]&&(\cLagGr_{T^*\bR^N}|_{\widetilde{L}},\sigma_{\bR^N}|_{\widetilde{L}})\ar[dl]\\
&\widetilde{L}&
},
\end{align*}
where the top rows consist of pairs of principal $U/O$-bundles together with a reference section, and the connecting map is doing the Lagrangian correspondence (\ref{eq: f_d, f_pi}) on the tangent space level. Relating the tautological sections (induced from $L$ on the left) determines a homotopy between $\gamma_{\widetilde{L}}$ and $\gamma_L\circ\pi_{\widetilde{L}}$.

\begin{proof}[Proof of Theorem \ref{thm: sec proof}]
The theorem follows directly from Proposition \ref{prop: proof of thm} (b) and the discussion above. We remark that there are three places where the canonical involution $\vartheta$ on $U/O$ are involved, which account for the extra $\vartheta$ in (\ref{thm: sec proof}). 
First, we have taken the opposite Hamiltonian flow $\varphi_{-Q}$ for each quadratic form $Q$ in Section \ref{subsec: F_L}. Second, we have used the presentation of the universal principal $\Omega(U/O)$-bundle on $U/O$ as the space of paths ending at the base point, whose classifying map is $\vartheta\circ\gamma$.  These two $\vartheta$s cancel out. Third, similarly to the definition of the associated $G$-vector bundle for a principal $G$-bundle $\cP\rightarrow B$ and a $G$-representation $V$, $(\cP\times V)/G=\{(s,v)\sim (sg, g^{-1}v)\}$, the appearance of $g^{-1}$ in the second factor accounts for the $\vartheta$ in (\ref{eq: thm sec proof}). 
\end{proof}

\subsection{An application: proof of Corollary \ref{cor: AbKr}}
In this subsection, we give the sheaf-theoretic proof of Corollary \ref{cor: AbKr}, which was obtained by Abouzaid--Kragh \cite{AbKr} using Floer homotopy types. 
\begin{proof}[Proof of Corollary \ref{cor: AbKr}]
First, by the main result of \cite{Guillermou}, generalized over ring spectra by \cite{JiTr}, we have a fully faithful functor 
\begin{align}\label{eq: }
\cQ_L: \Gamma(L,\mu\cShv_{L;\bS})\longrightarrow \Loc(X;\bS).
\end{align}
Here we will apply a slightly modified version of $\cQ_L$ than the original one considered in \cite{Guillermou} and \cite{JiTr} in the following way. Recall that the original $\cQ_L$ is defined by the following steps
\begin{itemize}
\item[(i)] For each section of $\Gamma(L,\mu\cShv_{L;\bS})$, apply convolution with $\bk_{(-\delta,0]}$ for any of its local representatives as sheaves on $U\times (a,b)\subset X\times \bR$ with singular support in the Legendrian lifting $\bL$ of $L$, where $U$ runs over an open covering of $X$ and $\delta>0$ is sufficiently small. The outcome is a sheaf $F$ on $X\times\bR$ with singular support in the doubling $\bL\cup T_{-\delta}\bL$, whose microlocalization along $\bL$ is given by the section of $\Gamma(L,\mu\cShv_{L;\bS})$ that we started with. This gives a fully faithful functor
\begin{align*}
 T_{(-\delta,0]}: \Gamma(L,\mu\cShv_{L;\bS})\longrightarrow \Shv^{<0}_{\bL\cup T_{-\delta}\bL}(X\times\bR;\bS). 
\end{align*}

\item[(ii)] Do a homogeneous Hamiltonian flow on $T^{*,<0}(X\times \bR)$ (whose support projects to compact subset in $X\times\bR$) that shifts $T_{-\delta}\bL$ to $T_{-R}\bL$, $R\gg 0$ and is the identity near $\bL$. For $R$ sufficiently large, we have the front projection of $T_{-R}\bL$ is completely disjoint from that of $\bL$. 

\item[(iii)] By the main result of \cite{GKS}, the Hamiltonian flow gives an equivalence of categories
\begin{align}\label{eq: proof delta, R}
\Shv^{<0}_{\bL\cup T_{-\delta}\bL}(X\times\bR;\bS)\overset{\sim}{\longrightarrow} \Shv^{<0}_{\bL\cup T_{-R}\bL}(X\times\bR;\bS).
\end{align}
For $R\gg 0$, the restriction functor for $i_{(-R/2)}: X\times \{-R/2\}\hookrightarrow X\times\bR$ gives
\begin{align*}
i_{(-R/2)}^*: \Shv^{<0}_{\bL\cup T_{-R}\bL}(X\times\bR;\bS)\longrightarrow \Loc(X;\bS).
\end{align*}
Now compose $T_{(-\delta,0]}$, the equivalence (\ref{eq: proof delta, R}) and $i_{(-R/2)}^*$, and this defines $\cQ_L$. 
\end{itemize}

The changes that we will make to the above are as follows
\begin{itemize}
\item[(i)] We will use $\overset{!}{*}\omega_{[0,\delta)}$ instead of $*\bk_{(-\delta,0]}$, as introduced in Subsection \ref{subsubsec: second corr}. Then we get 
\begin{align*}
 T_{[0,\delta)}: \Gamma(L,\mu\cShv_{L;\bS})\longrightarrow \Shv^{<0}_{\bL\cup T_{\delta}\bL}(X\times\bR;\bS). 
\end{align*}
instead of $T_{(-\delta,0]}$.

\item[(ii)] Just change $T_{-\delta}\bL$ (resp. $T_{-R}\bL$) to $T_\delta\bL$ (resp. $T_{R}\bL$). 

\item[(iii)] The change for (\ref{eq: proof delta, R}) is the obvious one, and we change $i_{(-R/2)}^*$ to $i_{R/2}^!$. 
\end{itemize}
After these modifications, $\cQ_L$ remains to be fully faithful. 

Second, \cite{Guillermou} gives a microlocal sheaf-theoretic proof that the projection $\pi_L: L\rightarrow X$ is a homotopy equivalence. Fix a homotopy inverse of $\pi_L$ and denote it by $\psi_X$. 
Similarly to the consideration in \cite[Section 25]{Guillermou}, we can twist $\mu\cShv_{L;\bS}$ by its dual local system of categories, denoted by $\mu\cShv_{L;\bS}^\vee$, and get the functor 
\begin{align*}
\cQ^{tw}_L: \Loc(L;\bS)\simeq \Gamma(L,\mu\cShv_{L;\bS}\otimes \mu\cShv_{L;\bS}^\vee)\longrightarrow \Gamma(X, \psi_L^{-1}\mu\cShv_{L;\bS}^\vee)
\end{align*}
which is also fully faithful. Then to show that $\mu\cShv_{L;\bS}$ is trivial, we just need to show that for the dualizing sheaf $\omega_L\in \Loc(L;\bS)$, $\cN=:\cQ_L^{tw}(\omega_L)$ is locally (restricted to a contractible open) a compact object with endomorphism algebra isomorphic to $\bS$. The property of compactness is obvious. 

Note that $\cQ_L^{tw}$ is compatible with tensoring with local systems (i.e. a functor between modules over $\Loc(X;\bS)^{\otimes,!}$), namely for any $\cL\in \Loc(X;\bS)$ and any object $M\in \Loc(L;\bS)$, we have
\begin{align*}
&\cQ_L^{tw}(M\overset{!}{\otimes} \pi_L^!\cL)\simeq \cQ_L^{tw}(M)\overset{!}{\otimes} \cL.
\end{align*}
Let $\eta_X:\Path_*X\rightarrow X$ be the universal principal $\Omega X$-bundle, and let $\cL=(\eta_X)_!\omega_{\Path_*X}$, where $\omega_{\Path_*X}$ means the monoidal unit of $(\Loc(\Path_*X, \bS), \overset{!}{\otimes})$. Then we have 
\begin{align}\label{eq: proof omega, cL}
&\Hom_{\Loc(L;\bS)}(\omega_L\overset{!}{\otimes} \pi_L^!\cL, \omega_L)\simeq \Hom_{\Gamma(X,\psi_L^{-1}\mu\cShv_{L;\bS}^\vee)}( \cN\overset{!}{\otimes} \cL,  \cN). 
\end{align}
Since $\pi_L$ is a homotopy equivalence, we have a homotopy Cartesian diagram
\begin{align*}
\xymatrix{\Path_*L\ar[r]\ar[d]_{\eta_L}&\Path_*X\ar[d]^{\eta_X}\\
L\ar[r]^{\pi_L}&X
},
\end{align*}
and (\ref{eq: proof omega, cL}) is equivalent to 
\begin{align}\label{eq: pre adjunction}
\Hom_{\Loc(L;\bS)}((\eta_L)_!\eta_L^!\omega_L, \omega_L)\simeq \Hom_{\Gamma(X,\psi_L^{-1}\mu\cShv_{L;\bS}^\vee)}( (\eta_X)_!\eta_X^!\cN,  \cN).
\end{align}
Here we use the adjunction pair
\[
\begin{tikzcd}[arrow style=tikz,>=stealth,row sep=4em]
\Gamma(X,\psi_L^{-1}\mu\cShv_{L;\bS}^\vee)
 \arrow[rr, shift left=.4ex, "R=\eta_X^!"]
  && \Gamma(\Path_*X, \eta_X^{-1}\psi_L^{-1}\mu\cShv_{L;\bS}^\vee)\ar[ll, shift left=.4ex, "L=(\eta_X)_!"].
\end{tikzcd}
\]
By adjunction, we get from (\ref{eq: pre adjunction})
\begin{align*}
\bS\simeq \Hom_{\Loc(\Path_*L;\bS)}(\eta_L^!\omega_L,\eta_L^!\omega_L)\simeq \Hom_{\Gamma(\Path_*X; \eta_X^{-1}\psi_L^{-1}\mu\cShv_{L;\bS}^\vee)}(\eta_X^!\cN,\eta_X^!\cN).
\end{align*}
Since $\Path_*X$ is contractible, $\eta_X^{-1}\psi_L^{-1}\mu\cShv_{L;\bS}^\vee$ is trivial, so after choosing an identification 
\begin{align*}
\eta_X^{-1}\psi_L^{-1}\mu\cShv_{L;\bS}^\vee\simeq \cLoc_{\Path_*X},
\end{align*}
$\eta_X^!\cN$ is sent to a trivial local system with costalk isomorphic to $N\in \Mod(\bS)$. Then we get 
\begin{align*}
\End_\bS(N)\simeq \bS
\end{align*}
as desired. 

\end{proof}

\appendix

\section{The canonical functor $\Corr(\Fun^\diamond(N(\Fin_*), \cC))_{\inert, \all}\longrightarrow \CAlg(\bCorr(\cC^\times))^{\rightlax}$}\label{sec: Appendix}

Let $\cC$ be an $\infty$-category that admits finite products. Then the Cartesian symmetric monoidal structure on $\cC$ determines a canonical symmetric monoidal structure on $\bCorr(\cC)$. In \cite{Jin}, we gave several concrete constructions of commutative algebra objects, their modules, and (right-lax) morphisms among them using $\Fin_*$-objects in $\cC$ and correspondences among them, under easy-to-check conditions. In this appendix, we upgrade these constructions to a canonical functor from a correspondence category of a certain functor category $\Fun^\diamond(N(\Fin_*), \cC)$ to the category of commutative algebras in $\bCorr(\cC^\times)$ with right-lax morphisms. The main results are stated in Theorem \ref{thm: appendix}. With a completely similar argument, one can get a version for the category of pairs of commutative algebras and their modules in $\bCorr(\cC)$ and a version for associative algebras and their modules (cf. Theorem \ref{thm: appendix mod}).

\subsection{An explicit model for the Gray tensor product with $[n]$}
First, for any ordinary $1$-category $J$ we present an explicit model\footnote{We remark that there is a difference between our Gray tensor product $\bX\Gray \bY$ and the traditional definition up to a swap of $X$ and $Y$. The difference is caused by our convention to put $\bX$ vertically (as indexing rows) and $\bY$ horizontally (as indexing columns). In particular, we have 
\begin{align*}
&\Maps_{\TwoCat}(\bX\Gray \bY, \bW)\simeq \Maps_{\TwoCat}(\bY, \Fun(\bX, \bW)_{\rightlax}),\\
&\Maps_{\TwoCat}(\bY\Gray \bX, \bW)\simeq \Maps_{\TwoCat}(\bY, \Fun(\bX, \bW)_{\leftlax}),
\end{align*}
where $\Fun(\bX,\bW)_{\rightlax}$ (resp. $\Fun(\bX,\bW)_{\leftlax}$) consists of genuine functors between $(\infty,2)$-categories and has 1-morphisms right-lax (resp. left-lax) natural transformations.}  for $\Seq_\bullet(J\Gray [n])$, where $\Gray$ is the Gray tensor product of $(\infty,2)$-categories. 
For any $\lambda\in \Seq_k([n])$,
we will represent it by $0\leq \lambda(0)\leq \cdots\lambda(k)\leq n$. 
Then the 1-category $\Seq_k(J\Gray [n])\underset{\Seq_k([n])}{\times} \{\lambda\}$
has the same objects as 
\begin{align}\label{eq: obj Seq Gray}
\Seq_kJ\underset{\prod\limits_{0\leq i<k}\Seq_1(J)}{\times} \prod\limits_{0\leq i<k}\Seq_{\lambda(i+1)-\lambda(i)+1}(J),
\end{align}
where the morphism $\Seq_kJ\rightarrow \prod\limits_{0\leq i<k}\Seq_1(J)$ is induced from the $k$ inert morphisms $[1]\cong \{i-1, i\}\hookrightarrow [k]$, and the morphism $\Seq_{\lambda(i+1)-\lambda(i)+1}(J)\rightarrow \Seq_{1}J$ is induced from the unique active morphism $[1]\rightarrow [\lambda(i+1)-\lambda(i)+1]$. We will represent any object of $\Seq_k(J\Gray [n])\underset{\Seq_k([n])}{\times} \{\lambda\}$ by a pair $(\alpha, (\beta_i)_{0\leq i<k})$, using the identification with objects in (\ref{eq: obj Seq Gray}). 
For two objects $(\alpha^{(j)}, (\beta_{i}^{(j)})_{0\leq i<k})$, $j=0,1,$ in $\Seq_k(J\Gray [n])\underset{\Seq_k([n])}{\times} \{\lambda\}$, we have 
\begin{align*}
&\Maps((\alpha^{(0)}, (\beta_{i}^{(0)})), (\alpha^{(1)}, (\beta_{i}^{(1)})))\\
=&\begin{cases}\{\alpha^{(0)}\}\underset{\prod\limits_{0\leq i<k}\Maps([1]\times [1], J)}{\times}\prod\limits_{0\leq i<k}(\Maps([1]\times [\lambda(i),\lambda(i+1)+1], J)\\
\underset{\prod\limits_{j=0,1}\Maps(\{j\}\times [\lambda(i),\lambda(i+1)+1], J)}{\times} \{(\beta_{i}^{(0)},\beta_{i}^{(1)})\}),  \text{ if }\alpha^{(0)}=\alpha^{(1)};\\
\emptyset, \text{ otherwise},
\end{cases}
\end{align*}
where $\alpha^{(0)}$ is also viewed as the functor $[1]\times [k]\rightarrow [k]\overset{\alpha_k^{(0)}}{\rightarrow} J$ (see Figure \ref{figure: beta morphism}). 
\begin{figure}[h]
\begin{tikzpicture}
\draw[->] (0,0) node[left]{$\cdots$} node[below] {$\alpha(i-2)$}--(2,0) node[below] {$\alpha(i-1)$};
\draw[->, blue] (2,0)--(3,1) node[cross, blue]{}--(4,1) node[cross, blue]{} --(5,1) node[cross, blue]{}--(6,0) node[black, below] {$\alpha(i)$};
\draw (4,1) node [above, blue] {$\beta^{(0)}_{i-1}$}; 
\draw[->, blue]  (2,0)--(3,-1) node[cross] {}--(4,-1) node[cross] {}--(5,-1) node[cross] {}--(6,0);
\draw[blue] (4,-1) node [below] {$\beta_{i-1}^{(1)}$};
\draw[->, red] (3,1)--(3,-1);
\draw[->, red](4,1)--(4,-1); 
\draw[->, red](5,1)--(5,-1); 
\draw[->] (6,0)--(8,0) node[below] {$\alpha(i+1)$};
\draw[->, blue] (8,0)--(9,1) node[cross, blue]{}--(10,1) node[cross, blue]{} --(11,1) node[cross, blue]{}--(12,0) node[black, below] {$\alpha(i+2)$};
\draw (10,1) node [above, blue] {$\beta^{(0)}_{i+1}$}; 
\draw[->, blue]  (8,0)--(9,-1) node[cross] {}--(10,-1) node[cross] {}--(11,-1) node[cross] {}--(12,0);
\draw[blue] (10,-1) node [below] {$\beta_{i+1}^{(1)}$};
\draw[->, red] (9,1)--(9,-1);
\draw[->, red](10,1)--(10,-1); 
\draw[->, red](11,1)--(11,-1); 
\draw[->] (12,0)--(14,0) node[below] {$\alpha(i+3)$} node[right] {$\cdots$};
\end{tikzpicture}
\caption{}\label{figure: beta morphism}
\end{figure} 
Pictorially, each object is representing by a path decorated with two kinds of nodes in which the circle nodes indicate $\alpha(j)$ and the blue cross node indicates $\beta_i(s)$ as shown in Figure \ref{figure: beta path}.

\begin{figure}[h]
\begin{tikzpicture}
\draw[->] (5,6) circle (2pt)--(5,5) node [cross, blue] {}--(4,5)--(4,4) node[cross, blue]{}--(3,4) node[cross, blue] {}--(2,4)--(2,3) circle (2pt)--(2,2) circle (2pt)--(2,1) node[cross, blue] {}--(1,1) node[cross, blue] {}--(0,1) node[cross, blue] {}--(-1,1)--(-1,0) circle (2pt); 
\draw[->] (5,-1) node[cross, blue]{}--(4,-1) node[cross, blue]{}--(3,-1) node[cross, blue]{}--(2,-1) node [cross, blue]{}--(1,-1) node [cross, blue]{}--(0,-1) node[cross, blue]{}--(-1,-1) node[below] {$\lambda(i+2)$};
\draw (5,-1) node[below]{$\ \ \lambda(i-1)$};
\draw (4,-1) node[below]{$\lambda(i)$};
\draw (2,-1) node [below]{$\lambda(i+1)$};
\draw[->] (-2, 6) circle (2pt) node [left] {$\alpha(i-1)$}--(-2,3) circle (2pt) node[left] {$\alpha(i)$}--(-2, 2) circle (2pt) node[left] {$\alpha(i+1)$}--(-2,0) circle (2pt) node[left] {$\alpha(i+2)$};
\draw[->] (-2,0)--(-2,-0.5); 
\end{tikzpicture}
\caption{}\label{figure: beta path}
\end{figure}

Equivalently, we can view each $\beta_i$ as a functor 
\begin{align}\label{eq: beta_i, path}
([\lambda(i)-1,\lambda(i+1)]\times [\lambda(i),\lambda(i+1)])^{0\leq y-x\leq 1}\overset{(\beta_{i;J}, \beta_{i;[n]})}{\longrightarrow} J\times [n]
\end{align}
such that $\beta_{i;J}$ maps every horizontal morphism to an identity morphism in $J$, and $\beta_{i;[n]}$ is taking the projection to the second factor and then including it to $[n]$ in the obvious way.

We take 
\begin{align*}
\Seq_k(J\Gray [n])=\coprod\limits_{\lambda\in \Seq_k([n])}\Seq_k(J\Gray [n])\underset{\Seq_k([n])}{\times} \{\lambda\}
\end{align*}
The rules for the compositions are obvious. For any degenerate map $[k]\rightarrow [\ell]$ in $\Delta$, the functor $\Seq_\ell(J\Gray [n])\rightarrow \Seq_k(J\Gray [n])$ is the obvious one, and for any active map $f: [k]\rightarrow [\ell]$,  the functor $\Seq_\ell(J\Gray [n])\rightarrow \Seq_k(J\Gray [n])$ is defined by joining $(\beta_v)_{f(i)\leq v< f(i+1)}$ via the presentation (\ref{eq: beta_i, path}) into a single $\beta_{[f(i), f(j)]}$
\begin{align}\label{eq: beta_f(i,j)}
\beta_{[f(i), f(j)]}: ([\lambda(f(i))-1, \lambda(f(j))]\times [\lambda(f(i)),\lambda(f(j))])^{0\leq y-x\leq 1}\longrightarrow J\times [n].
\end{align} 
that removes all the joint circle nodes (see Figure \ref{figure: beta path}). In the following, we will also use the notation
\begin{align}\label{eq: beta_(i,j)+}
\beta_{[i, j]}^+: &([\lambda(i)_-, (\lambda(i+1)-1)_+]\times [\lambda(i),\lambda(i+1)])^{0\leq y-x\leq 1}\underset{\{(\lambda(i+1)-1/2, \lambda(i+1))\}}{\cup}\\
\nonumber&([\lambda(i+1)_-, (\lambda(i+2)-1)_+]\times [\lambda(i+1),\lambda(i+2)])^{0\leq y-x\leq 1}\underset{\{(\lambda(i+2)-1/2, \lambda(i+2))\}}{\cup}\cdots\\
\nonumber&\underset{\{(\lambda(j-1)-1/2, \lambda(j-1))\}}{\cup}([\lambda(j-1)_-, (\lambda(j)-1)_+]\times [\lambda(j-1),\lambda(j)])^{0\leq y-x\leq 1}\\
\nonumber&\longrightarrow J\times [n].
\end{align} 
to denote the joining of $(\beta_v)_{i\leq v<j}$ \emph{without} removing the joint circle nodes. 
Here a joint circle node is marked by $(x_-, y)$ (resp. $(x_+, y)$) if it has an outgoing edge (resp. incoming edge) connecting to (resp. from) $(x,y)$; the poset $[a_-, b_+]$ is equivalent to $[a-1, b+1]$ with $a-1$ and $b+1$ marked by $a_-$ and $b_+$ respectively, in particular the operation $y-x$ is defined by identifying $a_-$ (resp. $b_+$) with $a-1$ (resp. $b+1$); in the formation of the coproduct, we have
\begin{align*}
\{(\lambda(v)-1/2,\lambda(v))\}&\hookrightarrow ([\lambda(v-1)_-, (\lambda(v)-1)_+]\times [\lambda(v-1),\lambda(v)])^{0\leq y-x\leq 1}\\
(\lambda(v)-1/2,\lambda(v))&\mapsto ((\lambda(v)-1)_+, \lambda(v)),\\
\{(\lambda(v)-1/2,\lambda(v))\}&\hookrightarrow ([\lambda(v)_-, (\lambda(v+1)-1)_+]\times [\lambda(v),\lambda(v+1)])^{0\leq y-x\leq 1}\\
(\lambda(v)-1/2,\lambda(v))&\mapsto (\lambda(v)_-, \lambda(v)).\\
\end{align*}

For simplicity, by some abuse of notations, we will denote the domain of $\beta_{[i,j]}^+$ in (\ref{eq: beta_(i,j)+}) by 
\begin{align}\label{eq: domain beta}
([\lambda(i)_-, (\lambda(j)-1)_+]\widetilde{\times} [\lambda(j-1),\lambda(j)])^{0\leq y-x\leq 1}.
\end{align}
In a similar fashion, assume for each $x\in [\lambda(v)_-, (\lambda(v+1)-1)_+], i\leq v<j$ we assign a value $\kappa(x)\geq 0$, then we use 
\begin{align}\label{eq: domain beta, kappa}
([\lambda(i)_-, (\lambda(j)-1)_+]\widetilde{\times} [\lambda(j-1),\lambda(j)])^{0\leq y-x\leq \kappa(x)}
\end{align}
to denote the coproduct as in the domain of $\beta_{[i,j]}^+$ (\ref{eq: beta_(i,j)+}) with $0\leq y-x\leq 1$ replaced by $0\leq y-x\leq \kappa(x)$. We also set
\begin{align*}
\beta_{[i,j];J}^+=\proj_{J}\circ \beta_{[i,j]}^+. 
\end{align*}

We remark that the above model of $\Seq_k(J\Gray [n])$ also works for any $\infty$-category $J$.  

\subsection{Definition of the canonical functor}\label{subsec: def canonical}
Define $\Fun^\diamond(N(\Fin_*),\cC)$ to be the full subcategory of $\Fun(N(\Fin_*),\cC)$ consisting of $\Fin_*$-objects satisfying the condition in \cite[Theorem 2.6]{Jin} (it was denoted by $\Fun'(N(\Fin_*),\cC)$ in \emph{loc. cit.}). Let $\inert$ (resp. $\mathsf{active}$) be the class of morphisms satisfying that the diagrams (2.4.9) (resp. (2.4.11)) in \emph{loc. cit.} are Cartesian squares. Then 
$\Seq_n(\Corr(\Fun^\diamond(N(\Fin_*),\cC^\times)))_{\inert, \all}$ is the full subcategory of 
\begin{align*}
\Maps(([n]\times [n]^{op})^{\geq \dgnl}, \Fun(N(\Fin_*), \cC)) 
\end{align*}
consisting of functors $F$ satisfying 
\begin{itemize}
\item[(i)] Each square 
\begin{align*}
\xymatrix{F(i,j+1)\ar[r]\ar[d]&F(i,j)\ar[d]\\
F(i+1,j+1)\ar[r]&F(i+1,j)
}
\end{align*}
is a Cartesian square. 

\item[(ii)] For each $(i,j)\in ([n]\times [n]^{op}$), $F(i,j)$ as a $\Fin_*$-object in $\cC$ satisfies the conditions in \cite[Theorem 2.6 (ii)]{Jin}, that is it defines a commutative algebra object in $\Corr(\cC)$;

\item[(iii)] For each fixed $j\in [n]^{op}$, the vertical morphisms between $\Fin_*$-objects satisfy the condition that the diagram (2.4.9) in \cite{Jin} is Cartesian for every $n$.  
\end{itemize}

We are going to construct a canonical functor in $\Spc^{\Delta^{op}}$
\begin{align}\label{eq: P_C, bullet}
P_{\cC^\times,\bullet}: &\Seq_\bullet(\Corr(\Fun^\diamond(N(\Fin_*),\cC^\times)))_{\inert, \all}\longrightarrow \\
\nonumber&\Maps_{(\OneCat)^{\Delta^{op}}_{/\Seq_{\clubsuit}(N(\Fin_*))}}(\Seq_{\clubsuit}(N(\Fin_*)\Gray[\bullet]), \Seq_\clubsuit(\Corr(\cC^\times)^{\otimes, \Fin_*})).
\end{align}

\subsubsection{Step 1: the case for $\bullet=n$}
For any ordinary 2-category (including 1-categories) $K$, let $I_K\rightarrow N(\Delta)^{op}$ be the coCartesian fibration from the Grothendieck  construction of the functor $\Seq_\bullet: N(\Delta)^{op}\rightarrow \OneCat^{\ord}$.

For any $([k],(\alpha, (\beta_i), \lambda))\in I_{N(\Fin_*)\Gray [n]}$, let 
\begin{align}\label{eq
: pi_[k], alpha}
\pi_{([k],\alpha,(\beta_i), \lambda)}:\Gamma_{([k],\alpha, (\beta_i), \lambda)}\longrightarrow([k]\times [k]^{op})^{\geq \dgnl}
\end{align}
be the Cartesian fibration from the Grothendieck construction for the natural functor
\begin{align*}
(([k]\times [k]^{op})^{\geq \dgnl})^{op}&\longrightarrow \OneCat^{\ord}\\
(i,j)&\mapsto 
([\lambda(i)_-, (\lambda(j)-1)_+]\widetilde{\times} [\lambda(i),\lambda(j)]^{op})^{(0\leq y-x\leq \kappa_{\beta,\lambda}(x))}
\end{align*}
(the right-hand-side was introduced in (\ref{eq: domain beta, kappa})), where for each $x\in [\lambda(v)_-, (\lambda(v+1)-1)_+]$, we define $\kappa_{\beta,\lambda}(x)=1$ (resp. $0$) if $x=\lambda(v)_-$ (resp. $x=(\lambda(v+1)-1)_+$) and in the other cases $\kappa_{\beta,\lambda}(x)$ is the largest integer $r$ such that the restriction 
\begin{align*}
\beta^+_{[v,v+1]; N(\Fin_*)}: ([x, x+r-1]\times [x, x+r])^{0\leq y_2-y_1\leq 1}\rightarrow N(\Fin_*)
\end{align*}
is a constant functor ($\beta^+_{[v,v+1]}$ was introduced in (\ref{eq: beta_(i,j)+}) with $J=N(\Fin_*)$). For any morphism $(i_0,j_0)\rightarrow (i_1,j_1)$ in $(([k]\times [k]^{op})^{ \geq \dgnl})^{op}$, i.e. $i_0\geq i_1,\ j_0\leq j_1$, the corresponding functor  
is uniquely determined by the natural inclusion. For any $x\in [\lambda(i)_-, (\lambda(j)-1)_+]$, let 
\begin{align}\label{eq: def p_n (1)}
x^+= \begin{cases}x,\text{ if }x\neq \lambda(v)_-, (\lambda(v+1)-1)_+\text{ for any  }i\leq v<j;\\
\lambda(v), \text{ if }x=\lambda(v)_-;\\
\lambda(v+1), \text{ if }x=(\lambda(v+1)-1)_+.
\end{cases}
\end{align}
Define
\begin{align}\label{F_[k],alpha}
F_{([k],\alpha,(\beta_i),\lambda)}:& \Gamma_{([k],\alpha, (\beta_i), \lambda)}\longrightarrow ([n]\times[n]^{op})^{\geq \dgnl}\\
\nonumber&(i,j;x,y)\mapsto(x^+,y)
\end{align}
that takes the Cartesian arrows $(i_1,j_1; x,y)\rightarrow (i_0,j_0;x,y)$ to the identity morphisms. 

Let $T^\Comm$ be the ordinary 1-category defined by
\begin{itemize}
\item Objects: pairs $(\lng n\rng, j \in\lng n\rng^\circ), \lng n\rng\in N(\Fin_*)$,
\item Morphisms: a morphism $(\lng n\rng, j)\rightarrow (\lng m\rng, k)$ is given by a morphism $\alpha: \lng m\rng\rightarrow \lng n\rng$ in $N(\Fin_*)$ satisfying the condition $\alpha(k)=j$.
\end{itemize}
There is a natural projection $T^\Comm\rightarrow N(\Fin_*)^{op}$. See \cite[Subsection 2.2.1]{Jin} for the role of $T^\Comm$ (and its associative version $T$) in the definition of a Cartesian fibration $\cC^{\times, (\Fin_*)^{op}}\rightarrow N(\Fin_*)^{op}$ that exhibiting the Cartesian symmetric monoidal (and plain monoidal) structure of $\cC$. Recall the definition of a Cartesian fibration $\cT^\Comm\longrightarrow I_{N(\Fin_*)}$ in \emph{loc. cit.} determined by the functor
\begin{align*}
I_{N(\Fin_*)}^{op}&\longrightarrow (\OneCat)^{\ord}\\
([k], \alpha)&\mapsto ([k]\times[k]^{op})^{\geq \dgnl}\underset{N(\Fin_*)^{op}}{\times}T^\Comm,
\end{align*}
where on the right-hand-side the functor $([k]\times[k]^{op})^{\geq \dgnl}\rightarrow N(\Fin_*)^{op}$ is given by first projecting to $[k]^{op}$ and then applying $\alpha^{op}$. 

Consider the natural functor 
\begin{align*}
I_{N(\Fin_*)\Gray [n]}^{op}&\longrightarrow \OneCat^{\ord}\\
([k],\alpha,(\beta_i), \lambda))&\mapsto \Gamma_{([k],\alpha, (\beta_i), \lambda)}\underset{N(\Fin_*)^{op}}{\times} T^{\Comm},
\end{align*}
where the functor $ \Gamma_{([k],\alpha, (\beta_i), \lambda)}\rightarrow N(\Fin_*)^{op}$ is from the composition of the projection to $[k]^{op}$ and applying $\alpha^{op}$, 
whose Grothendieck construction gives a Cartesian fibration
\begin{align}\label{eq: def cT^+}
\pi_{\cT^+,n}: \cT^+_{N(\Fin_*)\Gray [n]}\longrightarrow I_{N(
\Fin_*)\Gray [n]}.
\end{align}
It is clear that the functors $F_{([k],\alpha,(\beta_i),\lambda)}$ (\ref{F_[k],alpha}) induces a functor 
\begin{align*}
F_{\cT^+,n}:\cT^+_{N(\Fin_*)\Gray [n]}\longrightarrow ([n]\times [n]^{op})^{\geq \dgnl}. 
\end{align*}

Now for any $[n]\in N(\Delta)^{op}$, we have well defined functors
\begin{align}\label{diagram: p_n, q_n}
\xymatrix{
\cT^+_{N(\Fin_*)\Gray [n]}\ar[r]^{p_{n}=(p_n^{(1)}=F_{\cT^+,n}, p_n^{(2)})\ \ \ \ \ \ \ \ \ } \ar[d]_{q_n}&([n]\times [n]^{op})^{\geq \dgnl}\times N(\Fin_*)\\
I_{N(\Fin_*)\Gray[n]}\underset{I_{N(\Fin_*)}}{\times}\cT^\Comm
},
\end{align}
where $p_n^{(2)}$ is the canonical functor that sends
\begin{align*}
&([k], \alpha,(\beta_i),\lambda; (i_0,j_0; x,y); u\in \alpha(j_0)^\circ)\mapsto (\beta_{[i_0, j_0];N(\Fin_*)}^+(x\rightarrow (\lambda(j_0)-1)_+))^{-1}(u)\sqcup \{*\}
\end{align*}
For any morphism in $([k],\alpha,(\beta_i),\lambda;\Gamma_{([k],\alpha,(\beta_i), \lambda)}\underset{N(\Fin_*)^{op}}{\times} T^\Comm)$ whose projection to $([k]\times [k]^{op})^{\geq \dgnl}$ is a horizontal (resp. vertical) morphism,  $p_n^{(2)}$ sends it to the obvious inert (resp. active) morphism. It is not hard to check (cf. \cite[Lemma 2.10]{Jin}) that $p_n^{(2)}$  canonically determines a functor from $\cT^+_{N(\Fin_*)\Gray [n]}$ to $N(\Fin_*). $ The functor $q_n$ is determined by the projections $\pi_{([k],\alpha,(\beta_i),\lambda)}$ (\ref{eq
: pi_[k], alpha}) and  it is a Cartesian fibration.

Diagram (\ref{diagram: p_n, q_n}) defines a sequence of canonical functors
\begin{align*}
&\Maps(([n]\times [n]^{op})^{\geq \dgnl}, \Fun(N(\Fin_*),\cC))\overset{\circ p_n}{\longrightarrow} \Maps(\cT^+_{N(\Fin_*)\Gray [n]}, \cC)\\
&\overset{\RKan_{q_n}}{\longrightarrow}\Maps(I_{N(\Fin_*)\Gray[n]}\underset{I_{N(\Fin_*)}}{\times}\cT^\Comm, \cC),
\end{align*}
whose essential image lies in the full subcategory 
\begin{align*}
&\Maps_{I_{N(\Fin_*)}}(I_{N(\Fin_*)\Gray[n]}^\natural\underset{I_{N(\Fin_*)}}{\times}(\cT^\Comm)^\natural, (\cC\times I_{N(\Fin_*)})^\natural)\\
\simeq &\Maps_{(\OneCat)^{\Delta^{op}}_{/\Seq_\clubsuit (N(\Fin_*))}}(\Seq_\clubsuit(N(\Fin_*)\Gray [n]), f_{\Delta^\clubsuit}^{\cC,\Comm}).
\end{align*}
Here $I_{N(\Fin_*)\Gray[n]}^\natural$ and $\cC\times I_{N(\Fin_*)}$ (resp. $(\cT^\Comm)^\natural$) are marked by the coCartesian (resp. Cartesian) edges over $I_{N(\Fin_*)}$, and $f_{\Delta^\clubsuit}^{\cC,\Comm}$ is defined in \cite[(2.3.2)]{Jin} with $\cC^{\times, \Delta}$ replaced by $\cC^{\times, \Fin_*^{op}}$.

Now we take the composite functor of the above and restrict it to the full subcategory 
\begin{align*}
\Seq_n(\Corr(\Fun^\diamond(N(\Fin_*),\cC))_{\inert, \all})\subset \Maps(([n]\times [n]^{op})^{\geq \dgnl}, \Fun(N(\Fin_*),\cC))\
\end{align*}
(see the beginning of this subsection). It is straightforward to check that the functor canonically factors through the full embedding
\begin{align*}
&\Maps_{(\OneCat)^{\Delta^{op}}_{/\Seq_\clubsuit (N(\Fin_*))}}(\Seq_\clubsuit(N(\Fin_*)\Gray [n]), \Seq_\clubsuit(\bCorr(\cC)^{\otimes, \Fin_*}))\\
&\hookrightarrow \Maps_{(\OneCat)^{\Delta^{op}}_{/\Seq_\clubsuit (N(\Fin_*))}}(\Seq_\clubsuit(N(\Fin_*)\Gray [n]), f_{\Delta^\clubsuit}^{\cC,\Comm}).
\end{align*}
Then we let 
\begin{align*}
P_{\cC^\times, n}: &\Seq_n(\Corr(\Fun^\diamond(N(\Fin_*),\cC))_{\inert, \all})\longrightarrow \\&\Maps_{(\OneCat)^{\Delta^{op}}_{/\Seq_\clubsuit (N(\Fin_*))}}(\Seq_\clubsuit(N(\Fin_*)\Gray [n]), \Seq_\clubsuit(\bCorr(\cC)^{\otimes, \Fin_*}))
\end{align*}
be the resulting functor. 

\subsubsection{Step 2: Definition of $P_{\cC^\times, \bullet}$}
We just need to show the functoriality of $P_{\cC^\times, n}$ with respect to $n$. 

In the following, with some abuse of notation, let 
\begin{align*}
\beta_{i;J}: [\lambda(i)_-,(\lambda(i+1)-1)_+]\rightarrow J
\end{align*}
replace for $\beta_{i;J}$ defined in (\ref{eq: beta_i, path}), which contains the same amount of information. 
For any $\sigma: [n]\rightarrow [m]$ and $\lambda: [k]\rightarrow [n]$  in $N(\Delta)$ and 
$\beta_{i,J}$ as above, define 
\begin{align*}
\sigma_*\beta_{i;J}: [(\sigma\circ \lambda(i))_-,(\sigma\circ \lambda(i+1)-1)_+]\rightarrow J
\end{align*}
 as follows. 
For each $s\in [(\sigma\circ \lambda(i))_-,(\sigma\circ \lambda(i+1)-1)_+]$, let
\begin{align*}
r_i(s)=\begin{cases}
&\lambda(i)_-,\text{ if }s=(\sigma\circ\lambda(i))_-,\\
&\max\{v\in [\lambda(i),\lambda(i+1)-1]: \sigma(v)\leq s\}, \text{ if } \sigma\circ \lambda(i)\leq s< \sigma\circ\lambda(i+1),\\
&(\lambda(i+1)-1)_+, \text{ if }s=(\sigma\circ \lambda(i+1)-1)_+.
\end{cases}
\end{align*}
Then for each $s_1\leq s_2$ in $[\sigma\circ \lambda(i)_-,(\sigma\circ \lambda(i+1)-1)_+]$, it 
 is sent via $\sigma_*\beta_{i;J}$ to the composite morphism of $\beta_{i;J}|_{[r_i(s_1), r_i(s_2)]}$.

Now we can give a formula for the natural morphism in $\OneCat^{\Delta^{op}}$ associated to $\sigma: [n]\rightarrow[m]$
\begin{align}\label{eq: Seq, sigma}
\nonumber\Seq_\bullet(N(\Fin_*)\Gray [n])&\longrightarrow \Seq_\bullet(N(\Fin_*)\Gray [m])\\
([k],\alpha, (\beta_i),\lambda)\mapsto &([k], \sigma_*\alpha=\alpha, (\sigma_*\beta_i), \sigma\circ \lambda). 
\end{align}
 The formula extends naturally to a commutative diagram of functors between ordinary 1-categories 
\begin{align}\label{diagram: q_n, f_T+}
\xymatrix{
\cT^+_{N(\Fin_*)\Gray[n]}\ar[r]^{f_{\cT^+; \sigma}}\ar[d]_{q_n}&\cT^+_{N(\Fin_*)\Gray[m]}\ar[d]^{q_m}\\
I_{N(\Fin_*)\Gray[n]}\underset{I_{N(\Fin_*)}}{\times}\cT^\Comm\ar[r]^{f_{I;\sigma}}&I_{N(\Fin_*)\Gray[m]}\underset{I_{N(\Fin_*)}}{\times}\cT^\Comm
},
\end{align}
where $f_{I;\sigma}$ is the combination of (\ref{eq: Seq, sigma}
) and the identity functor on $\cT^\Comm$, and $f_{\cT^+,\sigma}$ is defined as follows. 
\begin{itemize}
\item[(i)] 
The functor $f_{\cT^+,\sigma}$ sends  each fiber of $q_n$ over $(([k],\alpha,(\beta_\nu), \lambda; i,j; u\in \alpha(j)^\circ)$ to the fiber of $q_m$ over $([k],\sigma_*\alpha, (\sigma_*\beta_\nu), \sigma\circ \lambda; i,j;u\in \sigma_*\alpha(j)^\circ=\alpha(j)^\circ)$: 
\begin{align*}
([k],\alpha,(\beta_\nu), \lambda; i,j; x,y;u\in \alpha(j)^\circ)\mapsto ([k],\sigma_*\alpha, (\sigma_*\beta_\nu), \sigma\circ \lambda; i,j; \hat{\sigma}_{\lambda}(x),\sigma(y); u\in \sigma_*\alpha(j)^\circ=\alpha(j)^\circ),
\end{align*}
where 
\begin{align*}
\hat{\sigma}_\lambda(x)=\begin{cases}&\sigma(x), \text{ if }x\in (\lambda(v)_-, (\lambda(v+1)-1)_+)\text{ and }\sigma(x)< \sigma\circ\lambda(v+1)\\
&(\sigma\circ\lambda(v))_-, \text{ if }x=\lambda(v)_-,\\
&(\sigma\circ\lambda(v+1)-1)_+, \text{ if }x=(\lambda(v+1)-1)_+\text{ or }\sigma(x)= \sigma\circ\lambda(v+1).
\end{cases}
\end{align*}

\item[(ii)] For each morphism in $I_{N(\Fin_*)\Gray[n]}\underset{I_{N(\Fin_*)}}{\times} \cT^\Comm$ induced from $c: [\ell]\rightarrow [k]$ in $N(\Delta)$: 
\begin{align*}
([k],\alpha, (\beta_\nu),\lambda; i,j; u\in \alpha(j)^\circ)\rightarrow ([\ell], c^*\alpha, (\beta'_\mu), \lambda\circ c; i',j'; u'\in \alpha(c(j'))^\circ) 
\end{align*}
with a morphism $c^*(\beta_\nu)\rightarrow (\beta'_\mu)$, $i\leq c(i')\leq c(j')\leq j$ and $(\alpha(c(j')\rightarrow j))(u')=u$, any Cartesian morphism over it is of the form
\begin{align}\label{eq: Cart q_n}
([k],\alpha, (\beta_\nu),\lambda; i,j;x,y; u\in \alpha(j)^\circ)\rightarrow ([\ell], c^*\alpha, (\beta'_\mu), \lambda\circ c; i',j'; x,y; u'\in \alpha(c(j'))^\circ),
\end{align}
where $x\in [\lambda\circ c(i')_-, (\lambda\circ c(j')-1)_+]$ and $y\in [\lambda\circ c(i'),\lambda\circ c(j')]$, and 
if $x=\lambda\circ c(\gamma)_-$ or $(\lambda\circ c(\gamma+1)-1)_+$ on the right-hand-side, then it is understood as $x=\lambda(c(\gamma))_-$ or $(\lambda(c(\gamma+1)-1)_+$ respectively on the left-hand-side. The functor $f_{\cT_+,\sigma}$ sends the morphism (\ref{eq: Cart q_n}) to 
\begin{align}\label{eq: image of Cart}
&([k],\sigma_*\alpha, (\sigma_*\beta_\nu),\sigma\circ\lambda; i,j; \hat{\sigma}_\lambda(x),\sigma(y); u\in \alpha(j)^\circ)\rightarrow\\
\nonumber& ([\ell], \sigma_*(c^*\alpha), (\sigma_*\beta'_\mu), \sigma\circ\lambda\circ c; i',j';\hat{\sigma}_{\lambda\circ c}(x),\sigma(y); u'\in \alpha(c(j'))^\circ). 
\end{align}
Note that (\ref{eq: image of Cart}) is in general \emph{not} a  Cartesian morphism, for example if there exists $c(v)<q<c(v+1)$ such that $\lambda\circ c(v)\leq \lambda(q-1)< \lambda(q)<\lambda(q+1)\leq \lambda\circ c(v+1)$ and $\sigma(\lambda(q)-1)=\sigma(\lambda(q))$, then 
\begin{align*}
&\hat{\sigma}_\lambda(\lambda(q)-1)=\sigma\circ\lambda(q)_-,\\
&\hat{\sigma}_{\lambda\circ c}(\lambda(q)-1)= \sigma\circ \lambda(q),
\end{align*}
but there is always a unique arrow $\hat{\sigma}_{\lambda}(x)\rightarrow \hat{\sigma}_{\lambda\circ c}(x)$ that makes (\ref{eq: image of Cart}) well defined. The assignment (\ref{eq: image of Cart}) is clearly well behaved under compositions of Cartesian morphisms over $q_n$. 

\item[(iii)] According to \cite[Lemma 2.10]{Jin}, to see that the data (i) and (ii) give a well defined $f_{\cT_+, \sigma}$, we just need to check that for each commutative diagram
\begin{align*}
\xymatrix{(\substack{[k],\alpha, (\beta_\nu),\lambda; i,j;\\x,y; u\in \alpha(j)^\circ})\ar[r]\ar[d]& (\substack{[\ell], c^*\alpha, (\beta'_\mu), \lambda\circ c; i',j';\\
 x,y; u'\in \alpha(c(j'))^\circ})\ar[d]\\
 (\substack{[k],\alpha, (\beta_\nu),\lambda; i,j;\\z,w; u\in \alpha(j)^\circ})\ar[r]&(\substack{[\ell], c^*\alpha, (\beta'_\mu), \lambda\circ c; i',j';\\
 z,w; u'\in \alpha(c(j'))^\circ})
}
\end{align*}
with $x\leq z\leq w\leq y$, 
we have the following diagram commute
\begin{align*}
\xymatrix{(\substack{[k],\sigma_*\alpha, (\sigma_*\beta_\nu),\lambda; i,j; \\
\hat{\sigma}_\lambda(x),\sigma(y); u\in \alpha(j)^\circ})\ar[r] \ar[d]&(\substack{[\ell], \sigma_*(c^*\alpha), (\sigma_*\beta'_\mu), \lambda\circ c; i',j';\\
\hat{\sigma}_{\lambda\circ c}(x),\sigma(y); u'\in \alpha(c(j'))^\circ})\ar[d]\\
(\substack{[k],\sigma_*\alpha, (\sigma_*\beta_\nu),\lambda; i,j; \\
\hat{\sigma}_\lambda(z),\sigma(w); u\in \alpha(j)^\circ})\ar[r]&(\substack{[\ell], \sigma_*(c^*\alpha), (\sigma_*\beta'_\mu), \lambda\circ c; i',j';\\
\hat{\sigma}_{\lambda\circ c}(z),\sigma(w); u'\in \alpha(c(j'))^\circ})
},
\end{align*}
but this is obvious. 
\end{itemize}
The diagrams (\ref{diagram: q_n, f_T+}) for different morphisms in $N(\Delta)$ determines a functor
\begin{align}\label{eq: q_bullet}
q^\bullet: \cT_{N(\Fin_*)\Gray[\bullet]}\longrightarrow I_{N(\Fin_*)\Gray[\bullet]}\underset{I_{N(\Fin_*)}}{\times}\cT^\Comm
\end{align}
in $(\OneCat^\ord)^{\Delta}$.

Now we have the following diagrams of functors (between ordinary 1-categories) and natural transformations
\begin{equation}\label{diagram: sigma, NT}
\begin{tikzpicture}[baseline=(current  bounding  box.center)]
\node (a) at (0,0) {$\cT^+_{N(\Fin_*)\Gray[n]}$};
\node (b) at (10,0) {$([n]\times [n]^{op})^{\geq \dgnl}\times N(\Fin_*)$};
\node (c) at (0,-3) {$\cT^+_{N(\Fin_*)\Gray[m]}$};
\node (d) at (10, -3) {$([m]\times [m]^{op})^{\geq \dgnl}\times N(\Fin_*)$}; 
\draw[->] (a) to [] node (f) [above] {$p_n$} (b);
\draw[->] (c) to [] node (ff) [above] {$p_m$} (d);
\draw[->] (a) to [] node (g) [left] {$f_{\cT^+;\sigma}$} (c);
\draw[->] (b) to [] node (gg) [right] {$(\sigma, \sigma^{op};id)$} (d);  
\draw[-{Implies},double distance=2.5pt,shorten >=2pt,shorten <=2pt] (b) to node[] [above] {$\delta_\sigma$} (c);
\end{tikzpicture}
\end{equation}
where for each object $A=([k], \alpha, (\beta_\nu), \lambda; i,j; x,y;u\in\alpha(j)^\circ)$, the corresponding morphism
\begin{align*}
\delta_\sigma(A):& (\sigma, \sigma^{op};id)\circ p_n(A)=((\sigma,\sigma^{op})(F_{[k],\alpha, (\beta_\nu),\lambda}(x,y)); \beta^+_{[i,j]}(x\rightarrow (\lambda(j)-1)_+)^{-1}(u)\cup\{*\})\\
&\rightarrow p_m\circ f_{\cT^+;\sigma}(A)= (F_{[k],\alpha, (\sigma_*\beta_\nu),\sigma\circ\lambda}(\hat{\sigma}_\lambda(x),\sigma(y)); (\sigma_*\beta)_{[i,j]}^+(\hat{\sigma}_\lambda(x)\rightarrow (\sigma(\lambda(j))-1)_+)^{-1}(u)\cup\{*\})
\end{align*}
is the identity morphism in the first factor and is the obvious active map in the second factor. It is straightforward to check that (\ref{diagram: sigma, NT}) for different morphisms in $N(\Delta)$ assemble to be a functor 
\begin{align}\label{eq: predelta, Fin_*}
N(\Delta)\Gray [1]\longrightarrow \bOneCat^{\ord}.
\end{align}

For any $\infty$-category $\cC$ that admits finite products, the composition of (\ref{eq: predelta, Fin_*}) with $\Fun(-,\cC)$ gives a functor 
\begin{align}\label{eq: delta, Fin_*}
p_{N(\Delta)^{op};\cC}: [1]\Gray N(\Delta)^{op}\longrightarrow \bOneCat
\end{align}
 representing a left-lax natural transformation between two functors $N(\Delta)^{op}\rightarrow \bOneCat$
\begin{align*}
\Fun([\bullet]\times [\bullet]^{op})^{\geq \dgnl}, \cC)\xdashrightarrow{p_{\bullet;\cC}} \Fun(\cT^+_{N(\Fin_*)\Gray[\bullet]},\cC).
\end{align*}
Composing with the right Kan extension along $q^\bullet$ (\ref{eq: q_bullet}), we get another 
left-lax natural transformation
\begin{align}\label{eq: hat p_C}
\Fun(([\bullet]\times [\bullet]^{op})^{\geq \dgnl}, \cC)\xdashrightarrow{\hat{p}_{\bullet;\cC}} \Fun(I_{N(\Fin_*)\Gray[\bullet]}\underset{I_{N(\Fin_*)}}{\times} \cT^\Comm,\cC).
\end{align}

\begin{lemma}\label{lemma: hat p, genuine}
In fact $\hat{p}_{\bullet, \cC}$ (\ref{eq: hat p_C}) is a genuine natural transformation. 
\end{lemma}
\begin{proof}
It suffices to check the following: given a functor 
\begin{align*}
Y: ([m]\times [m]^{op})^{\geq \dgnl}\times N(\Fin_*)\rightarrow \cC
\end{align*}
and $\sigma: [n]\rightarrow [m]$ in $N(\Delta)$, the induced morphism from $\delta_\sigma$ on limits
\begin{align*}
&\varprojlim\limits_{(x,y)\in \Gamma_{[k],\alpha,(\beta_\nu),\lambda}|_{(i,j)}}Y(\sigma(x^+), \sigma(y);\beta_{[i,j]}^+(x\rightarrow (\lambda(j)-1)_+)^{-1}(u)\cup\{*\})\longrightarrow \\
&\varprojlim\limits_{(x,y)\in \Gamma_{[k],\alpha,(\beta_\nu),\lambda}|_{(i,j)}}Y(\sigma(x^+), \sigma(y); \sigma_*\beta_{[i,j]}^+(\hat{\sigma}_\lambda(x)\rightarrow (\sigma\circ\lambda(j)-1)_+)^{-1}(u)\cup\{*\})
\end{align*}
is an isomorphism, for all $0\leq i\leq j\leq k$ and  $u\in \alpha(j)^\circ$. Due to the functoriality on $\sigma$, one can factor $\sigma$ as a composition of face maps and degeneracy maps, and reduce to the cases when $\sigma$ is either a face map or a degeneracy map, which can be checked directly. 
\end{proof}

Thanks to the establishment of Lemma \ref{lemma: hat p, genuine}, we can apply the same argument as in Step 1 to show that the restriction of $\hat{p}_{\bullet, \cC}$ to $\Seq_\bullet(\Corr(\Fun^\diamond(N(\Fin_*));\cC)_{\inert, \all})$ factors through $\Maps_{(\OneCat)^{\Delta^{op}}_{/\Seq_\clubsuit (N(\Fin_*))}}(\Seq_\clubsuit(N(\Fin_*)\Gray [\bullet]), \Seq_\clubsuit(\bCorr(\cC)^{\otimes, \Fin_*}))$, and this defines $P_{\cC^\times,\bullet}$ (\ref{eq: P_C, bullet}). 

\begin{thm}\label{thm: appendix}
There are  canonically defined functors 
\begin{align}
\label{eq: thm append lax}&\Corr(\Fun^\diamond(N(\Fin_*), \cC^\times))_{\inert, \all}\longrightarrow \CAlg(\bCorr(\cC^\times))^{\rightlax},\\
\label{eq: thm append}&\Corr(\Fun^\diamond(N(\Fin_*), \cC^\times))_{\inert, \act}\longrightarrow \CAlg(\Corr(\cC^\times)).
\end{align}
\end{thm}
For any $c\in \cC$, let $\Seg(c)$ denote for the category of Segal objects $C^\bullet$ in $\cC$ with $C^0\simeq c$. An immediate corollary of an associative version of the above theorem is 
\begin{cor}[Chapter 9, Corollary 4.4.5\cite{GaRo}]
There is a canonically defined functor
\begin{align*}
\Seg(c)\longrightarrow \Alg(\Corr(\cC)).
\end{align*}
\end{cor}
It is natural to expect that the canonical functors (\ref{eq: thm append lax}) and (\ref{eq: thm append}) are  equivalences. Since we will not use this in this paper, we will pursue it in another note. 

Recall the definition $N(\Fin_{*,\dagg})$ in \cite{Jin} as the ordinary $1$-category whose objects are finite sets of the form $\lng n\rng_\dagg=\lng n\rng\sqcup\{\dagg\}$ and whose morphisms $\lng n\rng_\dagg\rightarrow \lng m\rng_\dagg$ are those maps between finite sets that send $*$ to $*$ and $\dagg$ to $\dagg$. There is a natural functor $\pi_\dagg: N(\Fin_{*,\dagg})\rightarrow N(\Fin_*)$, and let  $I_{\pi_\dagg}$ be the coCartesian fibration over $\Delta^1$ by viewing $\pi_\dagg$ as a functor from $\Delta^1$ to $\OneCat^\ord$. 
By the almost same proof for Theorem \ref{thm: appendix}, one gets the following theorem. 
\begin{thm}\label{thm: appendix mod}
There are canonically defined functors 
\begin{align}
\label{eq: thm append lax mod}&\Corr(\Fun^\diamond(I_{\pi_\dagg}, \cC^\times))_{\inert, \all}\longrightarrow \cMod^{N(\Fin_*)}(\bCorr(\cC^\times))^{\rightlax},\\
\label{eq: thm append mod}&\Corr(\Fun^\diamond(I_{\pi_\dagg}, \cC^\times))_{\inert, \act}\longrightarrow \cMod^{N(\Fin_*)}(\Corr(\cC^\times)).
\end{align}
\end{thm}

\section{A list of notations and conventions for categories}\label{sec: notations}
We list a few notations of categories that we use, mostly following the literatures \cite{GaRo} and \cite{higher-algebra}. 
\begin{itemize}
\item A category in boldface means it is an $(\infty,2)$-category, and its underlying $(\infty,1)$-category (by removing all non-invertible 2-morphisms) is denoted by the same notation but not in boldface. 

\item $\OneCat$ (resp. $\TwoCat$): the $(\infty,1)$-category of $(\infty,1)$-categories (resp. $(\infty,2)$-categories); $\bOneCat$: the $(\infty,2)$-category of $(\infty,1)$-categories; $\bOneCat^\ord$: the full subcategory of $\bOneCat$ consisting of ordinary 1-categories, i.e. the morphism spaces are weakly homotopy equivalent to discrete sets; $\bPrstL$ (resp. $\bPrstR$): the $(\infty, 2)$-category of presentable stable $(\infty,1)$-categories with colimit-preserving (resp. limit-preserving) functors.  

\item $\bCorr(\cC)_{vert, horiz}^{adm}$: the $(\infty,2)$-category of correspondences for an $\infty$-category $\cC$ with the vertical (resp. horizontal) arrow in a correspondence (i.e. a 1-morphism) in the class $vert$ (resp. $horiz$) of morphisms in $\cC$, and 2-morphisms in the class $adm$ of morphisms in $\cC$. 

\item $\Slch$: the ordinary 1-category of locally compact Hausdorff spaces; $\fib$ (resp. $\propmap$, $\openmap$): the class of morphisms in $\Slch$ of locally trivial fibrations (resp. proper maps, open embeddings); $\Spc$: the $(\infty,1)$-category of topological spaces (equivalently $\infty$-groupoids). 

\item $\CAlg(\mathbf{C}^\otimes)^{\rightlax}$ (resp. $\CAlg(C^\otimes)$
): the $(\infty,1)$-category of commutative algebra objects in a symmetric monoidal $(\infty,2)$-category $\bC$ (resp. $(\infty,1)$-category $C$), with right-lax (resp. genuine) homomorphisms; $\cMod^{N(\Fin_*)}(\mathbf{C}^\otimes)^{\rightlax}$ (resp. $\cMod^{N(\Fin_*)}(C^\otimes)$): similar as before but with commutative algebras replaced by pairs of a commutative algebra and its module. 

\item $\Seq_\bullet: \TwoCat\rightarrow (\OneCat)^{\Delta^{op}}$: the full embedding of $\TwoCat$ into simplicial $\infty$-categories; $\Delta^n$ and $[n]$ both mean the ordinary 1-category $0\rightarrow 1\cdots\rightarrow n$ (though $\Delta^n$ usually means the nerve of $[n]$, but we do not distinguish them here). 

\item $\Shv(X;\Sp)=\Shv(X;\bS)$: both mean the category of sheaves on $X$ valued in spectra. 

\end{itemize}

\end{document}